\documentclass[10pt]{article}

\usepackage[frenchb,english]{babel}
\usepackage[latin1]{inputenc}
\usepackage[T1]{fontenc}

\usepackage{array,mathrsfs}
\usepackage{amsfonts,amstext,amsmath,amssymb,amsthm,mathrsfs}
\usepackage{graphicx,tikz}
\usetikzlibrary[patterns]
\usepackage{textcomp} 
\usepackage{graphics,xcolor}
\usepackage{stmaryrd}
\usepackage{subfig,placeins}

\usetikzlibrary{decorations.pathmorphing}
\newcommand{\Ressort}[4][]{
\node [minimum size=#2,#1] (ressort) at (#4) {};
\pgfmathparse{#2/#3}\let\pas\pgfmathresult
\draw [decorate,decoration={zigzag,segment length=\pas,amplitude=0.1cm}]
(ressort.east) -- (ressort.west);
}

\newcommand\ds{\displaystyle}
\newcommand\tn{\textnormal}

\newcommand{\norm}[1]{\left\Vert#1\right\Vert}
\newcommand\fcar{{\mathbf 1}}
\newcommand\N{\mathbb{N}}
\newcommand\R{\mathbb{R}}

\newtheorem{theorem}{Theorem}[section]

\newtheorem{corollary}[theorem]{Corollary}
\newtheorem{proposition}[theorem]{Proposition}

\newtheorem{remark}[theorem]{Remark}
\theoremstyle{definition}
\newtheorem{algorithm}{Algorithm}

\numberwithin{equation}{section}

 \title{Convergent algorithm based on Carleman estimates \\
 for the recovery of a potential in the wave equation.
 \footnote{The authors wish to thank Institut Henri Poincaré (Paris, France) and the Research In Paris funding. 
 Partially supported by the Agence Nationale de la Recherche (ANR, France),  Project MEDIMAX number ANR-13-MONU-0012. }}

\author{
Lucie Baudouin$^{1,}$\footnote{e-mail: {\tt baudouin@laas.fr}} \and Maya de Buhan$^{2,}$\footnote{e-mail: {\tt maya.de-buhan@parisdescartes.fr}} \and Sylvain Ervedoza$^{3,}$\footnote{e-mail: {\tt ervedoza@math.univ-toulouse.fr}}
\medskip\\
{\it\footnotesize $^{1}$ LAAS-CNRS, Université de Toulouse, CNRS, Toulouse, France}
\\
{\it\footnotesize $^{2}$ CNRS, UMR 8145, MAP5, Université Paris Descartes, Sorbonne Paris Cité, France
}
\\
{\it\footnotesize $^{3}$ Institut de Math\'ematiques de Toulouse ; UMR 5219 ;  Universit\'e de Toulouse ; CNRS ;}\\
{\it\footnotesize UPS IMT F-31062 Toulouse Cedex 9, France}
}

\date{\today}

\begin{document}
\maketitle
\begin{abstract}
	This article develops the numerical and theoretical study of a reconstruction algorithm of a potential in a wave equation from boundary measurements, using a cost functional built on weighted energy terms coming from a Carleman estimate. More precisely, this inverse problem for the wave equation consists in the determination of an unknown time-independent potential from a single measurement of the Neumann derivative of the solution on a part of the boundary. While its uniqueness and stability properties are already well known and studied, a constructive and globally convergent algorithm based on Carleman estimates for the wave operator was recently proposed in \cite{BaudouinDeBuhanErvedoza}. However, the numerical implementation of this strategy still presents several challenges, that we propose to address here. 
\bigskip
\end{abstract}
 \newpage
%
%

\section{Introduction and algorithms}
%
\subsection{Setting and previous results}
Let $\Omega$ be a smooth bounded domain of $\mathbb{R}^d$, $d\geq 1$ and $T>0$.
This article focuses on the reconstruction of the potential in a wave equation according to the following inverse problem:
\begin{quote}
	Given the source terms $f$ and $f_{\partial}$ and the initial data $(w_0, w_1)$, considering the solution of
	\begin{equation}\label{EqW}
	\left\{ \begin{array}{ll}
 		\partial_t^2 W-\Delta W+Q W=f,   
			&\tn{in } (0,T) \times \Omega,\\
		W =f_{\partial},  
			&\tn{on } (0,T) \times \partial \Omega,\\
		W(0)= w_0, \quad \partial_t W(0)= w_1,  
			&\tn{in } \Omega,
	\end{array}\right.
	\end{equation}
	can we determine the unknown potential $Q = Q(x)$, assumed to depend only on $x\in \Omega$, from the additional knowledge of the flux of the solution through a part  $\Gamma_0$ of the boundary $\partial \Omega$, namely
	\begin{equation}
		\label{Flux}
		\mathscr{M} = \partial_{n} W,  \quad \hbox{ on } (0,T) \times \Gamma_0~?
	\end{equation}
\end{quote}
Beyond the preliminary questions about the uniqueness and stability of this inverse problem, already very well documented as we will detail below, we are interested in the actual reconstruction of the potential $Q$ from the extra information given by the measurement of the flux $\mathscr{M}$ of the solution on a part of the boundary. 
This issue was already addressed theoretically in our previous work \cite{BaudouinDeBuhanErvedoza} based on Carleman estimates. 
However, the algorithm proposed in \cite{BaudouinDeBuhanErvedoza}, proved to be convergent, cannot be implemented in practice as it involves minimization processes of functionals containing too large exponential terms. Therefore, our goal is to address here the numerical challenges induced by that approach. \\
Before going further, let us recall that if $Q\in L^{\infty}(\Omega)$, $f \in L^1(0,T;L^2(\Omega))$, 
$f_{\partial} \in H^1((0,T)\times \partial\Omega)$, $w_0\in H^1(\Omega)$ and $w_1\in L^2(\Omega)$, and assuming the compatibility condition $f_{\partial}(0,x) = w_0(x)$ for all $x\in\partial\Omega$, the Cauchy problem \eqref{EqW} is well-posed in $C^0([0,T]; H^1(\Omega)) \cap C^1([0,T]; L^2(\Omega))$, and the normal derivative $\partial_n W$ is well-defined as an element of $L^2( (0,T) \times \partial \Omega)$, see e.g. \cite{Lions,LasieckaLionsTriggiani}.\\
Our results will require the following geometric conditions (sometimes called ``multiplier condition'' or ``$\Gamma$-condition''):
\definecolor{ffqqqq}{rgb}{1,0,0}
\definecolor{qqqqff}{rgb}{0,0,1}
\begin{eqnarray}
	&&\exists \, x_0 \not \in \overline{\Omega} , \tn{ such that } 
	\nonumber \\
	&&\quad \Gamma_0 \supset \{x \in \partial \Omega, \ (x-x_0) \cdot \vec{n}(x) \geq 0 \}, 	
	\label{GCC-multiplier}\\
	&&\quad T > \sup_{x \in \Omega} | x - x_0|.
	\label{GCC-Time}
\end{eqnarray}

Space and time conditions \eqref{GCC-multiplier}--\eqref{GCC-Time} are natural from the observability point of view, and appear naturally in the context of the multiplier techniques developed in \cite{Ho,Lions}. They are more restrictive than the well-known observability results \cite{Bardos} by Bardos Lebeau Rauch based on the behavior of the rays of geometric optics, but the geometric conditions \eqref{GCC-multiplier}--\eqref{GCC-Time} yield much more robust results, and this will be of primary importance in our approach.
\\

In fact, under the regularity assumption 
\begin{equation}
	\label{RegAssumptions}
	 W \in H^1(0,T; L^\infty(\Omega)),
\end{equation} 
the positivity condition
\begin{equation}
	\label{Positivity}
	\exists \alpha >0 \tn{ such that }|w_0| \geq \alpha \tn{ in } \Omega, 
\end{equation}
the knowledge of an \textit{a priori}  bound $m >0$ such that
\begin{equation}
	\label{A-priori-bound}
	\norm{Q}_{L^\infty(\Omega)} \leq m, \quad i.e. \quad Q \in L^\infty_{\leq m} (\Omega) = \{q \in L^\infty(\Omega), \|q\|_{L^\infty(\Omega)} \leq m \},
\end{equation}
and the multiplier conditions \eqref{GCC-multiplier}--\eqref{GCC-Time}, the results in \cite{Baudouin01} (and in \cite{Yam99} under more regularity hypothesis) state the Lipschitz stability of the inverse problem consisting in the determination of the potential $Q $ in \eqref{EqW} from the measurement of the flux $\mathscr{M}$ in \eqref{Flux}. 
\\

We will introduce our work by describing what was done in our former article \cite{BaudouinDeBuhanErvedoza}, in order to highlight stage by stage the main challenges when performing numerical implementations. 

In \cite{BaudouinDeBuhanErvedoza}, we proposed a prospective algorithm to recover the potential $Q$ from the measurement $\mathscr{M}$ on $(0,T) \times \Gamma_0 $, that we briefly recall below. We assume that conditions \eqref{GCC-multiplier}--\eqref{GCC-Time} are satisfied for some $x_0 \notin \overline\Omega$, and we set $\beta \in (0,1)$ such that
\begin{equation}
	\label{Time-Condition}
	\beta T > \sup_{x \in \Omega} |x - x_0|.
\end{equation}
We then define, for $(t,x) \in (-T,T) \times \Omega $, the Carleman weight functions
\begin{equation}
	\label{poids-double}
		\varphi(t,x) = |x-x_0|^2-\beta t^2, \quad \text{ and  for $\lambda>0$,} \quad
		\psi(t,x) = e^{\lambda (\varphi(t,x)+C_0)},
\end{equation}
where $C_0>0$ is chosen such that $\varphi + C_0 \geq 1$ in $(-T,T) \times \Omega$ and $\lambda>0$ is large enough.
The chore of the algorithm in \cite{BaudouinDeBuhanErvedoza} is the minimization of a functional $K_{s,q}[\mu]$ given for $s >0$, $q \in L^\infty_{\leq m}(\Omega)$ and $\mu \in L^2((0,T) \times \Gamma_0)$ by 
\begin{equation}
	\label{FunctionalK-k}
	K_{s,q}[\mu](z ) 
	= 
	\frac{1}{2} \int_0^T \int_{\Omega} e^{2s \psi} |\partial_{t}^2 z - \Delta z + q z |^2\,dxdt
	+ 
	\frac{s}{2} \int_0^T \int_{\Gamma_0} e^{2s \psi} | \partial_n z - \mu |^2\, d\sigma dt,
\end{equation}	
set on the trajectories $z \in L^2(0,T; H^\fcar_0(\Omega))$ such that  $\partial_{t}^2 z - \Delta z + q z \in L^2((0,T) \times \Omega)$, $\partial_n z \in L^2((0,T) \times \Gamma_0 )$ and $z(0,\cdot) = 0$ in $\Omega$. Note in particular that \cite{BaudouinDeBuhanErvedoza} shows that there exists a unique minimizer of the above functional under the aforementioned assumptions.
The algorithm then reads as follows:
\begin{algorithm}\label{Algo-old}(see \cite{BaudouinDeBuhanErvedoza})\\
	 {\bf \textit{Initialization:}} $q^0 = 0$ (or any guess in $L^\infty_{\leq m} (\Omega)$).
	\\
	 {\bf \textit{Iteration: From $k$ to $k+1$}}
		\\
		$\bullet $ {\it  Step 1 -} Given $q^k$, we set $\mu^k = \partial_t \left(\partial_n w[q^k] - \partial_n W[Q]\right)$ on $(0,T) \times \Gamma_0 $, where $w[q^k]$ denotes the solution of  \eqref{EqW}  with the potential $q^k$ and $\partial_nW[Q]$ is the measurement given in \eqref{Flux}.
		\\	
		$\bullet$ {\it  Step 2 -} Minimize $K_{s, q^k}[\mu^k]$ (defined in \eqref{FunctionalK-k}) on the trajectories $z \in L^2(0,T; H^1_0(\Omega))$ such that  $\partial_{t}^2 z - \Delta z + q^k z \in L^2((0,T) \times \Omega)$, $\partial_n z \in L^2((0,T) \times \Gamma_0 )$ and $z(0,\cdot) = 0$ in~$\Omega$. Let $Z^k$ be the unique minimizer of the functional $K_{s,q^k}[\mu^k]$.
		\\
		$\bullet$ {\it  Step 3 -} Set
		$$
			\tilde q^{k+1} = q^k + \frac{\partial_t Z^k(0)}{w_0},  \quad \tn{ in } \Omega, 
		$$
		where $w_0$ is the initial condition in \eqref{EqW} (recall assumption \eqref{Positivity}).
		\\
		$\bullet$ {\it Step 4 -} Finally, set
		$$
			q^{k+1} = T_m (\tilde q^{k+1}), \quad \textit{ with }~ 
			T_m(q)= \left\{ \begin{array}{ll} q, &\textit{ if } |q| \leq m, \\ \textit{sign}(q) m, &\textit{ if } |q| > m, \end{array}\right. 
		$$
		where $m$ is the a priori bound  in \eqref{A-priori-bound}.
\end{algorithm}
Algorithm \ref{Algo-old} comes along with the following convergence result:
\begin{theorem}[{\cite[Theorem~1.5]{BaudouinDeBuhanErvedoza}}]
	\label{Thm-Old-BdBE}
	Under assumptions \eqref{GCC-multiplier}-\eqref{GCC-Time}-\eqref{RegAssumptions}-\eqref{Positivity}-\eqref{A-priori-bound}-\eqref{Time-Condition}, there exist constants $C>0$, $s_0 >0$ and $\lambda > 0$ such that for all $s \geq s_0$,  Algorithm \ref{Algo-old} is well-defined and the iterates $q^k$ constructed by Algorithm~\ref{Algo-old} satisfy, for all $k \in \N$,
	\begin{equation}
		\label{ConvergenceBdBE13}
			\int_\Omega |q^{k+1} - Q|^2 e^{2 s \psi(0)} \, dx 
			\leq 
			\frac{C \norm{W[Q]}_{H^1(0,T;L^\infty(\Omega))}^2}{s^{1/2}\alpha^2} \int_\Omega |q^k - Q|^2 e^{2 s \psi(0)} \, dx.
	\end{equation}
	In particular, for $s$ large enough, the sequence $q^k$ strongly converges towards $Q$ as $k \to \infty$ in $L^2(\Omega)$.
\end{theorem}
This algorithm presents the advantage of being convergent for any initial guess $q^0 \in L^\infty_{\leq m}(\Omega)$ without any a priori guess except for the knowledge of $m$. This is why we call this algorithm {\it globally convergent}. However, while this algorithm is theoretically satisfactory as at each iteration, it simply consists in the minimization of the strictly convex and coercive quadratic functional $K_{s, q}$, it nevertheless contains several flaws and drawbacks in its numerical implementation. In particular, we underline that the functional $K_{s, q}$ involves two exponentials, namely
$$
	\exp(s \psi) = \exp( s \exp(\lambda (\varphi + C_0))),
$$
with a choice of parameters $s$ and $\lambda$ large enough and whose sizes are difficult to estimate. In particular, for $s = \lambda = 3$ - which are not so large of course - $\Omega = (0,1)$, $x_0 \simeq 0^-$, $T \simeq 1^+$ and $\beta \simeq 1^-$, the ratio
$$
	\frac{\ds \max_{(0,T)\times \Omega} \{\exp(2 s \psi) \}}{\ds \min_{(0,T)\times \Omega} \{\exp(2 s \psi) \}} 
$$
is of the order of $ 10^{340}$ ! The numerical implementation of Algorithm \ref{Algo-old} therefore seems doomed. 
\medskip
\\
The goal of this article is to improve the above algorithm so that it can fruitfully be implemented. This will be achieved following several stages: working on the construction of the cost functional (specifically on the Carleman weight function),  considering the preconditioning of the cost functional, and  adapting the new cost functional  to the discrete setting used for the numerics.
\medskip
\\
Before going further, let us mention that the inverse problem under consideration has been well-studied in the literature, starting with the uniqueness result in the celebrated article \cite{BuKli81}, see also \cite{Klibanov92}, which introduced the use of Carleman estimates  for these studies. Later on, stability issues were obtained for the wave equation, first based on the so-called observability properties of the wave equation \cite{PuelYam96,PuelYam97} and then refined with the use of Carleman estimates, among which \cite{ImYamIP01,ImYamCom01, ImYamIP03,KlibanovYamamoto06}. In fact, a great part of the literature in this area, concerning uniqueness, stability and reconstruction of coefficient inverse problems for evolution partial differential equations can be found in the survey article \cite{KliSurvey-JIIP13} and we refer the interested reader to it. A slightly different approach can also be found in the recent article \cite{SU-TAMS13} based on more geometric insights. 
\\
Let us also emphasize that we are interested in the case in which one performs only one measurement. The question of determining coefficients from the Dirichlet to Neumann map is different and we refer for instance to the boundary control method proposed in \cite{Belishev97} or to methods based on the complex geometric optics, see \cite{Isakov91}.
\\
Here, as we said, we will focus on the reconstruction of the potential in the wave equation \eqref{EqW} from the flux $\mathscr{M}$ in \eqref{Flux}. This question has been studied only recently, though the first investigation \cite{KliIous-SIMA95} appears in 1995, and we shall in particular point out the most recent works of Beilina and Klibanov \cite{BeiKli12}, \cite{BeiKli-NARWA15}, who study the reconstruction of a coefficient in a hyperbolic equation from the use of a Carleman weight function for the design of the cost functional. However, these techniques differ from ours as they work on the functions obtained after a Laplace transform of the equation. 
\\

In what follows, we propose to develop a numerical algorithm in the spirit of the one in \cite{BaudouinDeBuhanErvedoza}, study its convergence and his implementation. Before going further, let us also mention the fact that one can find in \cite{Cindea-FernandezCara-Munch} some numerical experiments based on the minimization of a quadratic functional similar to the one in \eqref{FunctionalK-k}, but with $s$ and $\lambda$ rather small, namely $s = 1$ and $\lambda = 0.1$, see \cite[Section 4]{Cindea-FernandezCara-Munch}. Our goal is to overcome this restriction on the size of the Carleman parameters, as we request them to be large for the convergence of the algorithm. 

\subsection{New weight functions, new cost functionals, and a new algorithm}

In a first stage, we aim at removing one exponential from the cost functional $K_{s,q}$ in \eqref{FunctionalK-k}. 
Similarly to \cite{BaudouinDeBuhanErvedoza}, looking again for a cost functional based on a Carleman estimate for the wave equation, we will work with the Carleman weight function $\exp( s \varphi)$ instead of $\exp (s \exp (\lambda(\varphi + C_0)))$. 
This requires an adaptation of the proof of \cite{BaudouinDeBuhanErvedoza} with such a weight function and the use of the Carleman estimates developed in \cite{LRS} (see also \cite{ImYamCom01}), that we will briefly recall in Section~\ref{SecCV}. 

In particular, instead of minimizing $K_{s, q}[\mu]$ introduced in \eqref{FunctionalK-k} as in Step 2 of Algorithm~\ref{Algo-old}, 
we will perform a minimization process on a new functional $J_{s,q}[\tilde \mu]$, to be defined later in \eqref{FunctionalJ}, 
based on the simplified weight function $\exp( s \varphi)$. Before introducing that functional, we shall define  the following restricted set $\mathcal{O}$:
\begin{multline}
	\label{mathcal-O}
	\mathcal{O} = \{(t,x )\in (0,T) \times \Omega,\, \beta t > |x-x_0|\} \\
	= \{ (t,x ) \in (0,T) \times \Omega,\,  |\partial_t\varphi(t,x)| \geq |\nabla \varphi(t,x)| \}, 
\end{multline}
which is depicted in Figure~\ref{FigDefQ}. 

\begin{figure}[h]
\begin{center}
\begin{tikzpicture}[line cap=round,line join=round,x=4cm,y=3cm]
\draw[->,color=black] (-0.3,0) -- (1.2,0);
\draw[->,color=black] (0,0) -- (0,1.5);
\fill[color=red,pattern=north west lines] (0.,1.32) -- (0.,0.5) -- plot [domain=0:1](\x,{1.1*(\x+0.2)}) -- (0,1.32) -- (1,1.32);
\draw[color=black] (0,0) node[below] {$0$};
\draw[color=black] (1.,0) node[below] {$1$};
\draw[color=red, smooth,samples=100,domain=-0.2:1.0] plot(\x,{1.1*(\x+0.2)});
\draw[color=blue, smooth,samples=100,domain=-0.2:1.0] plot(\x,{0.8*(\x+0.2)});
\draw [dash pattern=on 5pt off 5pt] (1,1.32)-- (0,1.32);
\draw [dash pattern=on 5pt off 5pt] (1,1.3)-- (1,0);
\draw [color=blue](0.1,0.29) node[rotate=30,anchor=north west] {$\varphi = 0$, slope $\frac{1}{\sqrt{\beta}}$};
\draw [color=red](0.25,0.54) node[rotate=39,anchor=north west] {$|x-x_0| = \beta t$, slope $\frac{1}{\beta}$};
\draw[shift={(-0.2,0)},color=black] (0pt,2pt) -- (0pt,-2pt) ;
\draw (0.25,1) node {$\mathcal O$};
\draw (-0.2,0) node[below] {$x_0$};
\draw (1.2,0) node[below] {$x$};
\draw (0,1.32) node[left] {$T$};
\draw (0,1.5) node[left] {$t$};
\draw (0.5,-0.05) node[below] {$\Omega$};
\end{tikzpicture}
		\caption{Illustration of domain $\mathcal O$ in the case $\Omega = (0,1)$.}\label{FigDefQ} 
	\end{center}
\end{figure}
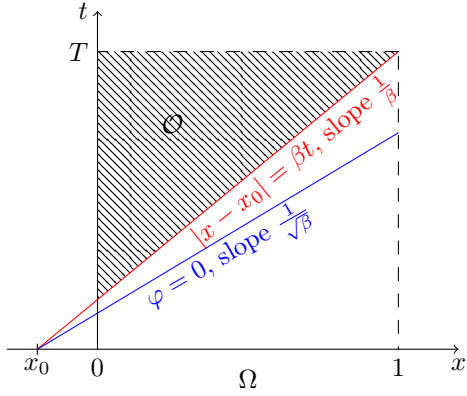 

For $s > 0$, $q \in L^\infty(\Omega)$ and $\tilde \mu \in L^2((0,T)\times \Gamma_0)$, we then introduce the functional $J_{s,q}[\tilde \mu]$ defined by
\begin{multline}
	\label{FunctionalJ}
	J_{s,q}[ \tilde \mu](z ) 
	=  
	\frac{1}{2} \int_0^T \int_{\Omega} e^{2s \varphi} |\partial_{t}^2 z - \Delta z + q z |^2\,dxdt 
	\\
	+ 
	\frac{s}{2} \int_0^T \int_{\Gamma_0} e^{2s \varphi} | \partial_n z -  \tilde \mu |^2 \, d\sigma dt
	+ 
	\frac{s^3}{2} \iint_{\mathcal O} e^{2s\varphi}  |z|^2 \,dxdt , 
\end{multline}
to be compared with $K_{s,q}[\mu]$ in \eqref{FunctionalK-k}, on the trajectories $z\in C^0([0,T];H_0^1(\Omega)) \cap C^1([0, T]; L^2(\Omega))$ 
such that  $\partial_{t}^2 z - \Delta z + q z \in L^2((0,T) \times \Omega )$ and $z(0,\cdot) = 0$ in $\Omega$. 
\\
This functional $J_{s,q}[\tilde \mu]$ is quadratic, and as we will show later in Section \ref{Subsec-ProofConvThm}, under conditions \eqref{GCC-multiplier}--\eqref{GCC-Time}--\eqref{Time-Condition}, it is strictly convex and coercive, therefore enjoying similar properties as the functional $K_{s,q}[\mu]$. Nevertheless, let us once more emphasize that the functional $J_{s,q}[\tilde \mu]$ is less stiff than the functional $K_{s,q}[\mu]$ as now the weights are of the form $\exp(2 s \varphi)$ instead of $\exp(2 s \psi) = \exp(2 s \exp(\lambda (\varphi+C_0)))$ in \eqref{FunctionalK-k}. This already indicates the possible gain we could have by working with the functional $J_{s,q}[\tilde \mu]$ in \eqref{FunctionalJ} instead of $K_{s,q}[\mu]$ in \eqref{FunctionalK-k}.
\\
It may appear surprising to note $\tilde \mu$ instead of $\mu$. These slightly different notations come from the fact that the functional $K_{s,q}[\mu]$ tries to find an optimal solution $Z$ of 
$$
	\partial_{t}^2 Z - \Delta Z + q Z \simeq 0 \hbox{ in } (0,T) \times \Omega, \quad \hbox{ and } \quad \partial_n Z \simeq \mu \hbox{ in } (0,T) \times \Gamma_0,
$$
while the functional $J_{s,q}[\tilde \mu]$ tries to find an optimal solution $\tilde Z$ of 
$$
	\partial_{t}^2 \tilde Z - \Delta \tilde Z + q \tilde Z \simeq 0 \hbox{ in } (0,T) \times \Omega,
	 ~  \quad 
	 \partial_n \tilde Z \simeq \tilde \mu \hbox{ in } (0,T) \times \Gamma_0, 
	~ \hbox{ and } \quad 
	\tilde Z \simeq 0 \hbox{ in } \mathcal{O}.
$$
Therefore, as $\tilde Z$ is sought after  such that it is small in $\mathcal{O}$, it is natural to introduce a smooth cut-off function $\eta \in C^2 (\mathbb{R})$ such that $0 \leq \eta \leq 1$ and 
\begin{equation}
	\label{eta}
	\eta(\tau)  = 0, \text{ if } \tau \leq 0, \quad \hbox{ and } \quad \eta(\tau)  = 1,  \text{ if } \tau \geq d_0^2 := d(x_0, \Omega)^2, 
\end{equation}	
(recall that $d_0^2 >0$ according to Assumption \ref{GCC-multiplier}) see Figure~\ref{isoval}.  Next, the idea is that if 
$$
	\tilde \mu = \eta(\varphi) \mu, \hbox{ in } (0,T) \times \Gamma_0.
$$
and if $Z$ denotes the minimizer of the functional $K_{s,q}[\mu]$ in \eqref{FunctionalK-k}, then the minimizer $\tilde Z$ of $J_{s,q}[\tilde \mu]$ in \eqref{FunctionalJ} should be close to $\eta(\varphi) Z$ in $(0,T) \times \Omega$ and in particular at $t = 0$ this should yield, due to the choice of $\eta$ in \eqref{eta}, $\partial_t \tilde Z(0) \simeq \partial_t Z(0)$ in $\Omega$.
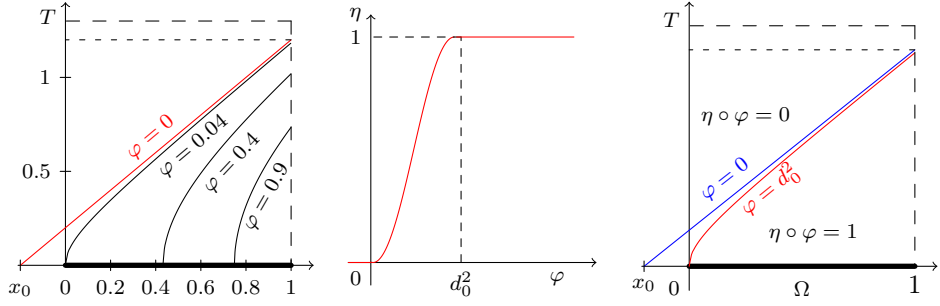
\begin{figure}[ht]
\begin{center}
\begin{tikzpicture}[line cap=round,line join=round,x=3cm,y=2.5cm]
\draw[->,color=black] (-0.21,0) -- (1.1,0);
\foreach \x in {0.2,0.4,0.6,0.8,1}
\draw[shift={(\x,0)},color=black] (0pt,2pt) -- (0pt,-2pt) node[below] {\footnotesize $\x$};
\draw[->,color=black] (0,0) -- (0,1.4);
\foreach \y in {0.5,1}
\draw[shift={(0,\y)},color=black] (2pt,0pt) -- (-2pt,0pt) node[left] {\footnotesize $\y$};
\draw[color=black] (0pt,-2pt) node[below] {\footnotesize $0$};
\draw[-, line width=2pt,color=black] (0,0) -- (1,0);
\draw[color=red, smooth,samples=100,domain=-0.2:1.0] plot(\x,{sqrt(((\x)+0.2)^2)});
\draw[smooth,samples=100,domain=0:1.0] plot(\x,{sqrt((((\x)+0.2)^2-0.04))});
\draw[smooth,samples=100,domain=0.748684000000003:1.0] plot(\x,{sqrt((((\x)+0.2)^2-0.9))});
\draw[smooth,samples=100,domain=0.4325:1.0] plot(\x,{sqrt(((\x)^2+0.4*(\x)-0.36))});
\draw [dash pattern=on 5pt off 5pt] (1,1.3)-- (0,1.3);
\draw [dash pattern=on 2pt off 4pt] (1,1.2)-- (0,1.2);
\draw [dash pattern=on 5pt off 5pt] (1,1.3)-- (1,0);
\draw [color=red](0.2,0.63) node[rotate=40,anchor=north west] {\footnotesize $\varphi = 0$};
\draw (0.34,0.53) node[rotate=40,anchor=north west] {\footnotesize $\varphi = 0.04$};
\draw (0.54,0.43) node[rotate=47,anchor=north west] {\footnotesize $\varphi = 0.4$};
\draw (0.72,0.2) node[rotate=60,anchor=north west] {\footnotesize $\varphi = 0.9$};
\draw[shift={(-0.2,0)},color=black] (0pt,2pt) -- (0pt,-2pt) ;
\draw (-0.2,-0.05) node[below] {\footnotesize $x_0$};
\draw (0,1.32) node[left] {\footnotesize $T$};
\end{tikzpicture}\quad 
\begin{tikzpicture}[line cap=round,line join=round,x=3cm,y=3cm]
\draw[->,color=black] (-0.1,0) -- (1,0);
\draw[->,color=black] (0,-0.1) -- (0,1.1);
\draw[color=black] (0,0) node[below left] {\footnotesize $0$};
\draw[color=black] (0.9,0.) node[below left] {\footnotesize $\varphi$};
\draw[color=black] (0.4,0.) node[below] {\footnotesize $d_0^2$};
\draw[color=black] (0,1) node[left] {\footnotesize $1$};
\draw[-,color=red] (-0.1,0) -- (0,0);
\draw[color=red,samples=1000,domain=0.:0.37] plot(\x,{(sin(500*(\x-1.63))+1)/2});
\draw[dash pattern=on 3pt off 3pt] (0.4,1)-- (0,1);
\draw[dash pattern=on 3pt off 3pt] (0.4,0)-- (0.4,1);
\draw[-,color=red] (0.38,1)-- (0.9,1);
\draw (0,1.1) node[left] {\footnotesize $\eta$};
\end{tikzpicture}\quad
\begin{tikzpicture}[line cap=round,line join=round,x=3cm,y=0.8cm]
\draw[->,color=black] (-0.21,0) -- (1.1,0);
\draw[->,color=black] (0,-0.4) -- (0,4.4);
\draw[color=black] (0pt,-2pt) node[below left] {\footnotesize $0$};
\draw[-, line width=2pt,color=black] (0.,0) -- (1,0);
\draw[color=blue, smooth,samples=100,domain=-0.2:1.0] plot(\x,{sqrt(9*((\x)+0.2)^2)});
\draw[color=red, smooth,samples=100,domain=0.:1.0] plot(\x,{sqrt(9*(((\x)+0.2)^2-0.04))});
\draw [dash pattern=on 5pt off 5pt] (1,4)-- (0,4);
\draw [dash pattern=on 2pt off 4pt] (1,3.6)-- (0,3.6);
\draw [dash pattern=on 5pt off 5pt] (1,4)-- (1,0);
\draw [color=blue](-0.02,1.3) node[rotate=38,anchor=north west] {\footnotesize $\varphi = 0$};
\draw [color=red](0.15,1.1) node[rotate=39,anchor=north west] {\footnotesize $\varphi = d_0^2$};
\draw[shift={(-0.2,0)},color=black] (0pt,2pt) -- (0pt,-2pt) ;
\draw (-0.2,-0.05) node[below] {\footnotesize $x_0$};
\draw (0,4.1) node[left] {\footnotesize $T$};
\draw (0.25,2.5) node {\footnotesize $\eta \circ \varphi =0$};
\draw (0.55,0.55) node {\footnotesize $\eta \circ \varphi =1$};
\draw (1.,0.) node[below] {1};
\draw (0.5,-0.1) node[below] {\footnotesize $\Omega$};
\end{tikzpicture}
\caption{Isovalues of the function $\varphi$ ($x_0=-0.2$, $\beta=1$). Definition and application of the cut-off function $\eta$.}\label{isoval}
\end{center}
\end{figure}
\\

We are then led to propose a revised version of our reconstruction algorithm, detailed in Algorithm \ref{Algo} given below.

\begin{algorithm}
\label{Algo}~~\\
	{\bf \textit{Initialization:}} $q^0 = 0$ (or any guess $q^0 \in L^\infty_{\leq m} (\Omega)$).
	\\
	{\bf \textit{Iteration: From $k$ to $k+1$}}
	\\
	$\bullet $ {\it  Step 1 -}
	Given $q^k$, we set $\tilde \mu^k = \eta(\varphi) \partial_t \left(\partial_n w[q^k] - \partial_n W[Q]\right)$ on $(0,T) \times \Gamma_0 $, where $w[q^k]$ denotes the solution of  
	\begin{equation}\label{Eqwk}
			\left\{ \begin{array}{ll}
 				\partial_t^2 w-\Delta w+q^k w=f,\qquad   & \tn{in }(0,T) \times \Omega,\\
				w =f_{\partial},  & \tn{on } (0, T) \times \partial \Omega,\\
				w(0)= w_0, \quad \partial_t w(0)= w_1,\qquad  &\tn{in } \Omega,
			\end{array}\right.
	\end{equation}
	corresponding to \eqref{EqW}  with the potential $q^k$ and $\partial_nW[Q]$ is the measurement in \eqref{Flux}.
	\\
	$\bullet $ {\it Step 2 -}
	We minimize the functional $J_{s,q^k}[\tilde \mu^k]$ defined in \eqref{FunctionalJ}, for some $s>0$ that will be chosen independently of $k$, 	
	on the trajectories $z\in C^0([-T,T];H_0^1(\Omega)) \cap C^1([-T, T]; L^2(\Omega))$ such that  $\partial_{t}^2 z - \Delta z + q^k z \in L^2((0,T) \times \Omega)$, $\partial_n z \in L^2((0,T) \times \Gamma_0)$ and $z(0, \cdot) = 0$ in $\Omega$. 
	%
	%
	Let $\widetilde Z^k$ be the unique minimizer of the functional $J_{s,q^k}[\tilde \mu^k]$.
	\\
	$\bullet $ {\it  Step 3 -}
	Set
	\begin{equation}
		\label{Def-q-k+1}
		\tilde q^{k+1}  =  q^k + \dfrac{\partial_t \widetilde Z^k(0)}{w_0}, \quad \tn{ in } \Omega,
	\end{equation}
	where $w_0$ is the initial condition in \eqref{Eqwk} (or \eqref{EqW}). 
	\\
	$\bullet $ {\it  Step 4 -}		
	Finally, set
	$$
		q^{k+1} = T_m (\tilde q^{k+1}), \quad \textit{ with }~ 
		T_m(q)= \left\{ \begin{array}{ll} q, &\textit{ if } |q| \leq m, \\ \textit{sign}(q) m, &\textit{ if } |q| \geq m, \end{array}\right. 
	$$
	where $m$ is the a priori bound  in \eqref{A-priori-bound}.
\end{algorithm}

Of course, if one compares Algorithm \ref{Algo} with Algorithm \ref{Algo-old}, the major difference is in Step 2  in which one minimizes the functional $J_{s,q^k}[\tilde\mu]$ in \eqref{FunctionalJ} instead of the functional $K_{s,q^k}[\mu]$ in \eqref{FunctionalK-k}. And as we have said above, the two functionals should have minimizers that are close at $t = 0$.  In fact, similarly as Theorem~\ref{Thm-Old-BdBE}, we will obtain the following result:
\begin{theorem}
	\label{Thm-Algo2Converges}
	Under assumptions \eqref{GCC-multiplier}-\eqref{GCC-Time}-\eqref{RegAssumptions}-\eqref{Positivity}-\eqref{A-priori-bound}-\eqref{Time-Condition}, there exist positive constants   $C$ and $s_0 $ such that for all $s \geq s_0$, Algorithm~$\ref{Algo}$ is well-defined and the iterates $q^k$ constructed by Algorithm~$\ref{Algo}$ satisfy, for all $k \in \N$,
	\begin{equation}
		\label{ConvergenceAlgo2}
			\int_\Omega |q^{k+1} - Q|^2 e^{2 s \varphi(0)} \, dx 
			\leq 
			\frac{C \norm{W[Q]}_{H^1(0,T;L^\infty(\Omega))}^2}{s^{1/2}\alpha^2} \int_\Omega |q^k - Q|^2 e^{2 s \varphi(0)} \, dx.
	\end{equation}
	In particular, for $s$ large enough, the sequence $q^k$ strongly converges towards $Q$ as $k \to \infty$ in $L^2(\Omega)$.
\end{theorem}
The proof of Theorem~\ref{Thm-Algo2Converges} is given in Section~\ref{SecCV} and closely follows the one of Theorem~\ref{Thm-Old-BdBE} in \cite{BaudouinDeBuhanErvedoza}. The main difference is that the starting point of our analysis, instead of being the Carleman estimate in \cite{ImanuvHyp}, is the Carleman estimate in \cite{LRS}.
\\
The main improvement with respect to Algorithm \ref{Algo-old} is the fact that the functional $J_{s,q}[\tilde \mu]$ in \eqref{FunctionalJ} contains weight functions with only one exponential, making the problem less difficult to implement. However, it is still numerically challenging to use such functionals, especially as the convergence of Algorithm \ref{Algo} gets better for large parameter $s$. We propose below two ideas to make it numerically tractable.

\subsection{Preconditioning, processing and discretizing the cost functional}

When considering the functional $J_{s,q}[\tilde \mu]$ in \eqref{FunctionalJ}, one easily sees that exponentials factors can be removed if considering the unknown $z e^{s \varphi}$ instead of $z$. Such transformation corresponds to a preconditioning of the functional $J_{s,q}[\tilde \mu]$. Indeed, that way, exponential factors do not appear anymore when computing the gradient of the cost functional $J_{s,q}[\tilde \mu]$. Nevertheless, there are still exponentials factors appearing in the measurements.
We therefore also develop a progressive algorithm in the resolution of the minimization process. The idea is to consider intervals in which the weight function $\varphi$ does not significantly change, allowing to preserve numerical accuracy despite the possible large values of $s$. Details will be  given in Section~\ref{SecminJ}.
\\

When implementing the above strategy numerically, one has to discretize the wave equation under consideration, and to adapt the functional $J_{s,q}[\tilde \mu]$ to the discrete setting. As it is well-known \cite{Tref,Zua05Survey}, most of the numerical schemes exhibit some pathologies at high-frequency, namely discrete rays propagating at velocity~$0$ or blow up of observability estimates. Therefore, we need to take some care to adapt the functional $J_{s,q}[\tilde \mu]$ to the discrete setting. In particular, following ideas well-developed in the context of the observability of discrete waves (see \cite{Zua05Survey}), we will introduce a naive discrete version of $J_{s,q}[\tilde \mu]$ and penalize the high-frequencies. 
\\
To simplify the presentation of these penalized frequency functionals, we will introduce it in full details on a space semi-discrete and time continuous 1d wave equations, where the space semi-discretization is done using the finite-difference method on a uniform mesh. In this case, our approach, even at the discrete level, can be made completely rigorous by adapting the arguments in the continuous setting and the discrete Carleman estimates obtained in \cite{BaudouinErvedoza11} (recently extended to a multi-dimensional setting in \cite{Baudouin-Ervedoza-Osses}). We refer to Section~\ref{SecImplAlg} for extensive details. 
\\

Section~\ref{SecNum} then presents numerical results illustrating our method on several examples. In particular, we will illustrate the good convergence of the algorithm when the parameter $s$ is large. We shall also discuss the cases in which the measurement is blurred by some noise and the case in which the initial datum $w_0$ is not positive everywhere.\\

\paragraph{Outline} Section \ref{SecCV} is devoted to the proof of the convergence of Algorithm \ref{Algo}. In Section \ref{SecminJ} we explain how the minimization process of the functional $J_{s,q}$ in \eqref{FunctionalJ} can be strongly simplified. Section \ref{SecImplAlg} then makes precise the new difficulties arising when discretizing the functional $J_{s,q}$, and Section \ref{SecNum} presents several numerical experiments.
%
%
%
\section{Study of Algorithm \ref{Algo}}\label{SecCV}
%
%
\subsection{Main ingredients}
The goal of this section is to prove Theorem~\ref{Thm-Algo2Converges}. As mentioned in the introduction, the proof will closely follows the one of Theorem~\ref{Thm-Old-BdBE}  in \cite{BaudouinDeBuhanErvedoza}. The main difference is that, instead of using the Carleman estimate developed in \cite{ImanuvHyp,Baudouin01}, we will base our proof on the following one:
\begin{theorem}\label{Carlemangral}
	Assume the multiplier conditions \eqref{GCC-multiplier}-\eqref{GCC-Time} and $\beta \in (0,1)$ as in \eqref{Time-Condition}. Define the weight function $\varphi$ as in \eqref{poids-double}. Then there exist  $s_{0}>0$ and a positive constant $M $ such that for all $ s \geq s_{0}$:
	\begin{multline}
		\label{Carlemgral}
		s \int_{-T}^{T} \int_{\Omega} e^{2s\varphi} \left(|\partial_t z|^2+|\nabla z|^2 + s^2 |z|^2 \right)\,dxdt
		\leq M\int_{-T}^{T} \int_{\Omega} e^{2s\varphi} |\partial_t^2 z - \Delta z|^2\,dxdt 
		\\
		+ Ms  \int_{-T}^{T} \int_{\Gamma_{0}} e^{2s\varphi} \left|\partial_n z\right|^2\,d\sigma dt 
		+M s^3 \iint_{(|t|, x) \in \mathcal O} e^{2s\varphi}  |z|^2 \,dxdt ,
	\end{multline}
	for all $z\in C^0([-T,T];H_0^1(\Omega)) \cap C^1([-T, T]; L^2(\Omega))$ with $\partial_t^2 z - \Delta z \in L^2((-T,T)\times \Omega)$, where the set $\mathcal{O}$ satisfies \eqref{mathcal-O}.
	\\
	Furthermore, if $z(0,\cdot) = 0$ in $\Omega$, one can add to the left hand-side of \eqref{Carlemgral}, the following term:
	\begin{equation}
		\label{dt-v-0}
		s^{1/2} \int_{\Omega} e^{2s\varphi(0)}|\partial_t z(0)|^2 \,dx.
	\end{equation}
\end{theorem}
The Carleman estimate of Theorem~\ref{Carlemangral} is quite classical and can be found in the literature in several places, among which \cite{LRS,Isakov,ZhangWave,FuYongZhang,Bel-}. For the convenience of the reader, we briefly sketch the proof in Section~\ref{subsecCarl}.
However, the proof of the fact that the term \eqref{dt-v-0} can be added in the left hand side of \eqref{Carlemgral} when $z(0,\cdot) = 0$ in $\Omega$ is not explicitly written in the aforementioned references, although this is one of the important point of the proof of the stability result in \cite{ImYamIP01,ImYamCom01}. Nevertheless, the idea can be adapted easily from \cite{BaudouinDeBuhanErvedoza}, as we will detail below.
\\

Before giving the details of the proof of Theorem~\ref{Thm-Algo2Converges}, let us first briefly explain the main idea of the design of Algorithm \ref{Algo}, which turns out to be very similar to the one of Algorithm \ref{Algo-old}.
Indeed, both Algorithms \ref{Algo-old} and \ref{Algo} are constructed from the fact that if $W[Q]$ is the solution of equation \eqref{EqW} and $w[q^k]$ solves \eqref{Eqwk}, then 
\begin{equation}
	\label{z-k-exact}
	z^k = \partial_t \left( w[q^k] - W[Q]\right)
\end{equation}
satisfies
\begin{equation}
	\label{EqExacteWk}
	\left\{
		\begin{array}{ll}
					\partial_{t}^2 z^k - \Delta z^k + q^k z^k = g^k, & \tn{in } (0,T) \times \Omega, 
				\\
					z^k = 0, & \tn{on } (0,T) \times \partial \Omega, 
				\\
					z^k(0) = 0, \quad \partial_t z^k(0) = z_1^k, \qquad& \tn{in } \Omega,
		\end{array}
	\right.	
\end{equation}
where
$g^k =   (Q- q^k) \partial_t W[Q]$,  $z_1^k =  (Q- q^k) w_0,$
and we have
$
		 \mu^k = \partial_n z^k \tn{ on } (0,T) \times \Gamma_0.
$

In system \eqref{EqExacteWk}, the source $g^k$ and the initial data  $z_1^k$ are both unknown, and we are actually interested in finding a good approximation of $z_1^k$, which encodes the information on $Q - q^k$. In order to do so, we will try to fit ``at best'' the flux $\partial_n z$ with $\mu^k$ on the boundary, approximating the unknown source term $g^k$ by $0$.

This strategy works as we can prove that the source term $g^k$ brings less information than $\mu^k$ does, and this is where the choice of the Carleman parameter $s$ will play a crucial role. This is actually the milestone of the 
construction of Algorithm~\ref{Algo-old} and its convergence result \cite{BaudouinDeBuhanErvedoza}. Here, when considering the functional $J_{s, q}[\eta(\varphi) \mu]$ defined in \eqref{FunctionalJ}, we rather try to approximate $\tilde z^k = \eta(\varphi) z^k$, which enjoys the following properties:
\begin{itemize}
	\item $\partial_t \tilde z^k(0,\cdot) = \eta(\varphi(0)) \partial_t z^k (0,\cdot) = (Q - q^k) w_0$ encodes the information on $Q- q^k$; 
	\item $\tilde z^k = \eta(\varphi) z^k$ vanishes in domain $\mathcal{O}$ defined by \eqref{mathcal-O} and on the boundary in time $t = T$;
	\item $\partial_n \tilde z^k = \tilde \mu^k$ in $(0,T) \times \Gamma_0$.
\end{itemize}
These ideas are actually behind the proofs of the inverse problem stability by compactness uniqueness arguments as in \cite{PuelYam96,PuelYam97,Yam99} or by Carleman estimates given in \cite{ImYamIP01,ImYamCom01,ImYamIP03,Baudouin01}.
%
\subsection{Sketch of the proof of the Carleman estimate}\label{subsecCarl}

Since a lot of different references, several of them mentioned right above, present detailed proof of Carleman estimates for the wave equation, we only give here the main calculations yielding the result presented in Theorem~\ref{Carlemangral}.

\begin{proof}
	Set 
	$
		y (t,x)=z(t,x)e^{s\varphi(t,x)} 
	$
	for all $ (t,x)\in  (-T,T)\times \Omega$, and introduce the conjugate operator $\mathscr{L}_{s} $ defined by  
	$\mathscr{L}_{s} y = e^{s\varphi}(\partial_t^2 - \Delta) (e^{-s\varphi} y)$. Easy computations give
	\begin{multline}
		\label{Def-L-s}
 		\mathscr{L}_{s} y = \underbrace{\partial_t^2 y -\Delta y +s^2(|\partial_t \varphi|^2-|\nabla\varphi|^2)y}_{ = P_1 y} 
		\, \underbrace{ -~2s \partial_t y\partial_t \varphi + 2 s \nabla y \cdot \nabla\varphi + \alpha s y }_{ = P_2 y}  
		\\
		\, \underbrace{-~ s (\partial_t^2 \varphi - \Delta \varphi ) y - \alpha s y }_{= Ry}
	\end{multline}
	where we have set $\alpha = 2 d - 2$, $d$ being the space dimension. Based on the estimate
	\begin{eqnarray}
	2\int_{-T}^{T} \int_{\Omega}P_{1}yP_{2}y\,dxdt 
	&\leq& 
	 \int_{-T}^{T} \int_{\Omega}\left( |P_{1}y| ^2 + | P_{2}y| ^2\right)\,dxdt
	+~2\int_{-T}^{T} \int_{\Omega}P_{1}yP_{2}y\,dxdt 
	\nonumber\\
	&\leq& 2 \int_{-T}^{T} \int_{\Omega} |\mathscr{L}_s y|^2 \,dxdt + 2 \int_{-T}^{T} \int_{\Omega} |Ry|^2\,dxdt,
		\label{PP}
	\end{eqnarray}
	the main part of the proof consists in the computation and bound from below of the cross-term 
	$$
		I = \int_{-T}^{T} \int_{\Omega}P_{1}y\, P_{2}y\,dxdt.
	$$ 
	Tedious computations and integrations by parts yield
	\begin{eqnarray*}
	I & =&  s \int_{-T}^{T} \int_{\Omega} |\partial_t  y|^2 (  \partial_t^2 \varphi +  \Delta \varphi-\alpha) \,dxdt 
	+ s \int_{-T}^{T} \int_{\Omega} |\nabla y|^2 \, (\partial_t^2 \varphi -\Delta \varphi + \alpha + 4  )  \,dxdt  
	   \\
	 &&+  ~ s^3 \int_{-T}^{T} \int_{\Omega}|y|^2 \left[ \partial_t \left( \partial_t \varphi (|\partial_t \varphi|^2-|\nabla\varphi|^2)   \right)  
	  	+ \alpha (|\partial_t \varphi|^2-|\nabla\varphi|^2) \right.
		\\
	&&\hspace{7cm}	\left. - \nabla \cdot \left( \nabla\varphi(|\partial_t \varphi|^2-|\nabla\varphi|^2)  \right) \right]  \,dxdt
	 \\
	&&-  ~ s \left[ \int_{\Omega} \left( |\partial_t  y|^2 + |\nabla y|^2 \right)\partial_t \varphi  \,dx\right]_{-T}^T + 2 s \left[ \int_{\Omega}\partial_t y \,( \nabla y \cdot \nabla\varphi) \,dx\right]_{-T}^T 
	\\
	&&-  ~s^3 \left[  \int_{\Omega} y^2 (|\partial_t \varphi|^2-|\nabla\varphi|^2)  \partial_t \varphi \,  \,dx\right]_{-T}^{T}
	+  \alpha  s \left[ \int_{\Omega}\partial_t y \, y \,dx\right]_{-T}^T 
	\\
	&&- ~ s \int_{-T}^{T} \int_{\partial \Omega} |\partial_n  y |^2 \partial_n \varphi \,d\sigma dt. 
	\end{eqnarray*}
	Let us now briefly explain how each term can be estimated.
	\\
	
	$\bullet$ We focus on the terms in $s| \partial_t y|^2 $ and $s|\nabla y|^2$ in order to insure that they are strictly positive. Taking $\alpha = 2 d - 2$, this means
\begin{eqnarray*}
		\partial_t^2 \varphi +  \Delta \varphi-\alpha = -2\beta + 2d - \alpha = 2 (1 - \beta)\quad \tn{ and } \\
		\partial_t^2 \varphi -\Delta \varphi + \alpha + 4 = -2\beta -2d + \alpha +4  = 2 (1 - \beta),
\end{eqnarray*}
	that are positive thanks  to the assumption $\beta\in(0,1)$.
	\\
	
	$\bullet$ The terms in $s^3|y|^2$ can be rewritten as follows (since $\nabla^2 \varphi = 2 \rm{Id}$):
	\begin{align*}
	 &\partial_t \left( \partial_t \varphi (|\partial_t \varphi|^2-|\nabla\varphi|^2)   \right)  + \alpha (|\partial_t \varphi|^2-|\nabla\varphi|^2)  - \nabla \cdot \left( \nabla\varphi(|\partial_t \varphi|^2-|\nabla\varphi|^2)  \right) \\
	 & =  (\partial_t^2 \varphi - \Delta \varphi + \alpha )  (|\partial_t \varphi|^2-|\nabla\varphi|^2) + 2 |\partial_t \varphi|^2 \partial_t^2 \varphi + 2 \nabla^2 \varphi \cdot  \nabla \varphi \cdot \nabla \varphi\\
	 & =  (-6 \beta - 2d + \alpha )  (|\partial_t \varphi|^2-|\nabla\varphi|^2) + 4 (1 - \beta)  |\nabla \varphi|^2 \\
	 & = - (2+6 \beta )  (|\partial_t \varphi|^2-|\nabla\varphi|^2) + 4 (1 - \beta)  |\nabla \varphi|^2.
	 \end{align*}
	This quantity is bounded from below by a strictly positive constant in the region of $(-T, T) \times \Omega$ in which
	$$
		|\partial_t \varphi(t,x)|^2-|\nabla\varphi(t,x)|^2 \leq 0 \Longleftrightarrow \beta t \leq |x-x_0|, 
	$$
	i.e. the complementary of the set $\big\{ (t,x) \in (-T, T) \times \Omega \hbox{ with } (|t|, x) \in \mathcal{O}\big\}$ where $\mathcal{O}$ satisfies \eqref{mathcal-O}.
	\\
	
	$\bullet$ We now estimate the boundary terms in time\footnote{The authors acknowledge Xiaoyu Fu for having pointed out to us the fact that these boundary terms have positive signs.} appearing at time $t=T$ and $t = - T$. We focus on the terms at time $T$, as the ones at time $-T$ can be handled similarly. Let us first collect them:
	\begin{multline*}
		I_T := 2 s \beta T \int_\Omega \left( |\partial_t  y(T)|^2 + |\nabla y(T)|^2 \right) \, dx
		+ 
		8 s^3 \beta T \int_{\Omega} |y(T)|^2 (\beta^2 T^2-|x-x_0|^2) \,  \,dx
		\\
		+ 
		4s \int_\Omega \partial_t y(T) \left(\nabla y(T)\cdot (x-x_0) + \frac{\alpha}{4} y(T) \right) \, dx.
	\end{multline*}	
	The first and second terms are obviously positive (under Condition \eqref{Time-Condition} for the second one), so we only need to check that they are sufficiently positive to absorb the last term, whose sign is unknown. We remark that
	\begin{eqnarray*}
		\lefteqn{
		\int_\Omega \left| \nabla y(T)\cdot (x-x_0) + \frac{\alpha}{4} y(T)\right|^2\, dx
		}\\
		& = &
		\int_\Omega \left| \nabla y(T)\cdot (x-x_0)\right|^2\, dx
		+
		\frac{\alpha}{4} \int_\Omega (x-x_0) \cdot \nabla \left(|y(T)|^2\right) \, dx
		+ 
		\frac{\alpha^2}{16} \int_\Omega |y(T)|^2\, dx
		\\
		& = &
		\int_\Omega \left| \nabla y(T)\cdot (x-x_0)\right|^2\, dx
		+ \left(\frac{\alpha^2}{16}- \frac{\alpha d}{4}\right) \int_\Omega |y(T)|^2\, dx
		\\
		&\leq&	
		\sup_\Omega \left\{ |x -x_0|^2\right\} \int_\Omega \left| \nabla y(T)\right|^2\, dx , 
	\end{eqnarray*}
	since $\alpha = 2d-2$ gives $\alpha^2 - 4 \alpha d = - 4 (d - 1) (d +1) \leq 0$. This inequality allows to deduce, by Cauchy-Schwarz inequality, that 
	\begin{multline*}
		4s \int_\Omega \partial_t y(T) \left(\nabla y(T)\cdot (x-x_0) + \frac{\alpha}{4} y(T) \right) \, dx
		\\
		 \leq 
		2s \sup_\Omega \left\{ |x -x_0|\right\} \left( \int_\Omega \left( |\partial_t  y(T)|^2 + |\nabla y(T)|^2 \right) \, dx\right).
	\end{multline*}
	Using again Condition \eqref{Time-Condition}, we easily obtain $I_T \geq 0$.
	\\
	\par
	Gathering these informations, and using the geometric condition \eqref{GCC-multiplier} on $\Gamma_0$, it yields that there exists a constant $M>0$ independent of $s$ such that
	\begin{multline*}
	\int_{-T}^{T} \int_{\Omega} P_{1} yP_{2}y\,dxdt  
	\geq 
	M s \int_{-T}^{T} \int_{\Omega} \left( |\partial_t y|^2 + |\nabla y |^2 + s^2 |y|^2\right) \,dxdt 
	\\
	- M s \int_{-T}^{T} \int_{\Gamma_0} \left| \partial_n y\right|^2 \,d\sigma dt 
	- M s^3 \iint_{ (|t|, x) \in \mathcal{O}} |y|^2 \,dxdt.
	\end{multline*}
	From \eqref{PP}, we easily derive
	\begin{multline*}
		s \int_{-T}^{T} \int_{\Omega}\left(|\partial_t y| ^2+|\nabla y|^2 + s^2 |y|^2 \right)\,dxdt 
		+\int_{-T}^{T} \int_{\Omega}\left( |P_{1}y|^2 + | P_{2}y|^2\right)\,dxdt 
		\\
		\leq M\int_{-T}^{T} \int_{\Omega} |\mathscr{L}_{s}y|^2\,dxdt  
		+M s^2 \int_{-T}^{T} \int_{\Omega}  |y|^2\,dxdt
		\\
		+~M s \int_{-T}^{T} \int_{\Gamma_0} \left| \partial_n y\right|^2 \,d\sigma dt + M s^3 \iint_{ (|t|, x) \in \mathcal{O}} |y|^2 \,dxdt.
	\end{multline*}
	We take now $s_0$ large enough in order to make sure that the term in $s^2|y|^2$ of the right hand side is absorbed by the dominant term in $s^3 |y|^2$ of the left  hand side as soon as $s\geq s_0$ and we obtain
	\begin{multline}
		s \int_{-T}^{T} \int_{\Omega}\left(| \partial_t y|^2+|\nabla y|^2 + s^2| y|^2 \right)\,dxdt +\int_{-T}^{T} \int_{\Omega}\left( |P_{1}y|^2 + | P_{2}y| ^2\right)\,dxdt
		 \\
		\leq M\int_{-T}^{T} \int_{\Omega} |\mathscr{L}_{s} y|^2\,dxdt  	+M s \int_{-T}^{T} \int_{\Gamma_0} \left| \partial_n y\right|^2 \,d\sigma dt + M s^3 \iint_{ (|t|, x) \in 		\mathcal{O}} |y|^2 \,dxdt 
		\label{CarlemW}
	\end{multline}
	We then deduce \eqref{Carlemgral} by substituting  $y = z e^{s \varphi}$.
	\\

	Furthermore, under the additional condition $z(0,\cdot) = 0$ in $\Omega$, we get $y(0,\cdot) = 0$ in~$\Omega$. 
	We then choose $\rho : t\mapsto  \rho(t)$ a smooth function such that $\rho(0) = 1$ and $\rho$ vanishes close to $t =  -T$. We multiply $P_1 y$ by $\rho \partial_t y$ and integrate over $(-T,0)\times \Omega$, to get
	\begin{eqnarray*}
	\lefteqn{
	\int_{-T}^0\int_{\Omega} P_1 y \,\rho \partial_ty\,dxdt 
		 =  \int_{-T}^0\int_{\Omega} \left(\partial_t^2y-\Delta y+s^2((\partial_t\varphi)^2-|\nabla\varphi|^2)y\right)\,\rho \partial_t y\,dxdt 
		 }
	\\
	& = & \dfrac 12 \int_{-T}^0\int_{\Omega} \rho \partial_t\left( |\partial_ty|^2 + |\nabla y|^2 \right)\; dxdt 
	 +~\dfrac  {s^2}2  \int_{-T}^0\int_{\Omega} \rho (|\partial_t\varphi|^2-|\nabla\varphi|^2) \partial_t (y^2)\,dxdt 
	 \\
	& = & \dfrac 12\int_{\Omega} |\partial_ty(0)|^2 \,dx - \dfrac 12 \int_{-T}^0\int_{\Omega} \partial_t\rho \left( |\partial_ty|^2 + |\nabla y|^2 \right)
	+ s^2\partial_t \left( \rho ( |\partial_t\varphi|^2-|\nabla\varphi|^2)\right) y^2 dxdt 
	\\
	& \geq &\dfrac 12\int_{\Omega} |\partial_ty(0)|^2  \,dx  - M  \int_{-T}^0\int_{\Omega}  \left( |\partial_t y|^2 + |\nabla y|^2 +s^2 |y|^2\right) \,dxdt.
	\end{eqnarray*}
	By Cauchy-Schwarz inequality, this implies
	$$
		s^{1/2} \int_{\Omega} |\partial_ty(0)|^2  \,dx 
		\leq	\int_{-T}^T\int_{\Omega} |P_1y|^2\,dxdt 
		+M s \int_{-T}^T\int_{\Omega}\left( |\partial_ty|^2 + |\nabla y|^2 + s^2 |y|^2\right)\,dxdt.
	$$
	Using \eqref{CarlemW} and $y = z e^{s \varphi}$, we easily deduce the estimate of term \eqref{dt-v-0} and conclude the proof of Theorem~\ref{Carlemangral}.
\end{proof}
From this proof of Theorem~\ref{Carlemangral}, we can directly  exhibit (see \eqref{CarlemW}) the following ``conjugate'' Carleman estimate, of practical interest later on:
\begin{corollary}\label{Cor-Carleman-Conjugate}
	Assume the multiplier condition \eqref{GCC-multiplier}-\eqref{GCC-Time} and $\beta \in (0,1)$ as in \eqref{Time-Condition}. Define the weight function $\varphi$ as in \eqref{poids-double}. Then there exist constants $M>0$ and $s_0 >0$ such that for all $s \geq s_0$, 
	\begin{multline}
		s \int_{-T}^{T} \int_{\Omega}\left(| \partial_t y|^2+|\nabla y|^2 + s^2| y|^2 \right)\,dxdt +\int_{-T}^{T} \int_{\Omega}\left( |P_{1}y|^2 + | P_{2}y| ^2\right)\,dxdt
		\\
		\leq M\int_{-T}^{T} \int_{\Omega} |\mathscr{L}_{s} y|^2\,dxdt  	+M s \int_{-T}^{T} \int_{\Gamma_0} \left| \partial_n y\right|^2 \,d\sigma dt + M s^3 \iint_{ (|t|, x) \in 		\mathcal{O}} |y|^2 \,dxdt 
		\label{CarlemY}
	\end{multline}
	for all $y \in C^0([-T,T]; H^\fcar_0(\Omega)) \cap C^1([-T,T]; L^2(\Omega))$, with $\mathscr{L}_s y \in L^2((-T,T) \times \Omega)$, where $\mathscr{L}_s$, $P_1$ and $P_2$ are defined in \eqref{Def-L-s}. 
	\\
	Furthermore, if $y(0, \cdot) = 0$ in $\Omega$, one can add the term $~s^{1/2} \ds\int_{\Omega} |\partial_ty(0)|^2  \,dx ~$
	to the left hand-side of \eqref{CarlemY}.
\end{corollary}

%
\subsection{Proof of the convergence theorem}\label{Subsec-ProofConvThm}
%
\begin{proof}[Proof of Theorem~\ref{Thm-Algo2Converges}]	
	Let us first begin by showing that Algorithm \ref{Algo} is well-defined. We introduce
	\begin{multline*}
		\mathcal{T}_{q} = 
			\Big\{
				z \in C^0([0,T]; H_0^1(\Omega)) \cap C^1([0,T]; L^2(\Omega)),  
				\\
				\hbox{ with }  \partial_{t}^2 z - \Delta z  +q z\in L^2((0,T) \times \Omega ) \hbox{ and }
				z(0,\cdot) = 0 \hbox{ in } \Omega
			\Big\}, 	
	\end{multline*}
	endowed with the norm
	\begin{multline*}
		\| z\|^2_{\textnormal{obs},s,q} = 
			 \int_{0}^{T}\int_\Omega e^{2s\varphi} |\partial_{t}^2 z - \Delta z +q z|^2 \,dxdt
			+s \int_{0}^{T}\int_{\Gamma_0} e^{2s\varphi}|\partial_n z|^2 \,d\sigma dt
			\\
			+ s^3 \iint_{\mathcal O} e^{2s\varphi}  |z|^2 \,dxdt .
	\end{multline*}
	The proof that this quantity is a norm on $\mathcal{T}_{q}$ stems from the Carleman estimate of Theorem~\ref{Carlemangral} applied to $z_e(t,x) = z(t,x)$ for $t \in [0,T]$ and $z_e(t,x) = - z(-t,x)$ for $t \in [-T,0]$, $x \in \Omega$. Indeed, \eqref{Carlemgral} applied to $z_e$ yields for all $s \geq s_0$,
	\begin{eqnarray*}
		s^3 \int_0^T \int_\Omega e^{2 s \varphi} |z|^2 \, dxdt 
		&\leq& 
		2 M\int_{0}^{T} \int_{\Omega} e^{2s\varphi} |\partial_t^2 z - \Delta z + q z|^2\,dxdt 
		\\
		&&+~ 
		2M \norm{q}_{L^\infty(\Omega)}^2 \int_0^T \int_\Omega e^{2 s \varphi} |z|^2 \, dxdt 
		\\
		&&+~ 
		Ms  \int_{0}^{T} \int_{\Gamma_{0}} e^{2s\varphi} \left|\partial_n z\right|^2\,d\sigma dt
		+
		M s^3 \iint_{\mathcal O} e^{2s\varphi}  |z|^2 \,dxdt 
		, 
	\end{eqnarray*}
	so that $\|\cdot \|_{\textnormal{obs},s,q}$ is a norm on  $\mathcal{T}_{q}$ provided $s$ is large enough, and then for all $s>0$ as the weight functions are bounded on $[0,T] \times \overline{\Omega}$. This immediately implies that $J_{s,q}[\tilde \mu]$ defined in \eqref{FunctionalJ} is coercive and strictly convex on the set $\mathcal{T}_{q}$,  so that it admits a unique minimizer and as a consequence, Algorithm \ref{Algo} is well-defined. 
	
	Moreover, this shows that the class $\mathcal{T}_{q}$, which was {\it a priori} dependent of $q$, is in fact independent of $q$ (for $q \in L^\infty(\Omega)$) and is simply given by 
\begin{multline*}
		\mathcal{T} = 
			\Big\{
				z \in C^0([0,T]; H_0^1(\Omega)) \cap C^1([0,T]; L^2(\Omega)),  
				\\
				\hbox{ with }  \partial_{t}^2 z - \Delta z  \in L^2((0,T) \times \Omega ) \hbox{ and }
				z(0,\cdot) = 0 \hbox{ in } \Omega
		\Big\}. 	
\end{multline*}
	In order to show estimate \eqref{ConvergenceAlgo2}, instead of considering only functionals of the form $J_{s,q}[\tilde \mu]$, we introduce slightly more general functionals $J_{s, q}[\tilde \mu, g]$ given for $s >0$, $q \in L^\infty(\Omega)$, $\tilde \mu \in L^2((0,T) \times \Gamma_0)$, $g \in L^2((0,T) \times \Omega)$ and for all  $z \in \mathcal{T}$, by:
	\begin{multline}
		\label{FunctionalJ-mu-g}
			J_{s,q}[ \tilde \mu,g](z ) 
				=  
			\frac{1}{2} \int_0^T \int_{\Omega} e^{2s \varphi} |\partial_{t}^2 z - \Delta z + q z -g |^2\,dxdt 
			\\
			+ 
			\frac{s}{2} \int_0^T \int_{\Gamma_0} e^{2s \varphi} | \partial_n z -   \tilde \mu |^2 \, d\sigma dt
			+ 
			\frac{s^3}{2} \iint_{\mathcal O} e^{2s\varphi}  |z|^2 \,dxdt.
	\end{multline}
	With the same argument as above, the functional $J_{s,q}[\tilde \mu, g]$ is coercive in the norm $\norm{\cdot}_{\textnormal{obs},s,q}$ and strictly convex, so that it admits a unique minimizer for each $\tilde \mu \in L^2((0,T) \times \Gamma_0 )$ and $g \in L^2((0,T) \times \Omega)$.

	We then observe that  $\tilde z^k := \eta(\varphi) z^k$, where $z^k$ satisfies \eqref{z-k-exact} (recall the definitions of $\eta$ in \eqref{eta} and $\varphi$ in \eqref{poids-double}, pictured in Figure~\ref{isoval}), is the minimizer of $J_{s, q^k}[\tilde \mu^k, \tilde g^k]$ with 
	\begin{equation}
	\label{gk}
		\tilde g^k = \eta(\varphi) (Q - q^k) \partial_t W[Q] + [\partial_t^2 - \Delta, \eta(\varphi)]  z^k,
	\end{equation}
	since it solves:
	\begin{equation}
	\label{Eqtildezk }
	\left\{
		\begin{array}{ll}
					\partial_{t}^2 \tilde z^k - \Delta \tilde z^k + q^k \tilde  z^k = \tilde g^k, & \tn{in } (0,T) \times \Omega, 
				\\
					\tilde z^k = 0, & \tn{on } (0,T) \times \partial \Omega, 
				\\
					\tilde z^k(0) = 0, \quad \partial_t \tilde z^k(0) = \eta(\varphi(0, \cdot)) z_1^k, \qquad& \tn{in } \Omega,
		\end{array}
	\right.	
	\end{equation}
	and 
	$
		\partial_n \tilde z^k = \tilde \mu^k = \eta(\varphi) \partial_t \left(\partial_n w[q^k] - \partial_n W[Q]\right) \tn{ on } (0,T) \times \Gamma_0 .
	$	
	
	We shall then compare $\widetilde Z^k$ and $\tilde z^k$, the minimizers of the functionals $J_{s, q^k}[\tilde \mu^k, 0]$ and $J_{s, q^k}[\tilde \mu^k, \tilde g^k]$ respectively, especially at the time $t = 0$ corresponding to the set in which the information on $(Q-q^k)$ is encoded. The  result is stated as follows:
	\begin{proposition}
		\label{PropDependenceG}
		Assume the geometric and time conditions \eqref{GCC-multiplier}-\eqref{GCC-Time} on $\Gamma_0$ and $T$, that $\beta$ is chosen as in \eqref{Time-Condition},
		and let $\mu \in L^2((0,T) \times \Gamma_0)$ and $g^a,\,  g^b \in L^2((0,T)\times \Omega )$. 
		Assume also that $q$ belongs to  $L^\infty_{\leq m} (\Omega)$ for $m>0$. 
		\\
		Let $Z^j$ be the unique minimizer of the functional $J_{s,q}[\mu, g^j]$ on $\mathcal{T}$ for $j \in \{a, b\}$. Then there exist positive constants $s_0(m)$ and $M = M(m)$ such that for $s \geq s_0(m)$ we have:
		\begin{equation}
			\label{EstMin-s}
			s^{1/2}\int_{\Omega} e^{2 s \varphi(0)} |\partial_t Z^a(0) -\partial_t Z^b(0)|^2  \,dx \leq M \int_0^T \int_{\Omega} e^{2 s\varphi} |g^a - g^b|^2 \,dxdt. 
		\end{equation}
		where $\varphi$ and $s_0(m)$ are chosen so that Theorem~$\ref{Carlemangral}$ holds.
	\end{proposition}
	We postpone the proof of Proposition \ref{PropDependenceG} to the end of the section and first show how it can be used for the proof of Theorem~\ref{Thm-Algo2Converges}. \\
	Recall now that $\partial_t \tilde z^k(0,\cdot) = (Q - q^k) w^0$. Setting $\tilde q^{k+1}$ as in \eqref{Def-q-k+1}, we get from Proposition \ref{PropDependenceG} applied to $Z^a = \widetilde Z^k$ and $Z^b = \tilde z^k$ that 
	\begin{equation}
		s^{1/2} \int_\Omega e^{2 s \varphi(0)} | \tilde q^{k+1}- Q |^2 \, |w^0|^2 dx \leq M \int_0^T \int_\Omega e^{2s \varphi} |\tilde g^k|^2 \, dx dt.
	\end{equation}
	The next step is to get an estimates of $\tilde g^k$ defined by \eqref{gk}. Using the fact that $\left[ \eta(\varphi), \partial_t^2 - \Delta \right]z^k$ has support in a region where $ \varphi \leq d_0^2 := d(x_0, \Omega)^2$, we obtain
	\begin{eqnarray*}
		 \int_0^T \int_{\Omega} e^{2 s\varphi} |\tilde g^k|^2 \,dxdt 
		&\leq&  M \int_0^T \int_{\Omega} e^{2 s\varphi}| \eta(\varphi) (Q - q^k) \partial_t W[Q]|^2 \, dxdt \\
		&&\hfill{}+ ~M \int_0^T \int_{\Omega} e^{2 s\varphi} |\left[ \eta(\varphi), \partial_t^2 - \Delta \right]z^k|^2 \,dxdt 	\\
		&\leq& M \norm{W[Q]}_{H^1(0,T; L^\infty(\Omega))}^2 \int_{\Omega} e^{2 s\varphi(0)}|q^k-Q|^2 \, dx \\
		&&\hfill{}+~M e^{2 s d_0^2 } \int_0^T \int_{\Omega} \left( |\nabla z^k|^2 +  |\partial_t z^k|^2 + |z^k|^2\right)  \,dxdt. 
	\end{eqnarray*}
	Usual \textit{a priori} energy estimates for $z^k$ solution of equation \eqref{EqExacteWk} also yields
	\begin{multline}
		\label{aprioriestimate}
		\| z^k \|_{L^\infty (0,T;H_0^1(\Omega))} + \| \partial_t z^k \|_{L^\infty (0,T;L^2(\Omega))} 
		\leq M \left( \| z_1^k\|_{L^2(\Omega)}  + \| g^k \|_{L^1(0,T;L^2(\Omega))}  \right)
		\\
		\begin{array}{l}
		\ds \leq M  \| Q-q^k\|_{L^2(\Omega)} \left(\| w_0 \|_{L^\infty(\Omega)}  + \| \partial_t W[Q] \|_{L^1(0,T;L^\infty(\Omega))}  \right)
		\smallskip\\
		 \ds \leq  M\norm{W[Q]}_{H^1(0,T; L^\infty(\Omega))}  \| Q-q^k\|_{L^2(\Omega)},
		\end{array}
	\end{multline}
	so that combining the above estimates, we get
	\begin{multline*}
		s^{1/2} \int_\Omega e^{2 s \varphi(0)} |\tilde q^{k+1}- Q |^2 \, |w^0|^2 dx
		\leq M \norm{W[Q]}_{H^1(0,T; L^\infty(\Omega))}^2 \int_{\Omega} e^{2 s\varphi(0)}|q^k-Q|^2 \, dx  
		\\
		+ M \norm{W[Q]}_{H^1(0,T; L^\infty(\Omega))}^2 e^{2 s d_0^2} \| Q-q^k\|_{L^2(\Omega)}^2.
	\end{multline*}
	Using $\varphi(0, x) \geq d_0^2~$ for all $x$ in $\Omega$ and Assumption \eqref{Positivity}, we deduce
	\begin{equation}
		s^{1/2} \alpha^2 \int_\Omega e^{2 s \varphi(0)} |\tilde q^{k+1}- Q |^2 dx
		\leq M \norm{W[Q]}_{H^1(0,T; L^\infty(\Omega))}^2 \int_{\Omega} e^{2 s\varphi(0)}|q^k-Q|^2 \, dx.
	\end{equation}
	Now, using the \textit{a priori} assumption \eqref{A-priori-bound}, i.e. $Q \in L^\infty_{\leq m}(\Omega)$, we easily check that this estimate cannot deteriorate in step 4 of Algorithm \ref{Algo}, which is there only to ensure that the sequence $q^k$ stays in $L^\infty_{\leq m}(\Omega)$ for all $k \in \N$. This completes the proof of Theorem~\ref{Thm-Algo2Converges}.
\end{proof}
It only remains to prove the former proposition.
\begin{proof}[Proof of Proposition \ref{PropDependenceG}]
	Let us write the Euler Lagrange equations satisfied by $Z^j$, for $j \in \{a, b\}$. For all  $z \in \mathcal{T}$, we have
	\begin{multline}
			\label{EulerLagrangeIdentity}
				\int_0^T \int_{\Omega} e^{2s \varphi} 
				(\partial_t^2 Z^j - \Delta Z^j + q Z^j - g^j) (\partial_t^2 z - \Delta z + q z) \,dxdt\\
				+ s\int_0^T \int_{\Gamma_0} e^{2s \varphi} (\partial_n Z^j - \mu) \partial_n z \, d\sigma dt 
				+ s^3 \iint_{\mathcal O} e^{2s\varphi}  Z^j z \,dxdt 
				=0.
	\end{multline}
	 Applying \eqref{EulerLagrangeIdentity} for $j= a$ and $j= b$ to $z= Z = Z^a - Z^b$  and subtracting the two identities, we obtain:
	\begin{multline*}
			 \int_0^T \int_{\Omega} e^{2 s \varphi} |\partial_t^2 Z - \Delta Z + q Z|^2\, dxdt  
			+ s \int_0^T \int_{\Gamma_0} e^{2s \varphi} |\partial_n Z|^2 \, d\sigma dt
			+ s^3 \iint_{\mathcal O} e^{2s\varphi}  |Z|^2 \,dxdt 
			\\
			= \int_0^T \int_{\Omega} e^{2 s \varphi} (g^a - g^b) (\partial_t^2 Z - \Delta Z + q Z) \,dxdt .
	\end{multline*} 
	This implies
	\begin{multline}
			\label{EstErrorsF}
				\frac{1}{2} \int_0^T \int_{\Omega} e^{2 s \varphi}  |\partial_t^2 Z - \Delta Z + q Z|^2 \,dxdt  
				+ s \int_0^T \int_{\Gamma_0} e^{2s \varphi} |\partial_n Z|^2 \, d\sigma dt 	\\
				+ s^3 \iint_{\mathcal O} e^{2s\varphi}  |Z|^2 \,dxdt 
					\leq \frac{1}{2}  \int_0^T \int_{\Omega} e^{2 s \varphi} |g^a- g^b|^2\,dxdt .
	\end{multline}
	Since the left hand side of \eqref{EstErrorsF}  is  precisely the right hand side of the Carleman estimate \eqref{Carlemgral}, applying Theorem~\ref{Carlemangral} to $Z$, we immediately deduce \eqref{EstMin-s}.
\end{proof}
%
%
\section{Technical issues on the minimization of the cost functional}\label{SecminJ}

The goal of this section is to give several details about the actual construction of an efficient numerical algorithm based on Algorithm~\ref{Algo}. 
The main step in Algorithm \ref{Algo} is to minimize the functional $J_{s,q}[\tilde \mu]$, that we recall here for convenience, 
\begin{eqnarray*}
		J_{s,q}[ \tilde \mu](z ) 
		=  \frac{1}{2} \int_0^T \int_{\Omega} e^{2s \varphi} |\partial_{t}^2 z - \Delta z + q z |^2\,dxdt 
		+ \frac{s}{2} \int_0^T \int_{\Gamma_0} e^{2s \varphi} | \partial_n z -   \tilde \mu |^2 \, d\sigma dt
		\\
		+ \frac{s^3}{2} \iint_{\mathcal O} e^{2s\varphi}  |z|^2 \,dxdt ,
\end{eqnarray*}
and which is minimized on the set 
\begin{multline}
	\label{Def-Mathcal-T}
	\mathcal{T}
	= 
		\Big\{
				z \in C^0([0,T]; H_0^1(\Omega)) \cap C^1([0,T]; L^2(\Omega)),  
				\\
				\hbox{ with }  \partial_{t}^2 z - \Delta z  \in L^2((0,T) \times \Omega ) \hbox{ and }
				z(0) = 0 \hbox{ in } \Omega
		\Big\}. 	
\end{multline}
Due to the presence of large exponential factors in the functional, the minimization of $J_{s,q}[\tilde \mu]$ is not a straightforward task from the numerical point of view, even if, as we emphasized earlier, the minimization of $J_{s,q}[\tilde \mu]$ is much less stiffer than the one of $K_{s,q}[\mu]$ defined in \eqref{FunctionalK-k} \cite{BaudouinDeBuhanErvedoza}. We therefore propose the two following ideas: 
\begin{itemize}
	\item Work on the conjugate variable $y = z e^{s \varphi}$. This change of unknown acts as a preconditioner. Details are given in Section~\ref{Subsec-Conjugate}.
	\item A progressive algorithm to minimize the functional $J_{s,q}[\tilde \mu]$ in subdomains in which the variations of the exponential factors are small, see Section~\ref{Subsec-Progressive}.
\end{itemize}
%
\subsection{Conjugate variable}\label{Subsec-Conjugate}
%
For $z$ in $\mathcal{T}$, we set $y = z e^{s\varphi}$, 
so that $y$ satisfies the following equation:
$$
	\left\{
		\begin{array}{ll}
					\partial_{t}^2 y - \Delta y + q y -2s\partial_t \varphi \partial_t y + 2 s \nabla \varphi \cdot \nabla y &\\
					\quad -s (\partial_t^2 \varphi - \Delta \varphi) y + s^2 (|\partial_t \varphi |^2 - |\nabla \varphi |^2) y = e^{s \varphi} (\partial_t^2 - \Delta + q )z , & \tn{in } (0,T)  \times \Omega, \\
					y = 0, & \tn{on } (0,T) \times \partial \Omega, 
				\\
					y(0) = 0, \quad \partial_t y(0) = z_1e^{s\varphi(0)}, \qquad& \tn{in } \Omega,
		\end{array}
	\right.	
$$
where $\partial_t \varphi = -2\beta t $, $\nabla \varphi = 2  (x-x_0) $, $\partial^2_t \varphi = - 2\beta $ and $\Delta \varphi = 2 d$. We set 
$\mathscr{L}_{s,q}$ defined by $\mathscr{L}_{s,q}  = e^{s\varphi}(\partial_t^2 - \Delta +q) e^{-s\varphi} $:
\begin{eqnarray}
	\mathscr{L}_{s,q} y 
	& = &
	 \partial_{t}^2 y - \Delta y + q y -2s\partial_t \varphi \partial_t y + 2 s \nabla \varphi \cdot \nabla y  -s (\partial_t^2 \varphi - \Delta \varphi) y 
	 \nonumber\\
	&&\quad  + s^2 (|\partial_t \varphi |^2 - |\nabla \varphi |^2) y \nonumber \\
	 & = &
	 \partial_{t}^2 y - \Delta y + q y + 4 s \beta t \partial_t y + 4 s (x-x_0) \cdot \nabla y  +2 s (\beta + d) y 
	 \nonumber\\
	&&\quad  + 4 s^2 (\beta^2 t^2 - |x-x_0|^2) y.
 	 \label{LsqConjugateExplicit}
\end{eqnarray}	
Thus, minimizing $J_{s,q}[\tilde \mu]$ in \eqref{FunctionalJ} on the set $\mathcal{T}$ is equivalent to minimize the functional $\widetilde J_{s,q}[\tilde \mu]$ defined by 
$$
	\widetilde J_{s,q}[\tilde \mu](y) 
		=  
	\frac{1}{2} \int_0^T \int_{\Omega} |\mathscr{L}_{s,q} y  |^2\,dxdt
	+ 
	\frac{s}{2} \int_0^T \int_{\Gamma_0} | \partial_n y -  \tilde \mu e^{s\varphi}  |^2 \ d\sigma dt
	+ 
	\frac{s^3}{2} \iint_{\mathcal O}  y^2 \,dxdt 
$$
on the same set $\mathcal{T}$. 
The minimization process for $\widetilde J_{s,q}[\tilde \mu]$ is then equivalent to the resolution of the following variational formulation: 
\\
Find $Y \in \mathcal{T}$ such that for all $y \in \mathcal{T}$, 
\begin{multline}
		\int_0^T \int_{\Omega} \mathscr{L}_{s,q} Y \mathscr{L}_{s,q} y  \,dxdt
		+ 
		s \int_0^T \int_{\Gamma_0}  \partial_n  Y \partial_n y \ d\sigma dt
		+ 
		s^3 \iint_{\mathcal O}  Y y \,dxdt 
		\\
		= 
		s \int_0^T \int_{\Gamma_0}  e^{s\varphi}\tilde \mu  \partial_n y \ d\sigma dt.\label{VarForm}
\end{multline}
From the Carleman estimate \eqref{CarlemY} applied to $y$ extended for negative times $t$ by $y(t) = - y(-t)$, the left-hand side of \eqref{VarForm} defines a coercive quadratic form, while the exponentials now appear only in the right hand side of \eqref{VarForm}. Therefore, no exponential factor appears anymore in the computation of the gradient of the functional $\widetilde J_{s,q} [\tilde \mu]$.
Our next goal is to deal with the exponential factor still in front of $\tilde \mu$.

\subsection{Progressive process}\label{Subsec-Progressive}
%
 The idea to tackle the exponential factor in the right hand side of \eqref{VarForm} is to develop a progressive process to compute the minimizer of $\widetilde J_{s, q}[\tilde \mu]$ as the aggregation of several problems localized in subdomains in which the exponential factors are all of the same order.
In this objective, from the smooth cut-off function $\eta$ equal to $1$ for $\tau \geq d_0^2$ defined in \eqref{eta}, we introduce $N$ cut-off functions $\{ \eta_j \}_{1\leq j \leq N}$ (these ones are not necessarily smooth) such that
\begin{equation}
	\label{Eta-j-withrespecto-Eta}
	\forall \tau \in \mathbb R, \quad \sum_{j=1}^N \eta_j(\tau) = \eta(\tau), 
\end{equation}
as illustrated in Figure~\ref{cutoff}.
%
\begin{figure}[h]
\begin{center}
\begin{tikzpicture}[line cap=round,line join=round,x=2cm,y=2cm]
\draw[->,color=black] (-0.5,0) -- (2.3,0);
\draw[->,color=black] (0,-0.5) -- (0,1.5);
\draw[color=black] (0,0) node[below left] {$0$};
\draw[color=black] (2.3,0.) node[below] {$\tau$};
\draw[color=black] (0.4,0.) node[below] {$d_0^2$};
\draw[color=black] (0,1) node[left] {$1$};
\draw[color=black] (0.5,1) node[above] {$\eta_3$};
\draw[color=black] (1.3,1) node[above] {$\eta_2$};
\draw[color=black] (2.,1) node[above] {$\eta_1$};
\draw[-, color=red] (-0.5,0) -- (0,0);
\draw[color=red, samples=1000,domain=0.:0.365] plot(\x,{(sin(500*(\x-1.625))+1)/2});
\draw[-, color=red] (0.365,1)-- (0.73,1);
\draw[color=red, samples=1000,domain=0.73:1.095] plot(\x,{1-(sin(500*(\x-1.63))+1)/2});
\draw[color=red, samples=1000,domain=0.73:1.095] plot(\x,{(sin(500*(\x-1.63))+1)/2});
\draw[-, color=blue] (1.095,1.)-- (1.46,1);
\draw[color=red, samples=1000,domain=1.46:1.825] plot(\x,{1-(sin(500*(\x-1.63))+1)/2});
\draw[color=red, samples=1000,domain=1.46:1.825] plot(\x,{(sin(500*(\x-1.63))+1)/2});
\draw[-, color=red] (1.825,1.)-- (2.,1);
\draw[-, dash pattern=on 2pt off 2pt, color=red] (2.,1.)-- (2.3,1);
\draw[-, color=red] (1.095,1.)-- (1.46,1);
\draw[dash pattern=on 2pt off 2pt] (0.365,1)-- (0,1);
\draw[dash pattern=on 2pt off 2pt] (0.365,0)-- (0.365,1);
\draw (0,1.5) node[left] {$\eta$};
\end{tikzpicture}
\caption{Example of  cut-off functions $\eta_j$ for $1\leq j \leq 3$.}\label{cutoff}
\end{center}
\end{figure}
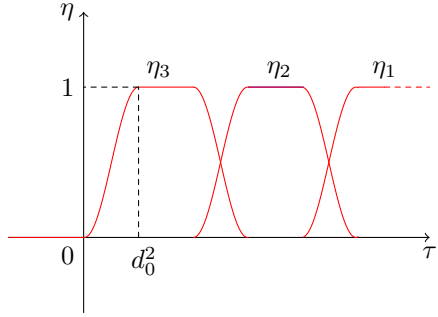

Therefore, the target flux $\tilde \mu = \eta(\varphi) \mu$ can be decomposed as follows:
\begin{equation}
	\label{Choice-of-tilde-muj}
	\tilde \mu = \eta(\varphi) \mu = \sum_{j = 1}^N \tilde \mu_j, 
\end{equation}
$$
\hbox{ where }  \tilde \mu_j (t,x) = \eta_j(\varphi(t,x)) \mu(t,x),   \forall  (t,x) \in (0,T) \times \Gamma_0, \forall j \in \{1, \cdots, N\}.
$$%
As the variational formulation in \eqref{VarForm} is linear in $\tilde \mu$, one immediately gets that, if for each $j \in \{1, \cdots, N\}$, we denote by $Y_j$ the minimizer of $\widetilde J_{s,q}[\tilde \mu_j]$ on $\mathcal{T}$, then the minimizer $Y$ of $\widetilde J_{s,q}[\tilde \mu]$ is simply given by 
$$
	Y = \sum_{j = 1}^N Y_j.
$$
The interest of this approach is that the target flux $\tilde \mu_j$ involves exponential terms in $\varphi$ on the support of $\eta_j(\varphi(t,x))$. This becomes particularly interesting if we impose that for each $j \in \{1, \cdots, N\}$,
\begin{equation}
	\label{Cond-Supp-EtaJ}
	\hbox{Supp } \eta_j \subset [a_j,b_j] \hbox{ with } b_j - a_j \leq C, 
\end{equation}
for some constant $C >0$. Indeed, in that case, we get 
$$
	\frac{\ds \sup_{\hbox{\scriptsize Supp\,} \eta_j(\varphi )} e^{s \varphi}}{ \ds \inf_{\hbox{\scriptsize Supp\,} \eta_j(\varphi )} e^{s \varphi} } \leq e^{s C}, 
$$
so that if $C \simeq 1/s$, all the exponentials are of the same order when computing $\tilde \mu_j$. Consequently, under the conditions \eqref{Eta-j-withrespecto-Eta}--\eqref{Choice-of-tilde-muj}--\eqref{Cond-Supp-EtaJ}, for all $j \in \{1, \cdots, N\}$, the minimization of $\widetilde J_{s,q}[\tilde \mu_j]$ over $\mathcal{T}$ is easier numerically than the direct minimization of $\widetilde J_{s,q}[\tilde \mu]$ over $\mathcal{T}$. Besides, this approach can be used, at least theoretically, to parallelize the minimization of $\widetilde J_{s,q}[\tilde \mu]$ over the set $\mathcal{T}$. 
\\

Let us present one possible way to construct the functions $\eta_j$ in practice, precisely the ones we used in our numerical experiments
(where we chose to use $C^\infty$ functions, even if it is not necessary). We set
$$
	d_0^2 = \inf_\Omega |x - x_0|^2
	\qquad \hbox{ and } \qquad
	L_0^2 = \sup_\Omega |x - x_0|^2.
$$
Let us then choose an integer $N \in \mathbb{N}^*$ and set $\varepsilon_0 = d_0^2/N$. Next, define the cut-off function $\eta$ as follows:
$$
	f(t) = \exp\left(\frac{-1}{t(\varepsilon_0-t)}\right), 
	\qquad 
	\text{and}
	\qquad 
	\eta(\tau) = 
	\left\{
		\begin{array}{ccc}
			0, & \text{if} & \tau \leq 0,
			\\
			1-\dfrac{\int_\tau^{\varepsilon_0} f(t) dt}{\int_0^{\varepsilon_0} f(t) dt}, &  \text{if} &0 < \tau < \varepsilon_0,
			\\
			1, &  \text{if} &\tau \geq \varepsilon_0.
		\end{array}
	\right.
$$
Thus we introduce the cut-off functions $\eta_j$ defined by the formula
\begin{eqnarray*}
		\eta_0(\tau) = \eta(\tau - L_0^2), 
		\quad
		\hbox{ and for } j \in \{1, \cdots, N\}, \quad \\
		\eta_j(\tau)  = \eta\left(\tau-L_0^2 \dfrac{N-j}{N}\right) - \eta\left(\tau-L_0^2 \dfrac{N-(j-1)}{N}\right)	.
\end{eqnarray*}
We then easily verify \eqref{Eta-j-withrespecto-Eta}, $\hbox{Supp\,} \eta_0 \subset \left] L_0^2 , +\infty\right[ $, and that 
$$
	\forall j \in \{1, \cdots, N\}, \quad \hbox{Supp\,} \eta_j \subset \left]L_0^2\left(1-\frac{j}{N}\right), L_0^2\left(1-\frac{j-1}{N}\right) +\frac{d_0^2}{N}  \right[.
$$
In particular, we have $\eta_0(\varphi(t,x)) = 0$ for all $(t,x) \in (-T,T)\times \Omega$ as $\varphi(t,x) \leq L_0^2$ for all $(t,x) \in (-T,T)\times \Omega$, so that we can omit $\eta_0(\varphi)$ in our approach. 
\\
By construction, the support of each $\eta_j$ for $j \in \{1, \cdots, N\}$ is included in an interval of size $(L_0^2 +d_0^2)/N$. We can then try to optimize the number $N$ of intervals in the progressive algorithm so that on each interval the weight function $\exp(s \varphi)$ varies of less than $5$ order of magnitude, for instance by taking $N$ as a function of $s$ as follows:
$$
	N = \Big\lfloor \frac{s (L_0^2+d_0^2)}{10}\Big\rfloor+1,
$$
where $\lfloor \cdot \rfloor$ denotes the integer part.

%
\section{Discrete setting for the algorithm}\label{SecImplAlg}

In this section, we present the technical solutions we have developed to implement numerically the algorithm. In order to simplify the presentation, from now on we focus on the one-dimensional case $\Omega = (0,L)$ and $\Gamma_0 = \{x = L\}$. We consider a semi-discrete in space and time-continuous approximation of our system, with a space discretization based on a finite-difference approximation method on a uniform mesh. In this restrictive setting, all our assertions can be fully proved rigorously by adapting the arguments in \cite{BaudouinErvedoza11,Baudouin-Ervedoza-Osses}. Though this might seem very restrictive, we believe that our approach can be generalized to fully discrete models and in higher dimensions for quasi-uniform meshes.
To begin with, we  introduce some notations for this 1-d space semi-discrete framework. The appropriate discrete Carleman estimate will follow. We will finally briefly present how we approximate the functional $J_{s,q}[\tilde \mu]$ in \eqref{FunctionalJ}. 
\subsection{Notations}
In our framework, the space variable $x \in [0,L]$ is taking values on a discrete mesh $[0,L]_h$ indexed by the number of points $N \in \N$. To be more precise, for $N \in \N$, we set $h = L/(N+1)$, $x_j  = j h$ for $j \in \{0, \cdots, N+1\}$, and $[0,L]_h = \{ x_j,\, j \in \{0,\cdots,  N+1\}\}$. For convenience, we will also note $(0,L)_h$, respectively $[0,L)_h$, the set of of discrete points $\{ x_j, \, j \in \{1, \cdots,N\} \}$, respectively $\{ x_j, \, j \in \{0, \cdots,N\} \}$.
\\
Below, we will use the subscript $h$ for discrete functions $f_h$ defined on a mesh of the form $[0,L]_h$ for some $N$, i.e. $f_h = ( f_{j}  )_{j \in \{0, \cdots, N+1\}}$. Analogously with the continuous case, we write:
\begin{equation}\label{intf}
	\ds\int_{(0,L)_h} f_h = h \ds\sum_{j=1}^{N} f_{j}, 
	\quad
	 \ds\int_{[0,L)_h} f_h	= h \ds\sum_{j=0}^{N} f_{j}. 
\end{equation}
We also make use of the following notation for the discrete operators:
\begin{eqnarray*}
	(\partial_h v_h)_{j} = \dfrac{v_{j+1} -v_{j-1}}{2 h}~;& 
	&
	(\partial^+_h v_h)_{j} = (\partial^-_h v_h)_{j+1} = \dfrac{v_{j+1} -v_{j}}{h}  ~;
	\\
	(\Delta_h v_h)_j & = &\dfrac{v_{j+1} - 2 v_{j}+v_{j-1}}{h^2}.
\end{eqnarray*}
By analogy with the definition of $\mathscr{L}_{s,q}$ in \eqref{LsqConjugateExplicit}, we finally introduce, for $s >0$ and $q_h$ a discrete potential, the conjugate operator $\mathscr{L}_{s,q_h,h}$ defined by 
\begin{equation}
 	 \label{LsqhConjugateExplicit}
	\mathscr{L}_{s,q_h,h} y_h = e^{s\varphi}(\partial_t^2 - \Delta_h + q_h) (e^{-s\varphi}y_h), 
\end{equation}
for $y_h$ functions of  $t \in (-T,T)$ and $x \in \{x_j, \, j \in \{1, \cdots, N\}\}$.
\\
Before going further, let us emphasize that the discrete operator $\mathscr{L}_{s,q_h,h}$ is different from the operator $\widetilde{\mathscr{L}}_{s,q_h,h}$ obtained by a naive discretization of $\mathscr{L}_{s,q}$ in \eqref{LsqConjugateExplicit} as follows:
\begin{eqnarray}
	\widetilde{\mathscr{L}}_{s,q_h,h} y_h
	  = 
	 \partial_{t}^2 y_h - \Delta_h y_h + q_h y_h + 4 s \beta t \partial_t y_h + 4 s (x-x_0) \partial_h y_h  
	 \nonumber
	 \\
	 +2 s (\beta + 1) y_h + 4 s^2 (\beta^2 t^2 - |x-x_0|^2) y_h, 
 	 \label{LsqhConjugateExplicit-Naive}
\end{eqnarray}
 for any function $y_h$ defined on $(-T,T)\times \{x_j, \, j \in \{1, \cdots, N\}\}$. 
\subsection{A discrete Carleman estimate for the discrete wave operator}\label{Subsec-DiscreteCarleman}
In this section, we provide the counterpart of Corollary \ref{Cor-Carleman-Conjugate} at the discrete level.
\begin{theorem}
	\label{Thm-DiscreteCarleman}
	Assume the multiplier condition \eqref{GCC-multiplier}-\eqref{GCC-Time} and $\beta \in (0,1)$ as in \eqref{Time-Condition}. Let $L>0$, take $x_0 < 0$, and define the weight function $\varphi$ as in \eqref{poids-double}. Then there exist  $s_{0}>0$, $N_0 >0$, $\varepsilon_0>0$ and a positive constant $M $ such that for all $ s \in [s_{0}, \varepsilon_0/h]$ and for all $N \geq N_0$,
	\begin{multline}
		\label{DiscreteCarleman}
		s \int_{-T}^{T} \int_{[0,L)_h} \left(|\partial_t y_h|^2 + |\partial_h^+ y_h|^2 + s^2 |y|^2 \right)\,dt 
		\\ \leq 
		M
		\int_{-T}^{T} \int_{(0,L)_h} |\mathscr{L}_{s,0, h} y_h |^2\,dt 
		+ Ms  \int_{-T}^{T} \left|\partial_h^- y_{N+1}(t)\right|^2\, dt 
		\\
		+M s^3 \int_{-T}^T \int_{(0,L_h)} \fcar_{(|t|, x_j) \in \mathcal O}  |y_h|^2 \,dt 
		+ M s h^2 \int_{-T}^T \int_{[0,L)_h} |\partial_t \partial_h^+ y_h|^2 \, dt,
	\end{multline}
	for all $y_h$ such that $y_{j}  \in H^2(-T,T)$ for all $j \in \{1, \cdots, N\}$, where $\mathcal{O}$ is defined in \eqref{mathcal-O}.
	\\
	Furthermore, if $y_h(0) = 0$ in $(0,L)_h$, the term 
	$
		s^{1/2} \ds\int_{(0,L)_h} |\partial_t y_h(0)|^2 
	$
	can be added to the left hand-side of \eqref{DiscreteCarleman}.
\end{theorem}
The proof of Theorem~\ref{Thm-DiscreteCarleman} is left to the reader as it follows step by step the proof of Theorem~\ref{Carlemangral} using discrete rules of integration by parts, which can be found in \cite[Lemma 2.6]{BaudouinErvedoza11}. It is actually particularly simple as the coefficients of $\mathscr{L}_{s,0,h}$ depend only on time or only on space variables.
\\

Let us now briefly comment Theorem~\ref{Thm-DiscreteCarleman}. First, compared with Corollary \ref{Cor-Carleman-Conjugate}, we see that the right-hand side of \eqref{DiscreteCarleman} contains one more term than \eqref{CarlemY}, namely 
\begin{equation}
	\label{High-Freq-Carl-Term}
	M s h^2 \int_{-T}^T \int_{[0,L)_h} |\partial_t \partial_h^+ y_h|^2\, dt.
\end{equation}
This is a high-frequency term. Indeed, as $h \partial_h^+$ is of the order of $h |\xi|$ for frequencies $\xi$, this term can be absorbed for large $s$ by the left hand-side of \eqref{DiscreteCarleman} for frequencies $\xi = o(1/h)$. However, for frequencies of the order of the mesh-size $h$, this term cannot be absorbed anymore by the left hand-side of \eqref{DiscreteCarleman}. This is not surprising in view of the lack of uniform observability for discrete waves, see \cite{Zua05Survey}, and the various comments done in \cite{BaudouinErvedoza11} on the discrete Carleman estimates for the wave equation with weight functions $\exp(s \psi) = \exp( s \exp(\lambda (\varphi + C_0)))$. 
\\
Let us also point out that as in \cite{BaudouinErvedoza11}, the parameter $s$ in Theorem \ref{Thm-DiscreteCarleman} cannot be made arbitrarily large as in Theorem \ref{Carlemangral}, but is limited to some $\varepsilon_0/h$. Roughly speaking, this condition comes from the following fact: 
\begin{equation}
	\label{Why-sh-small}
	\norm{\exp(s \varphi) \partial_h (\exp(-s \varphi))+  s \partial_x \phi}_{L^\infty((0,T)\times \Omega))} 
	\leq C sh, 
\end{equation}
so that the coefficients of $\mathscr{L}_{s,0,h}$ in \eqref{LsqhConjugateExplicit} and $\widetilde{\mathscr{L}}_{s,0,h}$ in \eqref{LsqhConjugateExplicit-Naive} are close only for $sh$ small enough.
\\
We end up this section with a warning. If we were considering the operator $\widetilde{\mathscr{L}}_{s,0,h}$ in \eqref{LsqhConjugateExplicit-Naive} instead of $\mathscr{L}_{s,0,h}$ in \eqref{LsqhConjugateExplicit}, the restriction on the size of the parameter $s$ could be removed as the errors done in the conjugation process, for instance in \eqref{Why-sh-small}, are inexistent. However, when conjugating back the discrete operator $\widetilde{\mathscr{L}}_{s,0,h}$, one would not obtain  the discretization of the wave operator $\partial_{tt} - \Delta_h$, and this would yield inaccuracies in our numerical experiments.
%
\subsection{Semi-discretization scheme and algorithm}
%
We now explain the discretization in space of the variational problem \eqref{VarForm}.
\\
First, we have to discretize the set $\mathcal{T}$ in \eqref{Def-Mathcal-T}. We thus introduce the set $\mathcal{T}_h$ defined as follows:
\begin{multline}
	\label{Def-Mathcal-T-h}
	\mathcal{T}_h = \{z_h \in H^2(0,T; \mathbb{R}^{N+2}) \hbox{ with } z_{0,h}(t) = z_{N+1,h} (t) = 0 \hbox{ for all }t  \in (0,T)\\
	 \hbox{ and } z_{j,h}(0) = 0 \hbox{ for all } j \in \{1, \cdots, N\} \}.
\end{multline}
Following Theorem \ref{Thm-DiscreteCarleman}, it is natural to discretize the  variational problem \eqref{VarForm} as follows: Find $Y_h \in \mathcal{T}_h$ such that for all $y_h \in \mathcal{T}_h$,
\begin{multline}
	\label{VarForm-h}
		\int_0^T \int_{(0,L)_h} (\mathscr{L}_{s,q_h,h} Y_h) (\mathscr{L}_{s,q_h,h} y_h)  \,dt
		+ 
		s \int_0^T  \frac{Y_{N_h,h}}{h} \frac{y_{N_h,h}}{h} \,dt\\
		+ 
		s^3  \int_{-T}^T \int_{(0,L_h)} \fcar_{(|t|, x_j) \in \mathcal O} Y_h y_h \, dt 
		\\
		+ 
		s h^2 \int_0^T \int_{[0,L)_h} (\partial_t \partial_h^+ Y_h) (\partial_t \partial_h^+ y_h) \, dt		
		= 
		s \int_0^T   e^{s\varphi}\tilde \mu  \left(\frac{- y_{N_h,h}}{h} \right) dt.
\end{multline}	
Actually we will use this variational formulation (\ref{VarForm-h}) in the numerical experiments.
\\
Compared with \eqref{VarForm}, we have added here the term
$$
		s h^2 \int_0^T \int_{[0,L)_h} (\partial_t \partial_h^+ Y_h) (\partial_t \partial_h^+ y_h),  		
$$
which is of course the counterpart of the term \eqref{High-Freq-Carl-Term} and aims at penalizing the spurious high-frequency waves which may appear in the discretization process. This term is indeed really helpful when considering noisy data, as we will illustrate in the numerical experiments in Section \ref{SecNum}. But this term also guarantees that the variational problem in \eqref{VarForm-h} is coercive uniformly with respect to the discretization parameter $h>0$, as it can be deduced immediately from Theorem \ref{Thm-DiscreteCarleman}. In particular, it allows us to prove the convergence of the algorithm given afterwards.

In order to state it precisely, by analogy with \eqref{FunctionalJ-mu-g}, for $h>0$, a discrete potential $q_h$, a parameter $s >0$, and $\tilde \mu \in L^2(0,T)$, $\tilde g_h \in L^2(0,T;\R^N)$, $\tilde \nu_h \in L^2(0,T ;\R^N)$, we introduce the discrete functional 
\begin{multline}
	\label{FunctionalJ-mu-g-nu-h}
	J_{s,q_h,h}[ \tilde \mu,\tilde g_h,\tilde \nu_h](z_h ) 
	=  \\
	\frac{1}{2} \int_0^T \int_{(0,L)_h} e^{2s \varphi} |\partial_{t}^2 z_h - \Delta_h z_h + q_h z_h - \tilde g_h |^2\, dt 
	+ 
	\frac{s}{2} \int_0^T  e^{2s \varphi} \left| \frac{-z_{N,h}}{h} -  \tilde \mu(t) \right|^2 \, dt
	\\
	+ 
	\frac{s^3}{2}  \int_{-T}^T \int_{(0,L)_h} \fcar_{(|t|, x_j) \in \mathcal O} e^{2s \varphi} |z_h|^2 \, dt 
	+ 
	\frac{s h^2}{2} \int_0^T \int_{[0,L)_h} e^{2s \varphi} |\partial_t \partial_h^+ z_h - \tilde \nu_h|^2 \, dt.
\end{multline}
defined on the set $\mathcal{T}_h$. Of course, one easily checks that the solution $Y_h \in \mathcal{T}_h$ of the variational formulation in \eqref{VarForm-h} corresponds to the minimizer $Z_h$ of $J_{s,q_h,h}[ \tilde \mu,0,0]$ over $\mathcal{T}_h$ through the formula $Y_h = e^{s \varphi} Z_h$.
\\

For any mesh-size  $h>0$, we define the discrete functions $w_{0,h}, w_{1,h}$ approximating the initial data $w_0, w_1$, and the discrete functions $f_h$ and $f_{\partial,h}$ approximating the source terms $f$ and $f_\partial$. 
We  construct Algorithm~\ref{AlgoDisc} as follows.

\begin{algorithm}
\label{AlgoDisc}~~\\
	{\bf \textit{Initialization:}} $q_h^0 = 0$.
	\\
	{\bf \textit{Iteration: From $k$ to $k+1$}}
	\\
	$\bullet $ {\it  Step 1 -}
	Given $q_h^k$, we set 
	$$
		\tilde \mu_h^k(t) = \eta(\varphi(t,L)) \partial_t \left(\frac{w_{N+1,h}^k(t)-w_{N,h}^k(t)}{h} - \partial_n W[Q](t,L)\right), \quad \hbox{ on } (0,T),
	$$ 
	where $w_h^k$ denotes the solution of  
	\begin{equation}\label{Eqwk-h}
			\left\{ \begin{array}{ll}
 				\partial_t^2 w_h-\Delta_h w_h+q_h^k w_h=f_h,\qquad   & \tn{in }(0,T) \times (0,L)_h,\\
				w_{0,h}(t) = f_\partial(t,0), \,  w_{N+1,h}(t) =f_{\partial}(t,L),  & \tn{on } (0, T),\\
				w_h(0)= w_{0,h}, \quad \partial_t w_h(0)= w_{1,h},\qquad  &\tn{in } (0,L)_h,
			\end{array}\right.
	\end{equation}
	corresponding to \eqref{Eqwk}  with the potential $q^k$ and $\partial_nW[Q]$ is the measurement in \eqref{Flux}. 
	\\
	And then set 
	\begin{equation}\label{eqn:nu}
		\tilde \nu_h^k = \partial_t \partial_h^+ \left(\eta(\varphi) \partial_t w_h[q_h^k]  \right)\quad \hbox{ in } (0,T) \times (0,L)_h.
	\end{equation}
	$\bullet $ {\it Step 2 -}
	We minimize the functional $J_{s,q_h^k,h}[\tilde \mu_h^k, 0,\tilde \nu_h^k]$ defined in \eqref{FunctionalJ-mu-g-nu-h}, for some $s>0$ that will be chosen independently of $k$, 
	on the trajectories $z_h \in \mathcal{T}_h$. 
	%
	%
	Let $\widetilde Z^k_h$ be the unique minimizer of the functional $J_{s,q_h^k,h}[\tilde \mu_h^k, 0,\tilde \nu_h^k]$.
	\\
	$\bullet $ {\it  Step 3 -}
	Set
	$$
		\tilde q_h^{k+1}  =  q_h^k + \dfrac{\partial_t \widetilde Z^k_h(0)}{w_{0,h}}, \quad \tn{ in } (0,L)_h.
	$$
	\\
	$\bullet $ {\it  Step 4 -}		
	Finally, set
	$$
		q^{k+1}_h = T_m (\tilde q_h^{k+1}), \quad \textit{ with }~ T_m(q)= 
		\left\{ \begin{array}{ll} q, &\textit{ if } |q| \leq m, \\ \textit{sign}(q) m, &\textit{ if } |q| \geq m., \end{array}\right. 
	$$
	 where $m$ is the a priori bound in \eqref{A-priori-bound}.
\end{algorithm}

One can then state a convergence result provided several assumptions are satisfied, basically corresponding to \eqref{RegAssumptions}--\eqref{Positivity}--\eqref{A-priori-bound} and the consistency of our approximation schemes. Namely we assume:
\\

$\qquad(i)~~$ Assumptions \eqref{RegAssumptions}--\eqref{Positivity}--\eqref{A-priori-bound} and \eqref{Time-Condition} are satisfied.
\\

$\qquad(ii)~$ There exists $\alpha>0$ independent of $h$ such that for all $h>0$, 
\begin{equation}
	\label{Positivity-h}
	\inf_{(0,L)_h} |w_{0,h}| \geq \alpha.	
\end{equation}

$\qquad(iii)$ There exists a sequence of discrete potential $(Q_h)_{h>0}$, each $Q_h$ being defined on $(0,L_h)$ such that: 
\begin{enumerate}
	\item For each $h>0$, $Q_h$ is bounded uniformly on $(0,L)_h$ by $m$:
		\begin{equation}
			\label{A-priori-bound-h}
			\sup_{(0,L)_h} |Q_h| \leq m. 	
		\end{equation}
	\item  The piecewise constant extensions of $Q_h$ strongly converge in $L^2(0,L)$ to $Q$ when $h \to 0$.	
		
	\item For each $h>0$, introducing $W_h[Q_h]$ the solution of 
		\begin{equation}\label{EqwQ-h}
			\left\{ \begin{array}{ll}
 				\partial_t^2 W_h-\Delta_h W_h+Q_h W_h=f_h,\qquad   & \tn{in }(0,T) \times (0,L)_h,\\
				W_{0,h}(t) = f_{\partial,h}(t,0), \,  W_{N+1,h}(t) =f_{\partial,h}(t,L),  & \tn{on } (0, T),\\
				W_h(0)= w_{0,h}, \quad \partial_t W_h(0)= w_{1,h},\qquad  &\tn{in } (0,L)_h,
			\end{array}\right.
		\end{equation}
		we get
		\begin{equation}
			\label{RegAssumptions-h}
			\sup_{h >0} \int_0^T \left| \sup_{(0,L)_h} |\partial_t W_h[Q_h]| \right|^2 \,dt < \infty, 
		\end{equation}
		and the following consistency assumptions:
	\end{enumerate}
		\begin{equation}
			\label{Consistency-h}
			\begin{array}{l}
				\ds 
				\lim_{h \to 0}\left(
				\int_0^T \eta(\varphi(t,L))^2 \left| \frac{\partial_t W_{N+1,h}[Q_h] - \partial_t W_{N,h}[Q_h]}{h} -  \partial_t \partial_n W[Q](t,L) \right|^2 \, dt
				\right) 
				= 0, 
				\smallskip\\
				\ds 
				\lim_{h \to 0}\left(
				\int_0^T \int_{[0,L)_h} |h \partial_h^+ \partial_t (\eta(\varphi) \partial_t W_{h}[Q_h])|^2 \, dt
				\right) 
				 = 0.
			\end{array}
		\end{equation}
These are natural assumptions regarding the inverse problem at hand. They have been widely discussed in \cite[Section 4]{BaudouinErvedoza11} and \cite[Section 4]{Baudouin-Ervedoza-Osses}. These two works give sufficient conditions for the existence of a sequence of discrete potential $Q_h$ satisfying \eqref{A-priori-bound-h}--\eqref{RegAssumptions-h}--\eqref{Consistency-h}. They also proved that, under some further suitable assumptions on the convergence of $f_h$, $f_{\partial,h}$, $w_{0,h}$, $w_{1,h}$, a sequence $Q_h$ satisfying \eqref{A-priori-bound-h} and \eqref{Consistency-h}$_{(1,2)}$ necessarily converges to the potential $Q$ in $L^2(0,L)$ (after having been extended as piecewise constant functions in a natural way).

We get the following result:
\begin{theorem}
	\label{Thm-Convergence-h}
	Under assumptions (i)-(ii)-(iii) above, Algorithm \ref{AlgoDisc} is well-posed for all $h>0$ small enough. Specifically, the discrete sequence $q_h^k$ satisfies for some constants $C_0, C_1 >0$  independent of $s >0$ and $h >0$,
	\begin{multline}
		\label{Conv-Iterate}
		\int_{(0,L_h)} e^{2 s \varphi} |q_h^{k+1} - Q_h|^2 
		\leq
		\frac{C_0}{\sqrt{s}} \int_{(0,L_h)} e^{2 s \varphi} |q_h^k - Q_h|^2 
		\\
		+ 
		C_1 s^{1/2} 
		\int_0^T \int_{[0,L)_h} e^{2s \varphi} |h \partial_h^+ \partial_t (\eta(\varphi) \partial_t W_{h}[Q_h])|^2 dt
		\\
		+ 
		C_1 s^{1/2} 
		\int_0^T e^{2s \varphi} \left| \frac{\partial_t W_{N+1,h}[Q_h] - \partial_t W_{N,h}[Q_h]}{h} -  \partial_t \partial_n W[Q](t,L) \right|^2  dt
		.
	\end{multline}
	In particular, for $s \geq 4 C_0^2$, we get, for all $k \in \mathbb{N}$,
	\begin{multline}
		\label{Conv-Iterate-2}
		\int_{(0,L_h)} e^{2 s \varphi} |q_h^k - Q_h|^2 
		\leq
		\frac{1}{2^k} \int_{(0,L_h)} e^{2 s \varphi} |Q_h|^2  	
		\\
		+ 
		2 C_1 s^{1/2} 
		\int_0^T \int_{[0,L)_h} e^{2s \varphi} |h \partial_h^+ \partial_t (\eta(\varphi) \partial_t W_{h}[Q_h])|^2  dt
		\\
		+ 
		2 C_1 s^{1/2} 
		\int_0^T e^{2s \varphi} \left| \frac{\partial_t W_{N+1,h}[Q_h] - \partial_t W_{N,h}[Q_h]}{h} -  \partial_t \partial_n W[Q](t,L) \right|^2 dt, 
	\end{multline}
	so that as $k \to \infty$, $q_h^k$ enters a neighborhood of $Q_h$, whose size depends on $h$ and $s$ and goes to zero as $h\to 0$ according to \eqref{Consistency-h}.
\end{theorem}

\begin{proof}
	We focus on the proof of \eqref{Conv-Iterate}. As in the continuous case, it mainly consists in showing that $\widetilde{Z}_h^k$ is close to $\tilde z_h^k = \eta(\varphi) z_h^k$, where 
	$$
		z_h^k = \partial_t \left( w_h[q_h^k] - W_h[Q_h]\right).
	$$
	The main idea is to remark that $z_h^k$ satisfies
	$$
			\left\{
			\begin{array}{ll}
					\partial_{t}^2 z_h^k - \Delta_h z_h^k + q_h^k z_h^k = g_h^k, & \tn{in } (0,T) \times (0,L)_h, 
				\\
					z^k_{0,h} = z^k_{N+1,h} = 0, & \tn{on } (0,T), 
				\\
					z^k_h(0) = 0, \quad \partial_t z_h^k(0) = z_{1,h}^k, \qquad& \tn{in } (0,L)_h,
			\end{array}
			\right.	
	$$
	with 
	$$
		g^k_h =   (Q_h- q_h^k) \partial_t W_h[Q_h], \quad  
		\quad 
		z_{1,h}^k =  (Q_h- q_h^k) w_{0,h}.
	$$
	In particular, $\tilde z_h^k$ satisfies:
	\begin{equation}\label{EqExactezk-h-tilde}
			\left\{
			\begin{array}{ll}
					\partial_{t}^2 \tilde z_h^k - \Delta_h \tilde z_h^k + q_h^k \tilde z_h^k = \tilde g_h^k, & \tn{in } (0,T) \times (0,L)_h, 
				\\
					\tilde z^k_{0,h} = \tilde z^k_{N+1,h} = 0, & \tn{on } (0,T), 
				\\
					\tilde z^k_h(0) = 0, \quad \partial_t \tilde z_h^k(0) = z_{1,h}^k, \qquad& \tn{in } (0,L)_h,
			\end{array}
			\right.	
	\end{equation}
	with
	$$
		\tilde g^k_h =   \eta(\varphi) (Q_h- q_h^k) \partial_t W_h[Q_h]
		+ [\partial_t^2 - \Delta_h, \eta(\varphi)] z^k_h.
	$$
	Moreover, one has the following boundary data
	%
	\begin{equation}
		\label{Observation-k-h}
			 \frac{- \tilde z_{N,h}^k(t)}{h} 
			 = 
			  \widetilde{\mu}^k_h(t) 
			 -  \delta_h(t)
			 \quad \tn{ on } (0,T),
	\end{equation}
where
$$
\delta_h(t) = \eta(\varphi(t,L))\partial_t\left(\frac{ W_{N+1,h}[Q_h](t) -  W_{N,h}[Q_h](t)}{h} -  \partial_n W[Q](t,L)\right).
$$
	Therefore, $\tilde z^k_h$ is the minimizer of the functional $J_{s,q^k_h,h}[\widetilde \mu^k_h - \delta_h, \tilde g_h^k, \tilde \nu_h^k - \hat \nu_h] $ where $\hat \nu_h$ is given by 
	\begin{equation}
		\label{Def-Hat-nu-h}
		\hat \nu_h =  \partial_h^+ \partial_t\left(\eta(\varphi) \partial_t W_h[Q_h]\right), \hbox { in } (0,T) \times (0,L)_h.
	\end{equation}
	But by construction, $\widetilde Z^k_h$ is the minimizer of $J_{s,q^k_h,h}[\widetilde \mu^k_h, 0,\tilde \nu_h^k]$. We thus only need to compare minimizers corresponding to the other coefficients ($\delta_h$, $\tilde g_h^k$ and $\hat \nu_h$). As in the proof of Proposition \ref{PropDependenceG}, using Euler-Lagrange formulation and using the Carleman estimate \eqref{DiscreteCarleman}, one easily gets:
	\begin{multline}
		\label{Comparison-Discrete}
		s^{1/2} \int_{(0,L)_h} e^{2 s \varphi(0)} |\partial_t Z^k_h(0) - \partial_t \widetilde z^k_h(0)|^2 
		\leq
		C 
		s \int_0^T e^{2 s\varphi(t,L)} |\delta_h|^2 \, dt
		\\
		+ 
		C 
		\int_0^T \int_{(0,L)_h} e^{2 s \varphi} |\tilde g_h^k|^2 \, dt
		+ 
		C 
		s h^2 \int_0^T \int_{(0,L)_h} e^{2s \varphi} |\hat \nu_h|^2\, dt.
	\end{multline}
	Following now the proof of Theorem \ref{Thm-Algo2Converges} we can show that 
	\begin{equation*}
		\int_0^T \int_{(0,L)_h} e^{2 s \varphi} |\tilde g_h^k|^2 \, dt
		\leq 
		C \int_{(0,L)_h} e^{2 s \varphi} |q^k_h - Q_h|^2 \, dt, 
	\end{equation*}
	while, by construction, 
	\begin{eqnarray*}
		s^{1/2} 
		\int_{(0,L)_h} e^{2 s \varphi(0)} |\partial_t Z^k_h(0) - \partial_t \widetilde z^k_h(0)|^2 
		 \geq 
		s^{1/2} 
		\alpha^2 \int_{(0,L)_h} e^{2 s \varphi(0)} |\tilde q_h^{k+1} - Q_h|^2\\
		 \geq 
		s^{1/2} 
		\alpha^2 \int_{(0,L)_h} e^{2 s \varphi(0)} |q_h^{k+1} - Q_h|^2.
	\end{eqnarray*}
	We then put together the two last estimates in \eqref{Comparison-Discrete}. Recalling that $\delta_h$ and $\nu_h$ are respectively given by \eqref{Observation-k-h} and \eqref{Def-Hat-nu-h}, we immediately obtain \eqref{Conv-Iterate}.
	The proof of estimate \eqref{Conv-Iterate-2} easily follows from \eqref{Conv-Iterate}. Indeed, by recurrence, one can easily show that, if $s \geq 4 C_0^2$, for all $k \in \N$,
	\begin{multline*}
		\int_{(0,L_h)} e^{2 s \varphi} |q_h^k - Q_h|^2 
		\leq
		\frac{1}{2^k} \int_{(0,L_h)} e^{2 s \varphi} |Q_h|^2  
		\\	
		+ 
		\left(\sum_{j =0}^{k-1} \frac{1}{2^j} \right)C_1 s^{1/2} \int_0^T \int_{[0,L)_h} e^{2s \varphi} |h \partial_h^+ \partial_t (\eta(\varphi) \partial_t W_{h}[Q_h])|^2  dt
		\\
		+ 
		\left(\sum_{j =0}^{k-1} \frac{1}{2^j} \right)C_1 s^{1/2} 
		\int_0^T e^{2s \varphi} \left| \frac{\partial_t W_{N+1,h}[Q_h] - \partial_t W_{N,h}[Q_h]}{h} -  \partial_t \partial_n W[Q](t,L) \right|^2 dt, 
	\end{multline*}
	which is slightly stronger than \eqref{Conv-Iterate-2} and concludes the proof of Theorem \ref{Thm-Convergence-h}.
\end{proof}
Note that we presented the above theoretical results by restricting ourselves to the $1$d case for the time continuous and space semi-discrete approximation of the inverse problem. Though, this analysis can very likely be carried on in much more general settings, for instance higher dimensions or fully discrete approximations. Of course, the key missing point is then the counterpart of the Carleman estimate in Theorem \ref{Carlemangral}. Despite important recent efforts for developing this powerful tool in the discrete setting, see in particular \cite{KlibanovSantosa91,BoyerHubertLeRousseau10a,BoyerHubertLeRousseau10b,BoyerHubertLeRousseau11,ErvedozadeGournay11,BoyerLeRousseau14} for discrete elliptic and parabolic equations, and \cite{BaudouinErvedoza11,Baudouin-Ervedoza-Osses} for discrete wave equations, the validity of discrete Carleman estimates in the discrete settings remains mainly limited to smooth deformations of cartesian grids for the finite-difference method.
\\
We would like also to emphasize that Theorem \ref{Thm-Convergence-h} is not a proper convergence theorem, as it only says that the sequence of discrete potentials $q_h^k$ will enter a neighborhood of $Q_h$ as $k \to \infty$. The size of this neighborhood, given by 
\begin{multline*}
	2C_1 s^{1/2} 
		\int_0^T \int_{[0,L)_h} e^{2s \varphi} |h \partial_h^+ \partial_t (\eta(\varphi) \partial_t W_{h}[Q_h])|^2  dt
	\\
	+ 
	2 C_1 s^{1/2} 
		\int_0^T e^{2s \varphi} \left| \frac{\partial_t W_{N+1,h}[Q_h] - \partial_t W_{N,h}[Q_h]}{h} -  \partial_t \partial_n W[Q](t,L) \right|^2 dt, 
\end{multline*}
see \eqref{Conv-Iterate-2}, is in fact very much related to the consistency error. It is nonetheless interesting to point out that choosing $s$ large to improve the speed of convergence of the algorithm also increases the size of this neighborhood. One should keep in mind that remark, which also applies in the presence of noise.
\begin{remark}\label{rem1}	
The choice made in \eqref{eqn:nu} does not seem natural because it is not based on the difference between $w_h[q_h^k] $ and $W_h[Q_h]$, the latter being unknown. It is also important to mention that for some reasons that we still do not fully understand the numerical results given by Algorithm \ref{AlgoDisc} with this choice show numerical instabilities. Instead, we propose to replace Algorithm \ref{AlgoDisc} by 

\begin{algorithm}
\label{AlgoDisc-Bis}~~\\
	Everything as in Algorithm \ref{AlgoDisc} except: 
	\\
	{\bf Iteration:}
	\\
	$\bullet$ Step 2: We minimize the functional $J_{s,q_h^k,h}[\tilde \mu^k_h,0,0]$ defined in (49), for some $ s\geq 0$
that will be chosen independently of $k$, on the trajectories $z_h \in \mathcal{T}_h$. Let $Z_h^k$ be the unique minimizer of the functional $J_{s,q_h^k,h}[\tilde \mu^k_h,0,0]$.
\end{algorithm}

With this choice, we do not know how to prove a convergence result of the algorithm similar to Theorem \ref{Thm-Convergence-h}. 
\\
However, this choice coincides more with the insights we have on the algorithm as $\tilde z_h^k$ in \eqref{EqExactezk-h-tilde} is the minimizer of $J_{s,q_h^k,h}[\tilde \mu^k_h- \delta_h,0,\tilde \nu_h^k - \hat \nu_h]$, and if convergence occurs, $\tilde \nu_h^k - \hat \nu_h$ should be small and converge to zero.
\\
The numerical results presented in Section \ref{SecNum} will all be performed using Algorithm \ref{AlgoDisc-Bis}. As we will see, this will lead to good numerical results, in agreement with the above insights. 
\end{remark}
\subsection{Full discretization}
When implementing Algorithm \ref{AlgoDisc} numerically, one should of course consider fully discrete wave equations. We will not give all the details of this discretization process, but simply state how we implement the minimization process of the functional $J_{s,q_h,h}$. 
\\
First, we shall of course consider a fully discrete version $J_{s,q_h^k,h,\tau}$ of the functional $J_{s,q_h,h}$ in \eqref{FunctionalJ-mu-g-nu-h}, in which we have implemented a time-discretization of $J_{s,q_h,h}$ of time-step $\tau$. This implies in particular that:
\begin{itemize}

	\item The minimization space $\mathcal{T}_h$ has to be replaced by the set of time discrete functions $z_{h,\tau} \in \R^{N_t} \times \R^{N+2}$, with $N_t = \lceil T/ \tau \rceil$ and  the corresponding boundary conditions.
	
	\item The time continuous integral in \eqref{FunctionalJ-mu-g-nu-h} shall be replaced by discrete sums $\tau \sum_{t \in [0,T] \cap \tau \mathbb{Z}}$.	
	\item The wave operator should be replaced by a time-discrete version of the space semi-discrete wave operator $\partial_{tt} - \Delta_h + q_h$. We simply choose to approximate $\partial_{tt}$ by the usual $3$-points difference operator $\Delta_{\tau}$ (similar to $\Delta_h$ but applied in time now). Similarly, the operator $\partial_t$ in the last term of \eqref{FunctionalJ-mu-g-nu-h} will be replaced by the operator $\partial_{\tau}^+$ which is the approximation of $\partial_t$ computed with the subsequent time-step.
	
	\item The solution $w_h$ of \eqref{Eqwk-h} has to be computed on a fully discrete version of \eqref{Eqwk-h}. We choose to discretize using an explicit Euler method.

	\item There is no need to add a new penalization term for high-frequency spurious terms as we will impose a Courant-Friedrichs-Lax (CFL) type condition $\tau \leq h$, so that the last term in \eqref{FunctionalJ-mu-g-nu-h} already penalizes the spurious high-frequency solutions.
		
\end{itemize}

Of course, the strategies that have been presented in Section \ref{SecminJ}  to make the numerical implementation of the minimization of $J_{s,q}$ more efficient can be successfully applied to the functional $J_{s,q_h^k,h,\tau}$ as well. Namely, in the implementation of Algorithm \ref{AlgoDisc-Bis}, we will always work on the conjugated functional, i.e. the one given in the conjugated variable $y = e^{s \varphi} z$, and we will always decompose the domain using the progressive argument presented in Section \ref{Subsec-Progressive}. 

We also point out that the minimization of the quadratic functional $J_{s,q_h, h, \tau}$ obtained that way can be recast using a variational formulation similar to \eqref{VarForm}, which presents the advantage to underline the fact that we are actually solving a sparse linear system. We therefore use a Compressed Sparse Row (CSR) tool as sparse matrix storage format and solve the linear system thanks to an LU factorization. 

The iterative process on the potential is supposed to reach convergence when the following stop criterion is satisfied
\begin{multline}\label{stop}
		\ds \frac{\ds\int_{(0,L)_h} |q_h^{k+1}-q_h^k|^2}{\ds\int_{(0,L)_h} |q_h^1-q_h^0|^2} \leq \epsilon_0
		\quad \hbox{ or } \quad \\
		\frac{1}{\ds \int_{[0,T)_{\tau}} |  \partial_n W[Q](t,L)|^2 } \ds \int_{[0,T)_{\tau}} \left| (\partial_h^- w_h^k)_{N+1}(t) - \partial_n W[Q](t,L) \right|^2   \leq \epsilon_1.	
\end{multline}
for given choices of the parameters $\epsilon_0>0$ and $\epsilon_1 >0$, in which the integrals have to be interpreted in the discrete sense.

\section{Numerical results}\label{SecNum}

This section is devoted to the presentation of some numerical examples to illustrate the properties of the reconstruction algorithm and its efficiency. All simulations are executed with the software \textsc{Scilab}. The source codes are available on request. 

\subsection{Synthetic noisy data}

In this article, we work with synthetic data. To discretize the wave equations with potential \eqref{EqW}, we use a finite differences scheme in space and a $\theta$-scheme in time. The space and time steps are denoted by $h$ and $\tau$ respectively. We set  $L = (N_{x}+1) h $ and $T = N_t \tau$, and we define, for $0 \leq j \leq N_x+1$ and $0\leq n \leq N_t$, $W_j^n$ a numerical approximation of the solution $W(t^n,x_j)$ with $t^n=n\tau$ and $x_j=jh$. It is solution of the following system:
\begin{equation}\label{EqDiscr}
	\left\{\begin{array}{l}
			\dfrac{W_j^{n+1}-2W_j^n+W_j^{n-1}}{\tau^2} - \dfrac{\theta}{2} (\Delta_h W_h)^{n+1}_j -(1-\theta)(\Delta_h W_h)^n_j ~
			\\
			 \hfill{}- \dfrac{\theta}{2}  (\Delta_h W_h)^{n-1}_j + Q(x_j) W_j^n = f(t^n,x_j),
				 \\ 
			W_j^1 =  w_0(x_j) +\tau w_1(x_j)+ \dfrac{\tau^2}{2} \left((\Delta_h  w_0)(x_j) - q(x_j) w_0(x_j)+f(0,x_j)\right),   
			\\
			 W_j^0 = w_0(x_j) \hfill{}				 1\leq j \leq N_x,\vspace{0.1cm}\\
			W_{0}^n = f_\partial(t^n,0) \quad \text{and} \quad 
			W_{N_x+1}^n = f_\partial(t^n,L), 
				\hfill{}	 1\leq n\leq N_t.
				\end{array}
	\right.
\end{equation}
Then, we compute $\mathscr{M}_{\tau}$ the counterpart of the continuous measurement $\mathscr{M}$ given in \eqref{Flux} as follows:
$$
\mathscr{M}_{\tau}(t^n) = \dfrac{W^n_{N_x+1} - W^n_{N_x}}{h}, \qquad  0 \leq n \leq N_t.
$$
On the computed data, we may add a Gaussian noise:
\begin{equation}
	\mathscr{M}_{\tau}(t^n) \longleftarrow (1+\alpha \mathcal{N}(0,0.5)) \mathscr{M}_{\tau}(t^n), \quad 0 \leq n \leq N_t
\end{equation}
where $\mathcal{N}(0,0.5)$ satisfies a centered normal law with deviation $0.5$ and $\alpha$ is the level of noise. Note that the model of noise, that we chose, is a multiplicative noise. It allows to model the experimental error in the measurements.

One of the main drawbacks of the method presented in Algorithm \ref{AlgoDisc-Bis} is that we have to derive in time the observation flux. On Figure~\ref{fig:flux}, we plot the flux $\mathscr{M}$ with respect to time (on the left hand side) and of its time derivative (on the right hand side). For each of the graphs, the red line is the exact value and the black line the generated noisy data. It shows that even a small perturbation on the observations gives rise to a large perturbation on its derivative. In order to partially remedy to this problem, we regularize the data thanks to a convolution process with a Gaussian:
\begin{equation}
\label{denoising}
	\mathscr{M}(t) \longleftarrow \frac{1}{\sqrt{2\pi}} \int_0^\infty   \mathscr{M}(t-r) \exp\left(-\dfrac{r^2}{4}\right) dr.
\end{equation}
The number of iterations in this regularization process must be chosen in accordance to the {\it a priori} knowledge of the noise level. On Figure~\ref{fig:flux}, the new regularized data that we use as an entry for the algorithm is plotted in blue.

\begin{figure}[h]
	\centering
	\subfloat[Flux $\mathscr{M}(t)$]
		{\includegraphics[width=0.4\textwidth]{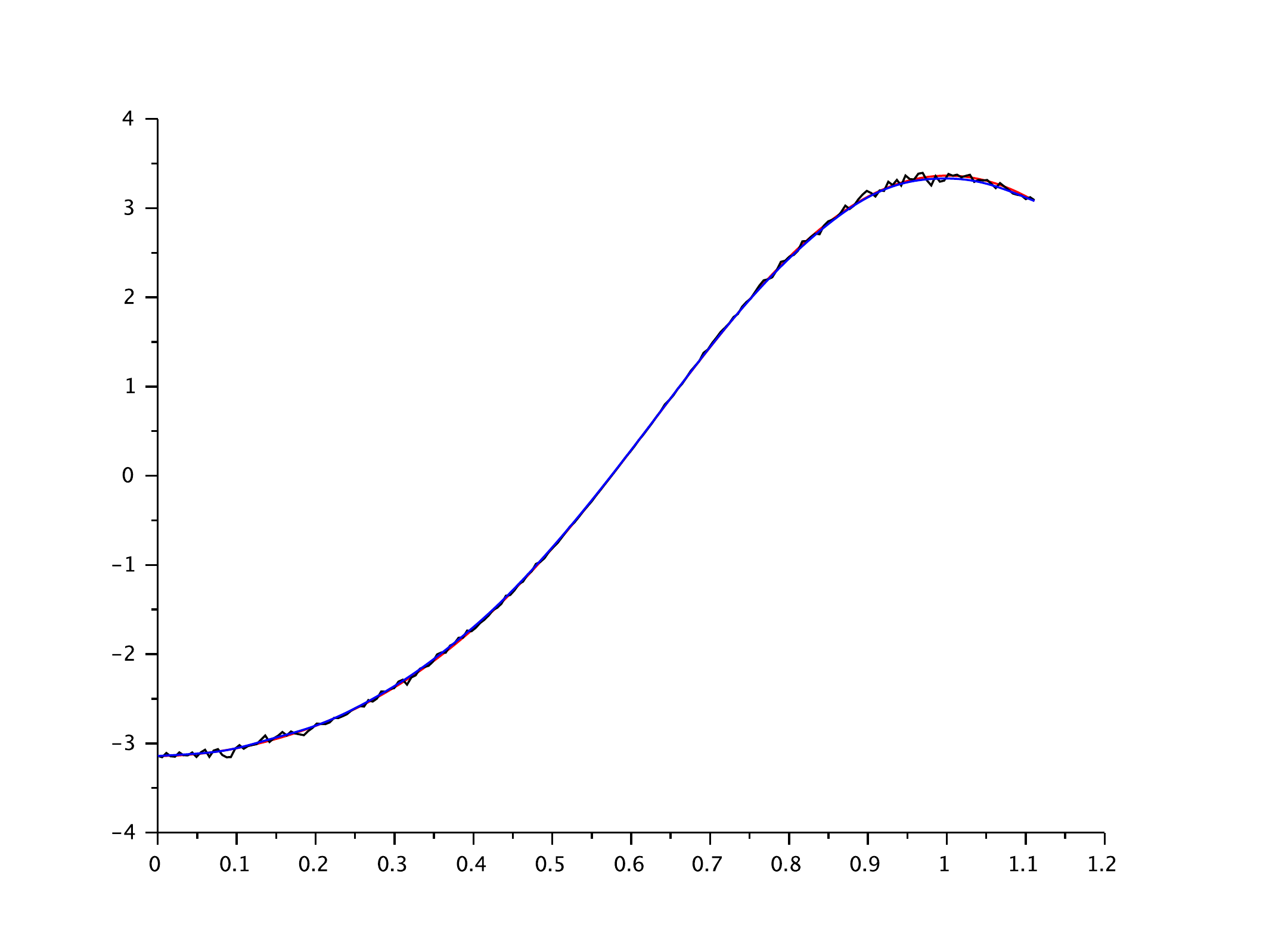}}\hspace{1cm}
	\subfloat[Time derivative of the flux $\partial_t \mathscr{M}(t)$]	
		{\includegraphics[width=0.4\textwidth]{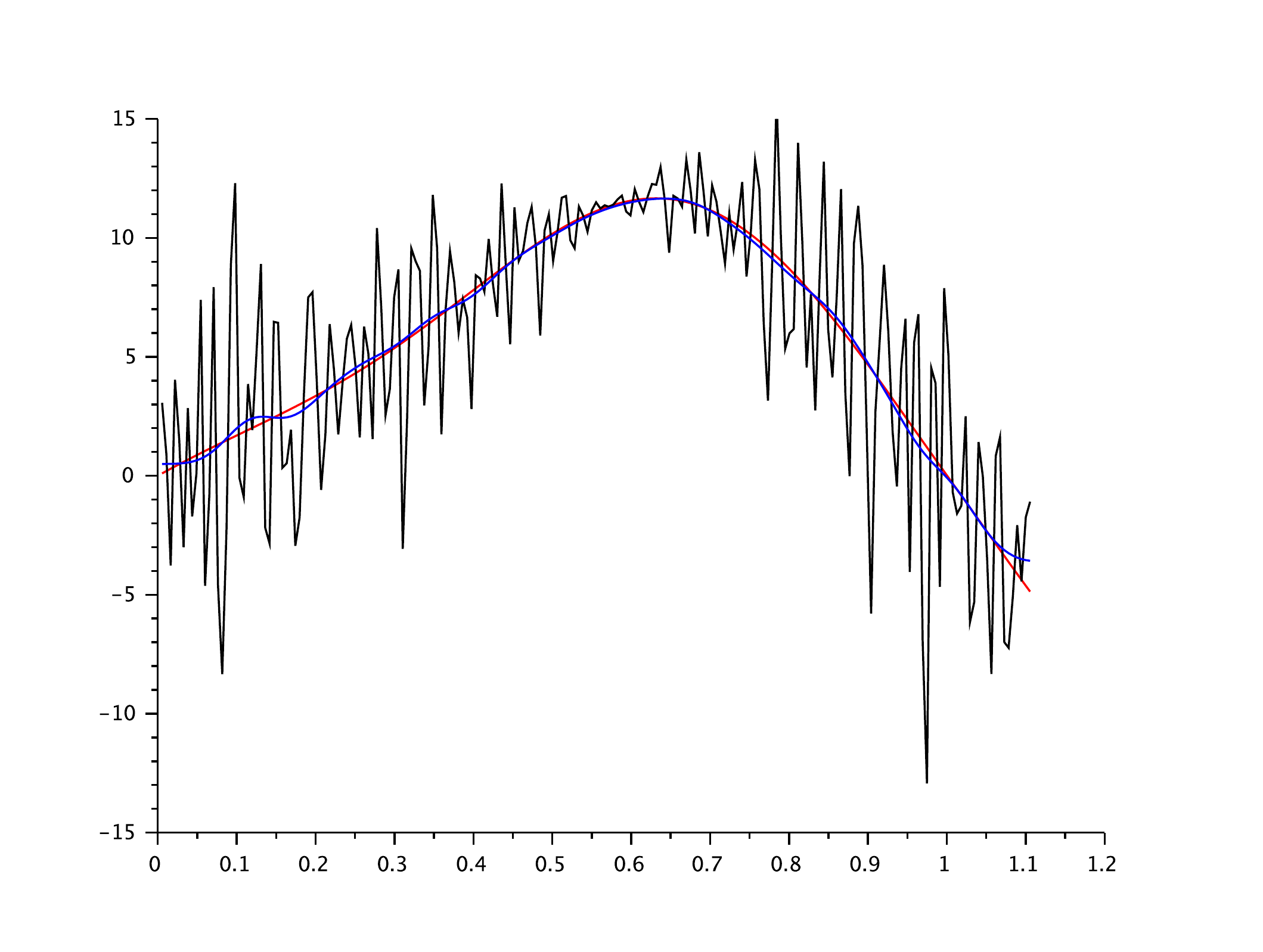}}
	\caption{The measurement $\mathscr{M}$ in the presence of $2\%$ noise. \label{fig:flux}}
\end{figure}

In order to avoid the {\it inverse crime}, we use neither the same schemes nor the same meshes for the direct and the inverse problems. Hence, we solve \eqref{EqW} thanks to an implicit scheme ($\theta = 1$) with $\tau = 0.00033$ and $h = 0.00025$ and we use an explicit scheme ($\theta = 0$) for equation \eqref{Eqwk} in Algorithm \ref{AlgoDisc-Bis}, with $\tau = 0.01$ and $ h =  \dfrac{\tau}{\textnormal{CFL}}$. Table \ref{fig:values} gathers the numerical values used for all the following examples, unless specified otherwise where appropriate. In all the figures, the exact potential that we want to recover is plotted by a red line, the numerical potential recovered by the algorithm is represented by black crosses.
\begin{table}[h]
\begin{center}
\begin{tabular}{|c|c|c|c|c|c|c|c|c|c|c|c|c|}
\hline
   $L$ & $f$ & $f_\partial$ & $w_0$ & $w_1$ & $x_0$ & $\beta$ & $T$ & $s$ & $m$ & CFL\\
   \hline
  $1$ & $0$ & $2$ & $2+\sin(\pi x)$ & $0$ &  $-0.3$ & $0.99$ & $1.3$ & $100$ & $3$ & 0.9~ \hbox{ or } ~1 \\
   \hline
\end{tabular}
\caption{Numerical values for the variables.}\label{fig:values}	
\end{center}
\end{table}

\subsection{Simulations from data without noise}

In this subsection, we present the results obtained for $CFL = 1$. For that very special choice, the explicit scheme used to discretized \eqref{EqW} is of order $2$. We observe that in this case, the additional regularization term \eqref{High-Freq-Carl-Term} in the functional does not seem to be necessary and $s$ can be chosen as large as wanted independently of the value of $h$  to achieve convergence. The successive results at each iteration of Algorithm~\ref{AlgoDisc-Bis} in the case of the reconstruction of the potential $Q(x) = \sin(2 \pi x)$ are presented in Figure~\ref{fig:iteration}. One can observe that in less than 3 iterations, the convergence criteria \eqref{stop} for $\epsilon_0 = 10^{-5}$ is met. 

\begin{figure}[!h]
	\centering
	\subfloat[$q^0$]
		{\includegraphics[width=0.25\textwidth]{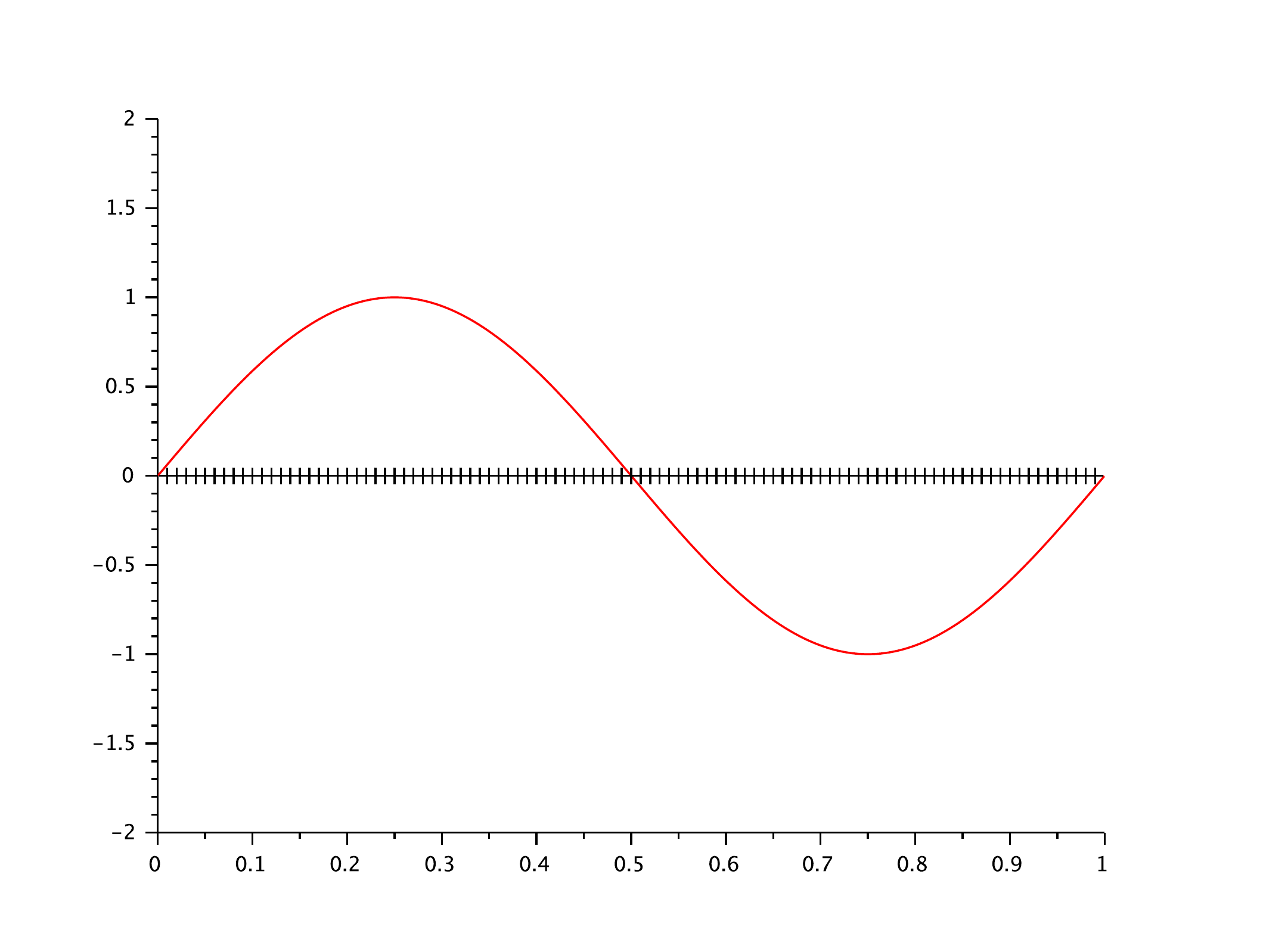}}
	\subfloat[$q^1$]	
		{\includegraphics[width=0.25\textwidth]{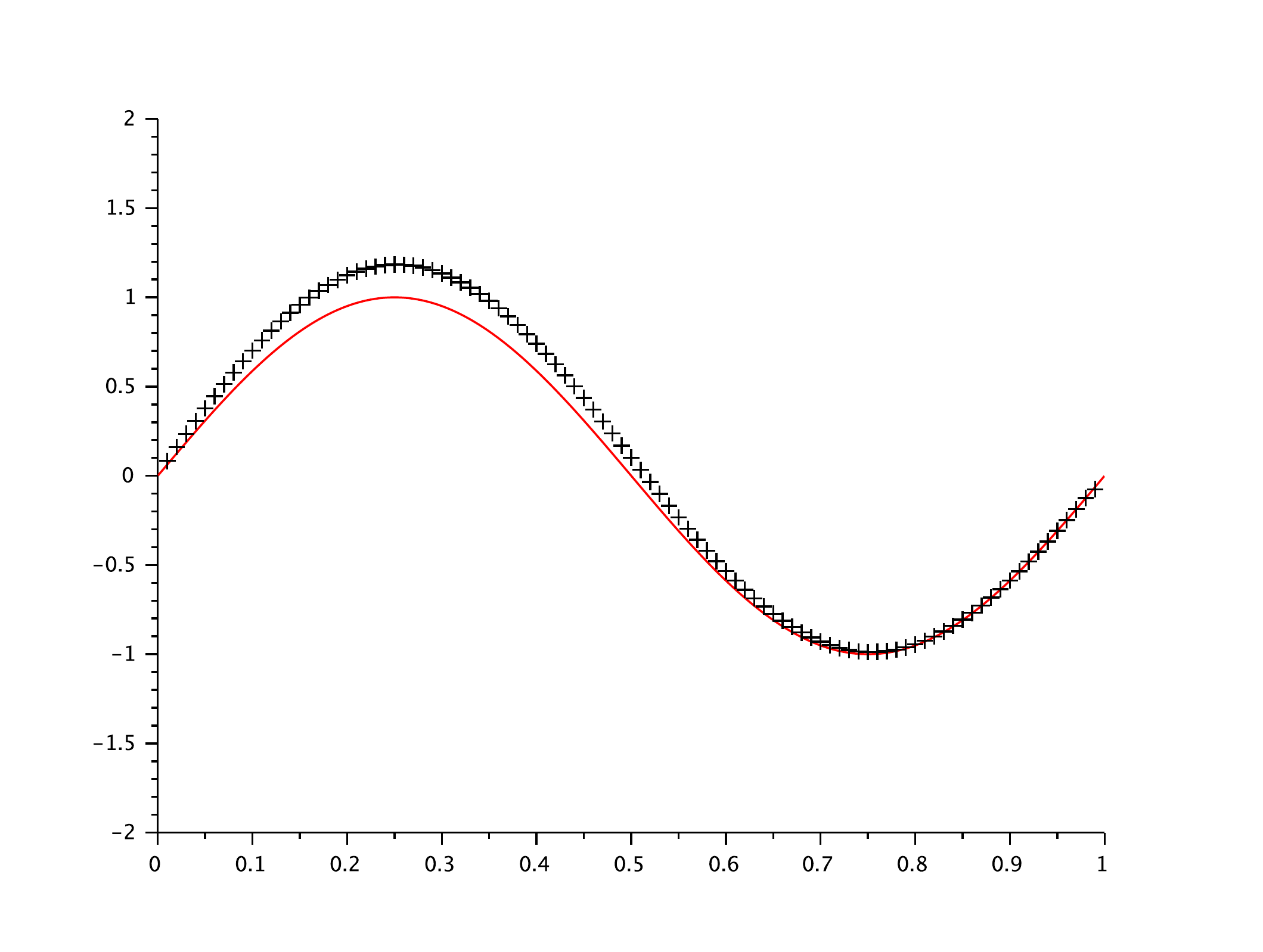}}
	\subfloat[$q^2$]
		{\includegraphics[width=0.25\textwidth]{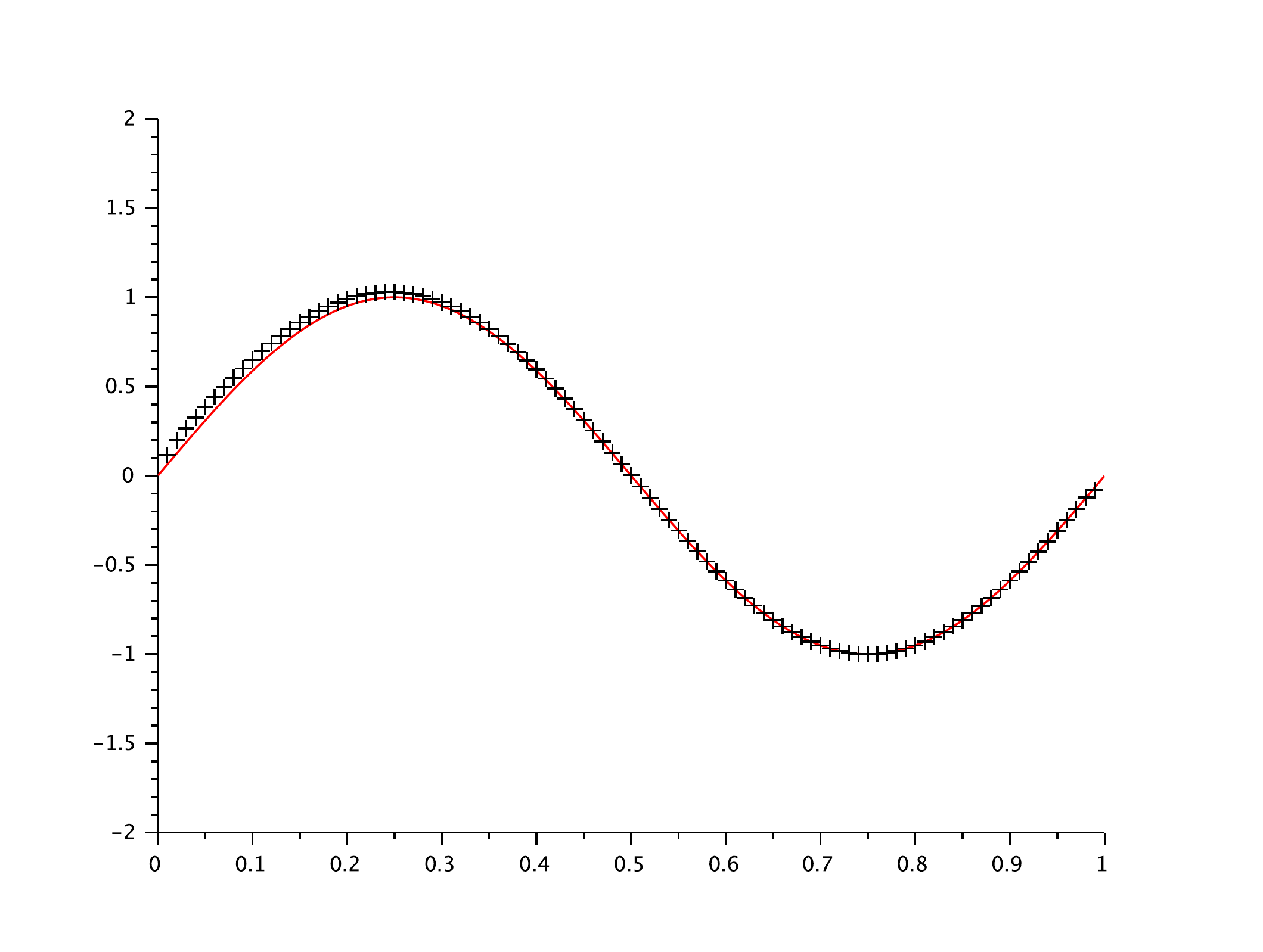}}
	\subfloat[$q^3$]
		{\includegraphics[width=0.25\textwidth]{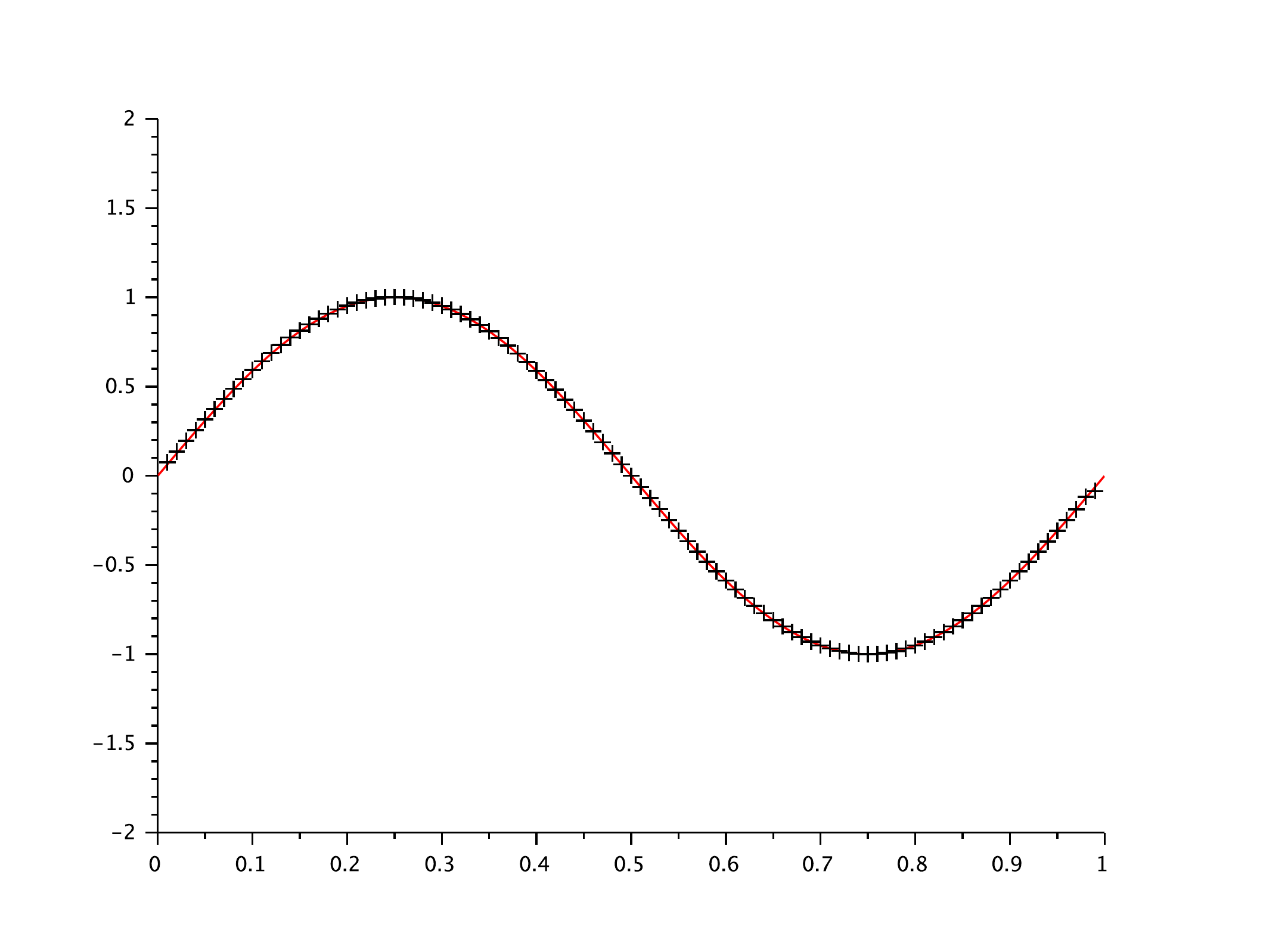}}	
	\caption{Illustration of the convergence of the algorithm for $CFL=1$ and $s=100$. \label{fig:iteration}}
\end{figure}

Using the same target potential, Figure~\ref{fig:progressive} illustrate the progressive process on the first iteration of Algorithm~\ref{AlgoDisc-Bis}. From an initial data $q_0^0 = 0$, we represent successively 
$$
q^0_j = q^0_{j-1} + \frac{\partial_t Y^0_j(0)}{e^{s\varphi(0)}w_0}, \qquad 1\leq j \leq 5,
$$
where $Y_j^0$ is the minimizer of $\tilde J_{s,q^0}[\tilde \mu_j^0]$.
\begin{figure}[!h]
	\centering
	\subfloat[$q^0_0=q^0$]		
		{\includegraphics[width=0.3\textwidth]{q0}}\hfill
	\subfloat[$q^0_1$]
		{\includegraphics[width=0.3\textwidth]{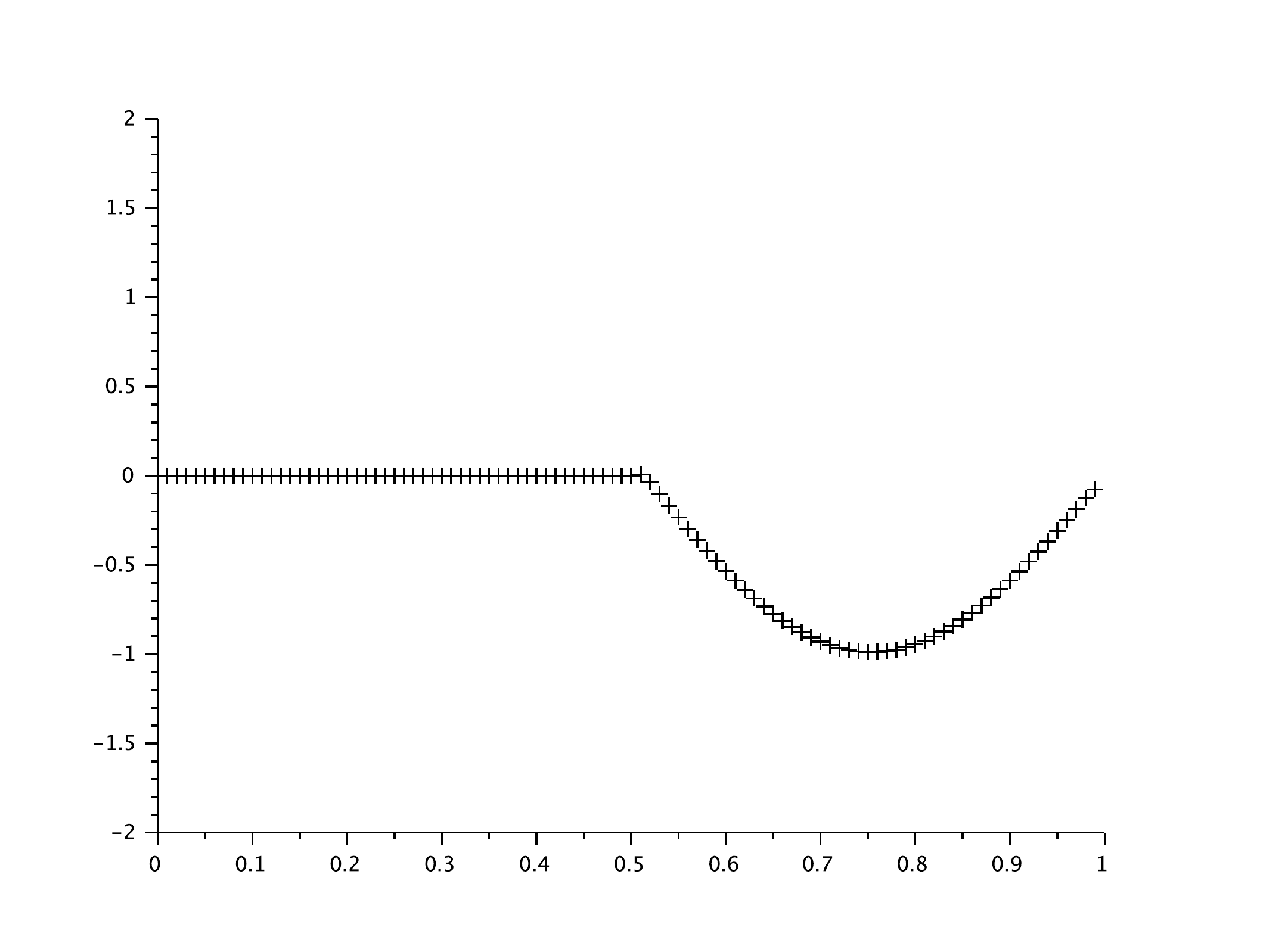}}\hfill
	\subfloat[$q^0_2$]	
		{\includegraphics[width=0.3\textwidth]{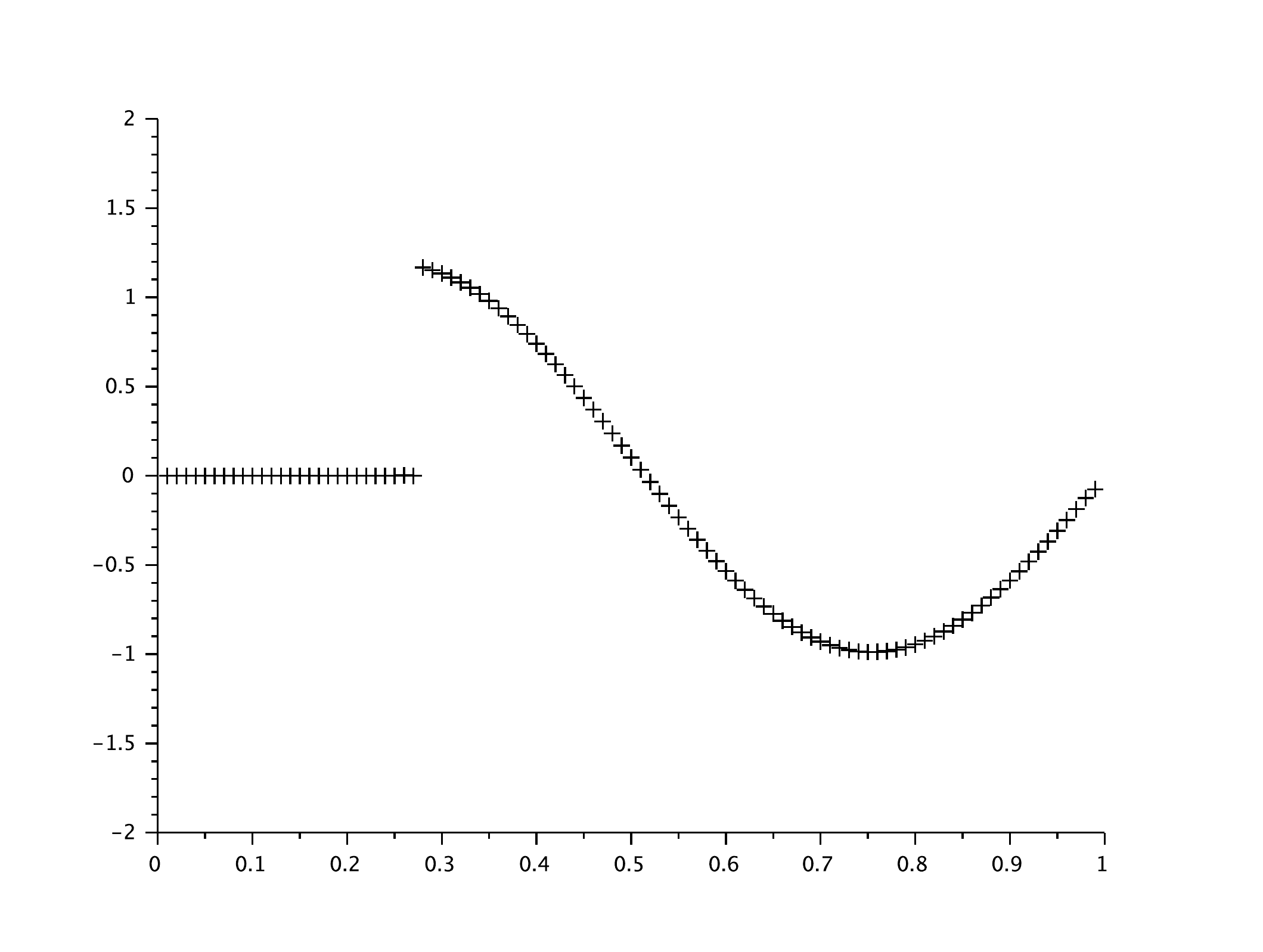}}\\
	\subfloat[$q^0_3$]
		{\includegraphics[width=0.3\textwidth]{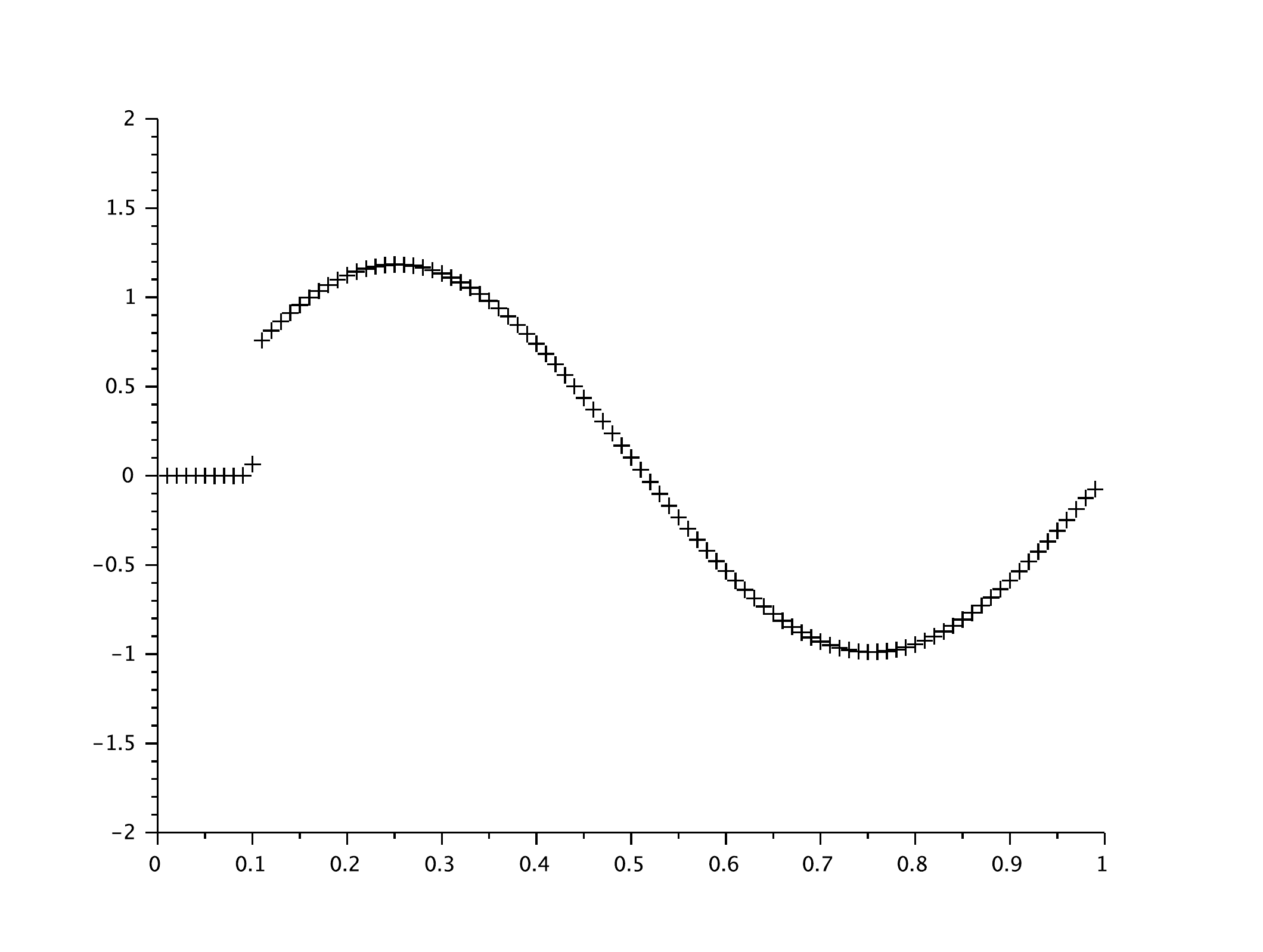}}\hfill
	\subfloat[$q^0_4$]
		{\includegraphics[width=0.3\textwidth]{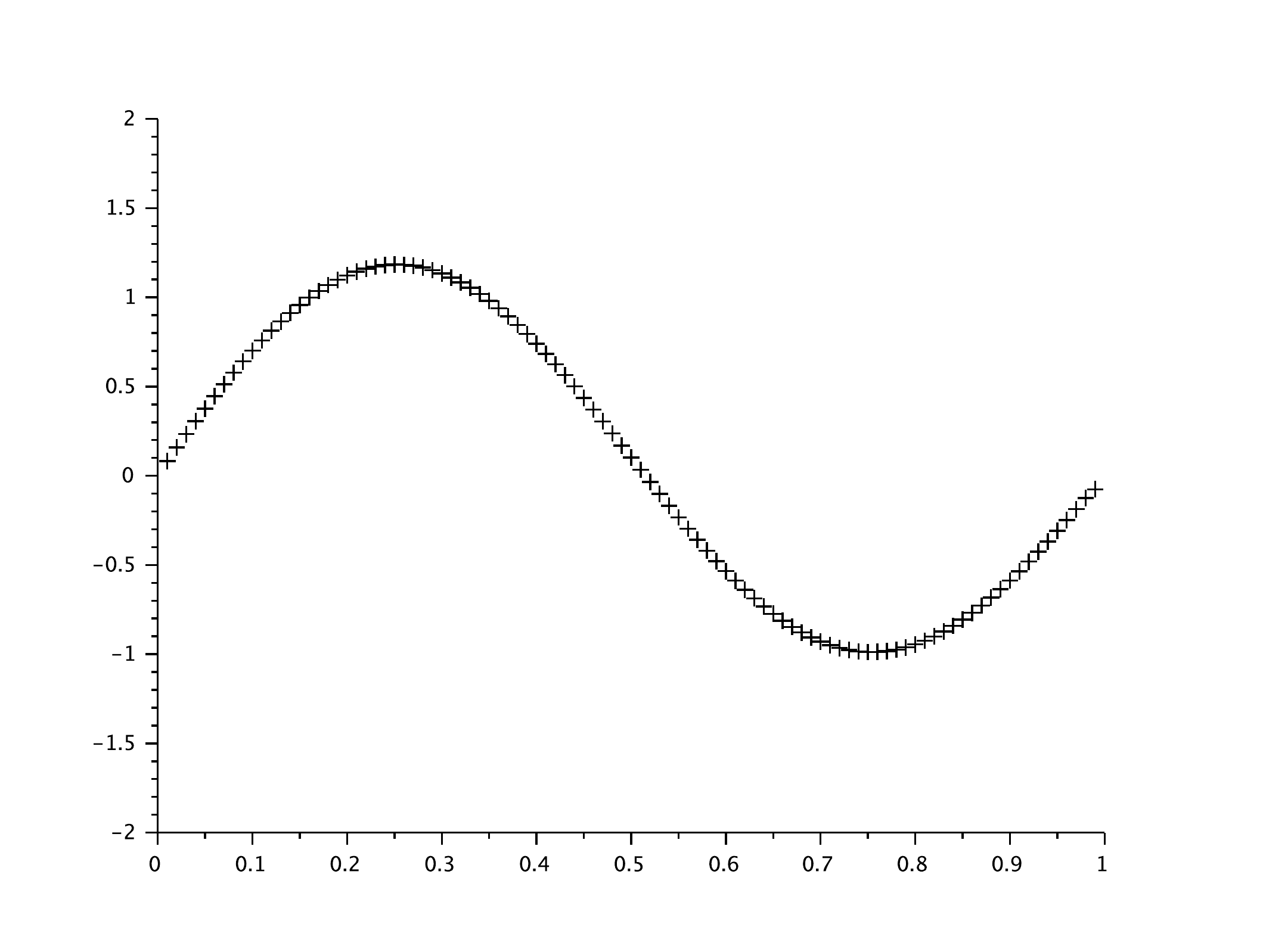}}\hfill		
	\subfloat[$q^0_5 = q^1$]
		{\includegraphics[width=0.3\textwidth]{q1}}\hfill	
	\caption{Illustration of the progressive process for $Q(x)=\sin(2 \pi x)$ for $s=100$. \label{fig:progressive}}
\end{figure}

In Figure~\ref{fig:examples}, several results of reconstruction of potentials obtained using Algorithm~\ref{AlgoDisc-Bis} in the absence of noise are given.
\begin{figure}[!h]
	\centering
	\subfloat[$Q=-x$]
		{\includegraphics[width=0.3\textwidth]{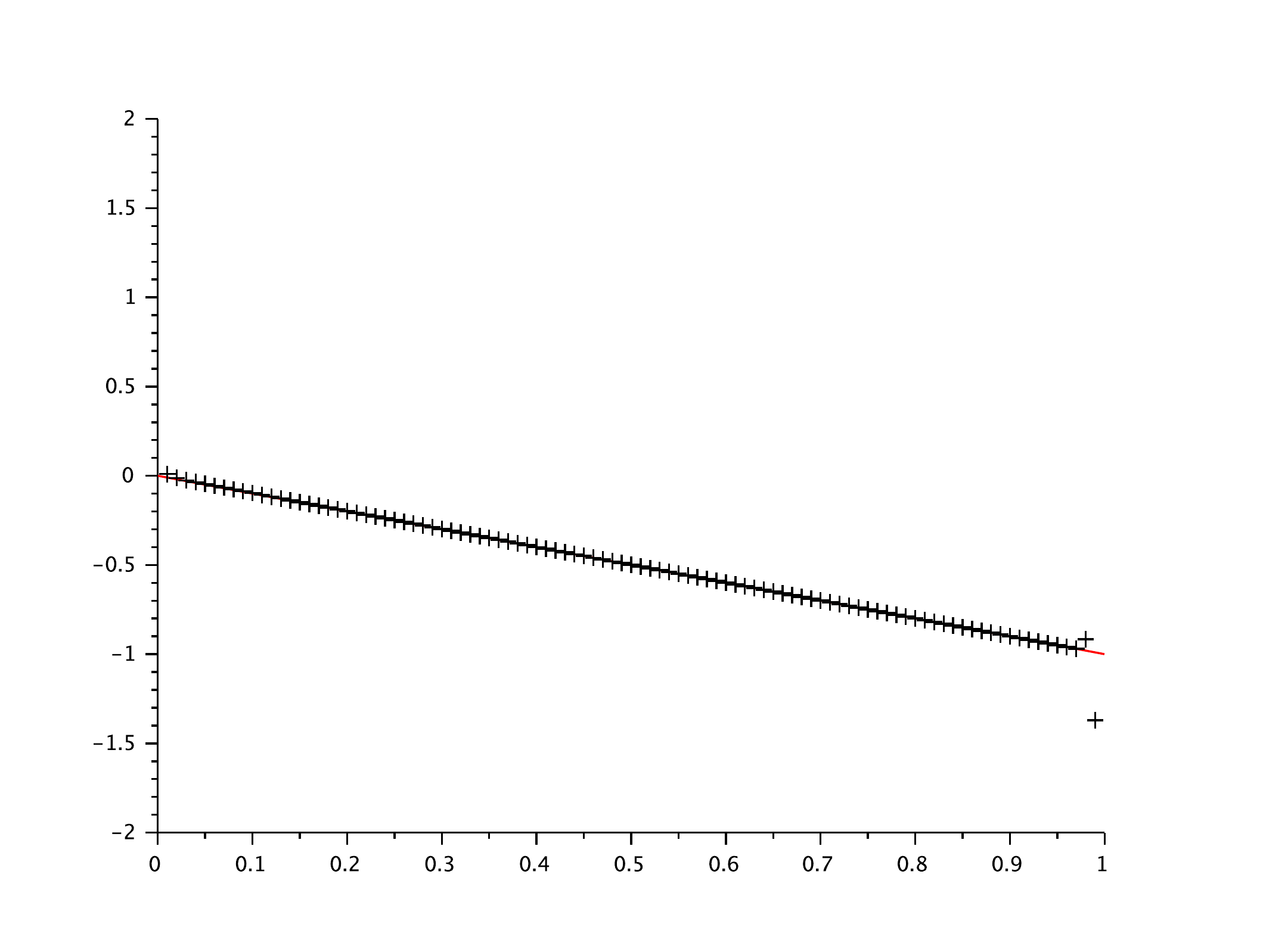}}\hfill
	\subfloat[$Q$ heaviside]
		{\includegraphics[width=0.3\textwidth]{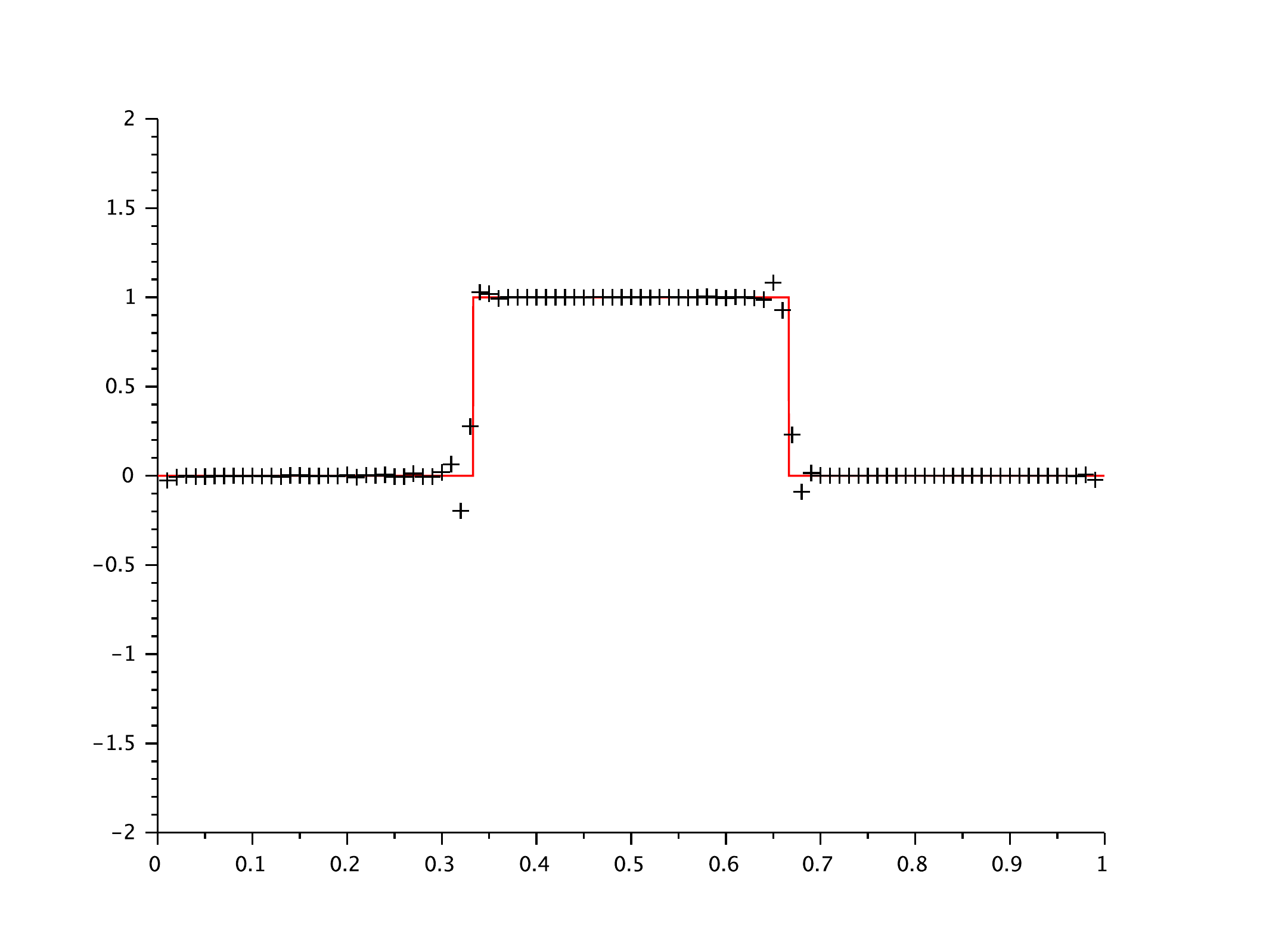}}\hfill
	\subfloat[$Q=\sin(\frac{x}{1-x})$]	
		{\includegraphics[width=0.3\textwidth]{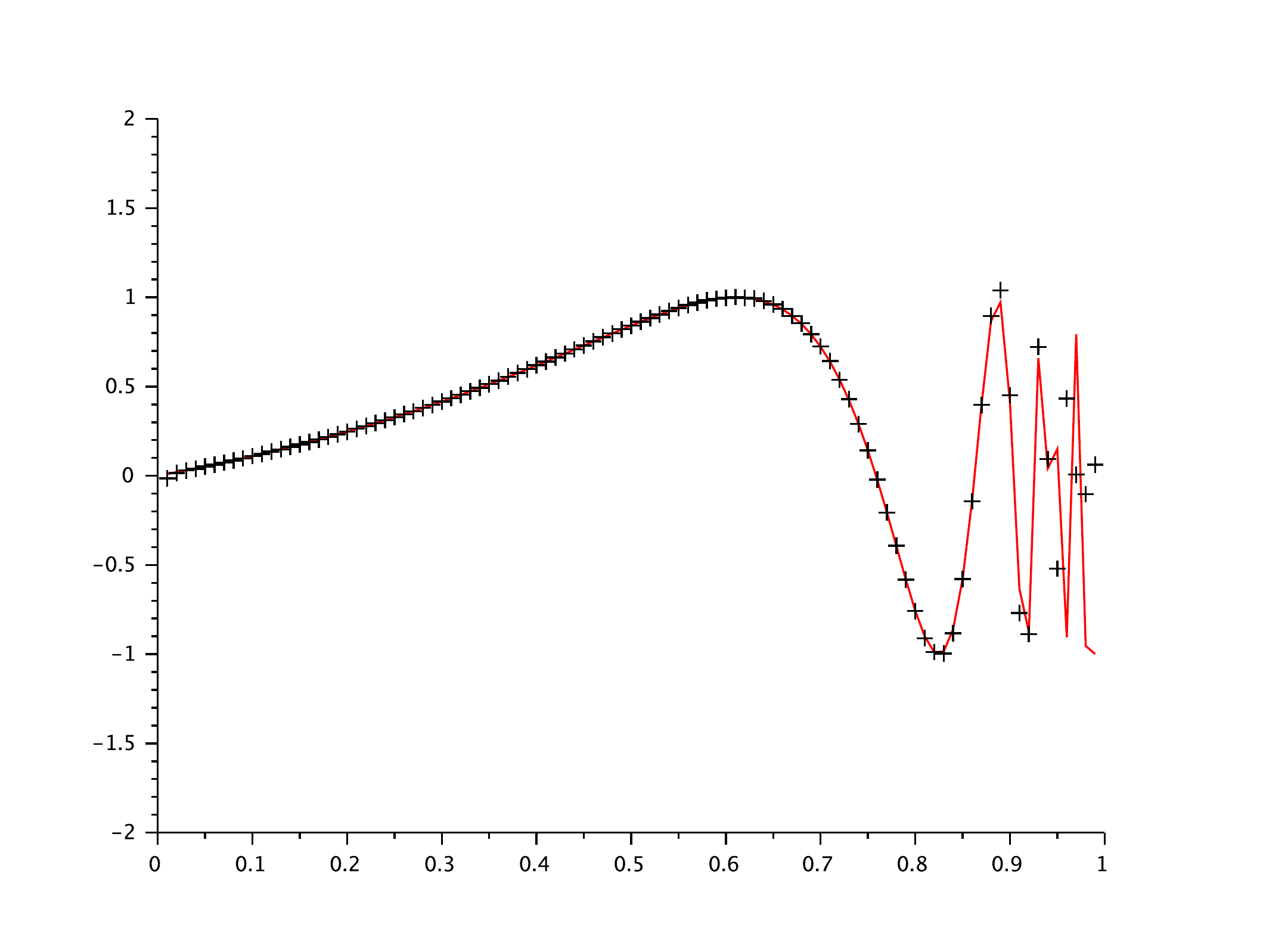}}\hfill
	\caption{Different examples of reconstruction for $CFL=1$ and $s=100$. \label{fig:examples}}
\end{figure}

We recall that in our approach, it is mandatory to know the {\it a priori} bound $m$ such that $Q \in L^\infty_{\leq m}(\mathbb{R})$. On Figure~\ref{fig:m}, we illustrate the behavior of the algorithm in the case where an  error is made on that bound. One can observe that the recovery of the potential is correct only in the zones where the potential $Q$ is effectively bounded by $m$. In this situation, the convergence of the process doesn't occur. In practice, if the retrieved potential meets the value of $m$ in several points, it is recommended to repeat the reconstruction process after choosing a greater value of $m$.
\begin{figure}[!h]
	\centering
	\subfloat[]
		{\includegraphics[width=0.3\textwidth]{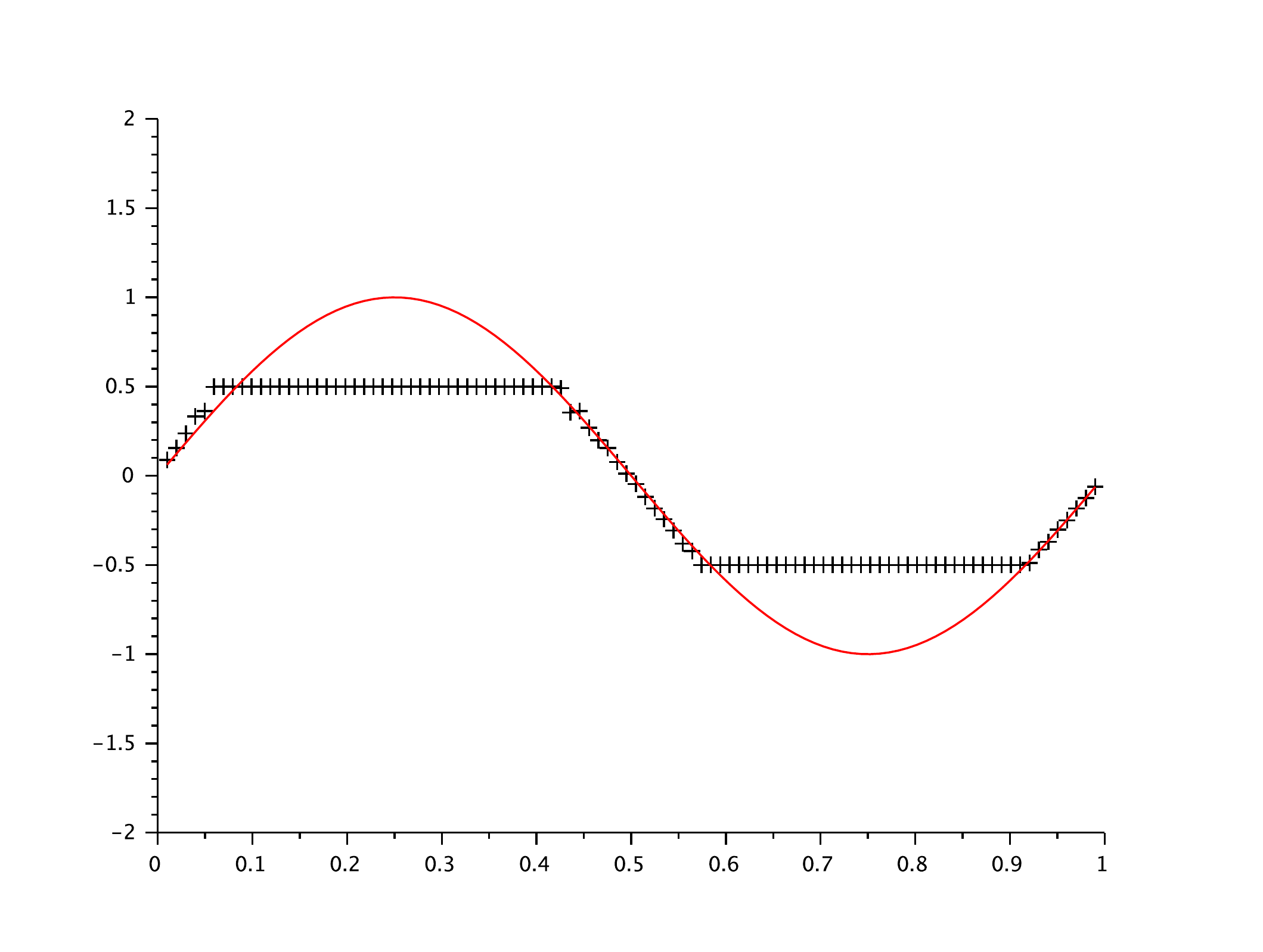}}\hspace{1cm}
	\subfloat[]	
		{\includegraphics[width=0.3\textwidth]{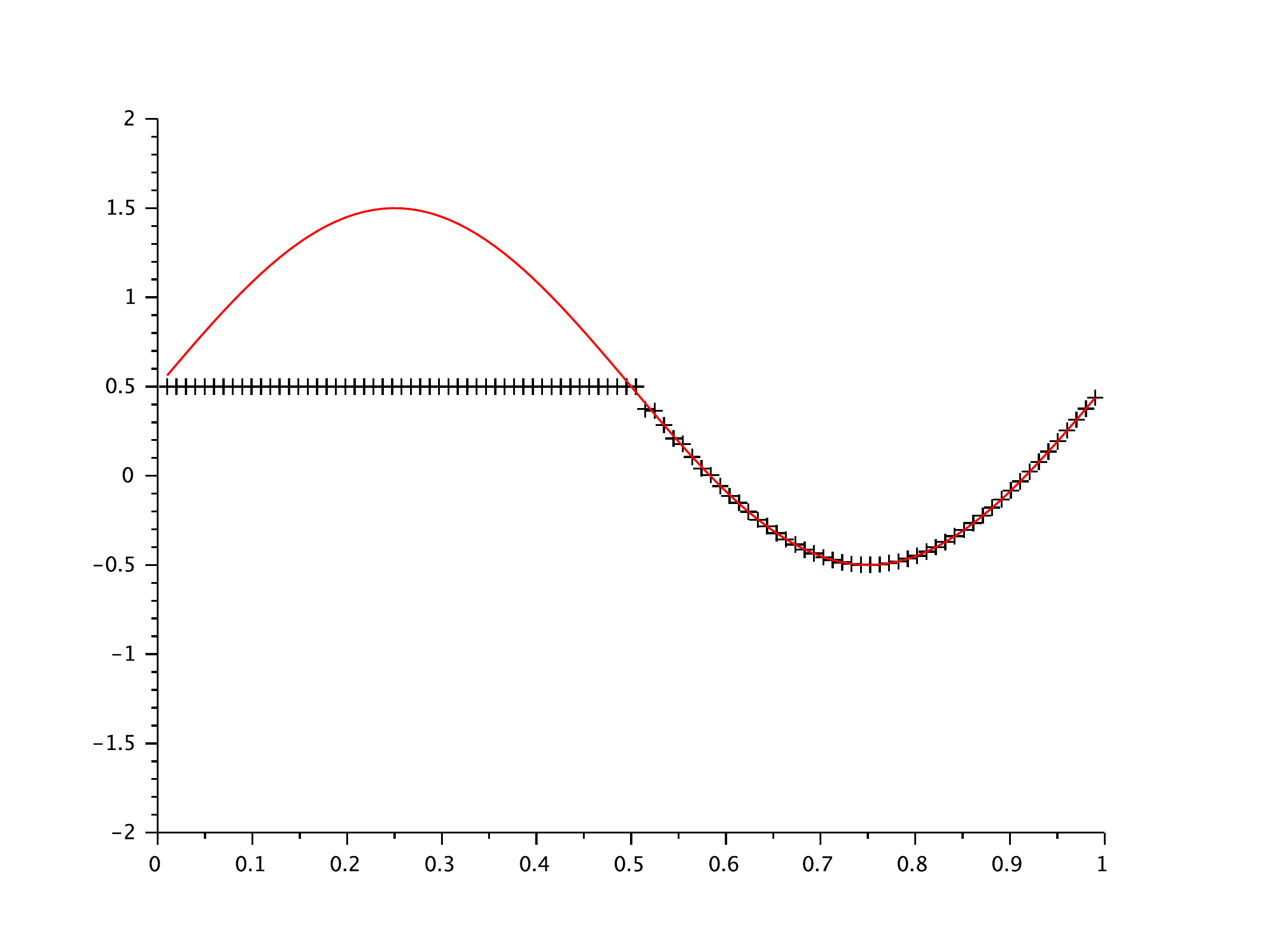}}
	\caption{ Reconstruction of exact potentials with the wrong choice of the {\it a priori} bound $m=0.5$, for $CFL=1$ and $s=100$.\label{fig:m}}
\end{figure}

\FloatBarrier

\subsection{Simulations with several levels of noise}

If we slightly modify the stability condition and take a $CFL$ condition strictly smaller that  $1$, the explicit numerical scheme used to solve \eqref{Eqwk} leads to a non negligible approximation error, acting as a noise. The presence of the additional regularization term \eqref{High-Freq-Carl-Term} in the functional is therefore necessary. In that case, if the mesh size $h$ (through $\tau$) is given, it is not possible to take $s$ as large as desired. Nevertheless, even for smaller values of $s$, Algorithm \ref{AlgoDisc-Bis} gives good results, that can be improved by refining the mesh. In Figure~\ref{fig:examples2}, several results of reconstruction of potentials obtained for $\alpha=0$, $CFL = 0.9$ and $s=10$ are presented.
\begin{figure}[!h]
	\centering
	\subfloat[$Q=\sin(2 \pi x)$]
		{\includegraphics[width=0.3\textwidth]{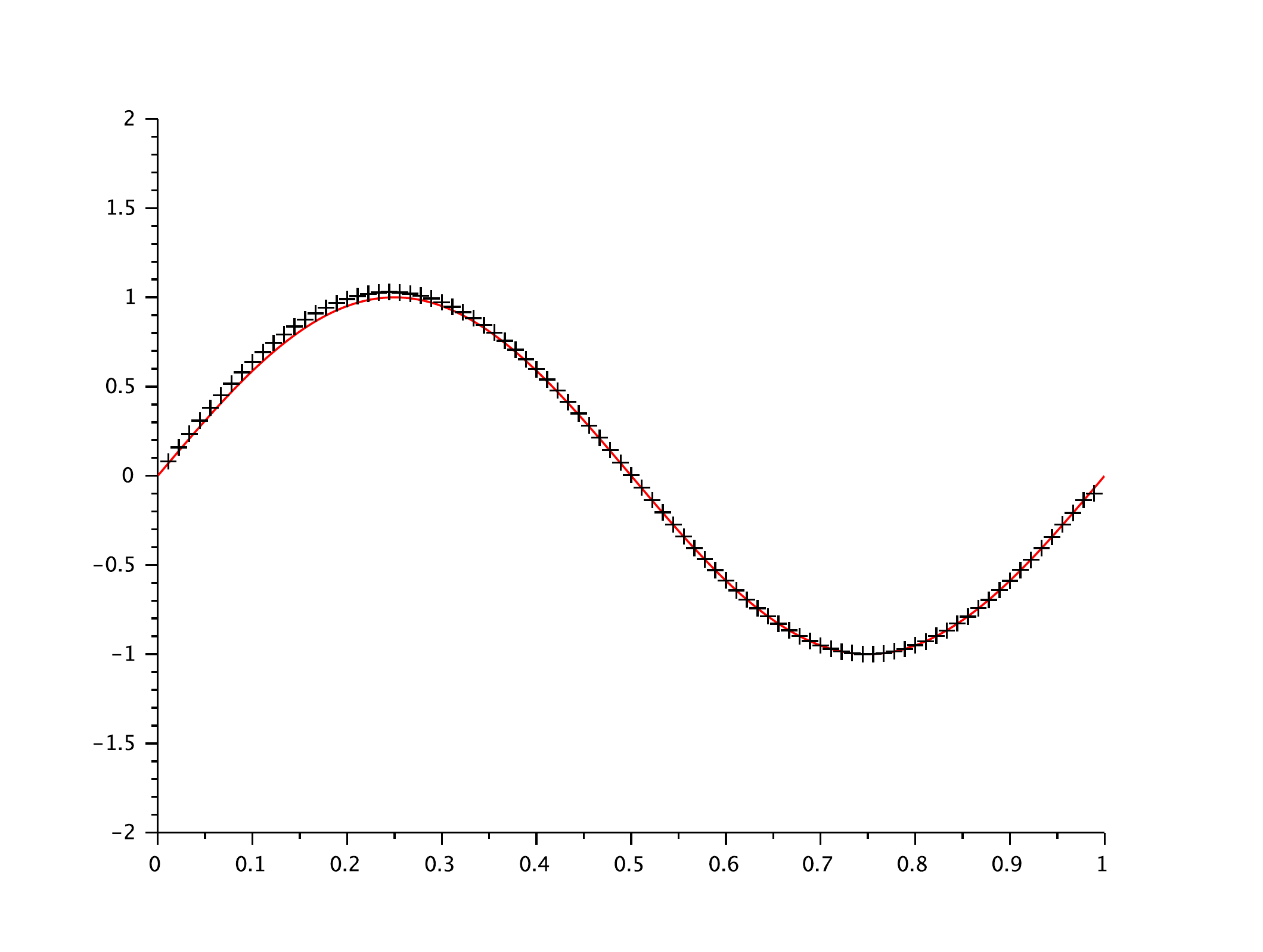}}\hfill
	\subfloat[$Q$ heaviside]
		{\includegraphics[width=0.3\textwidth]{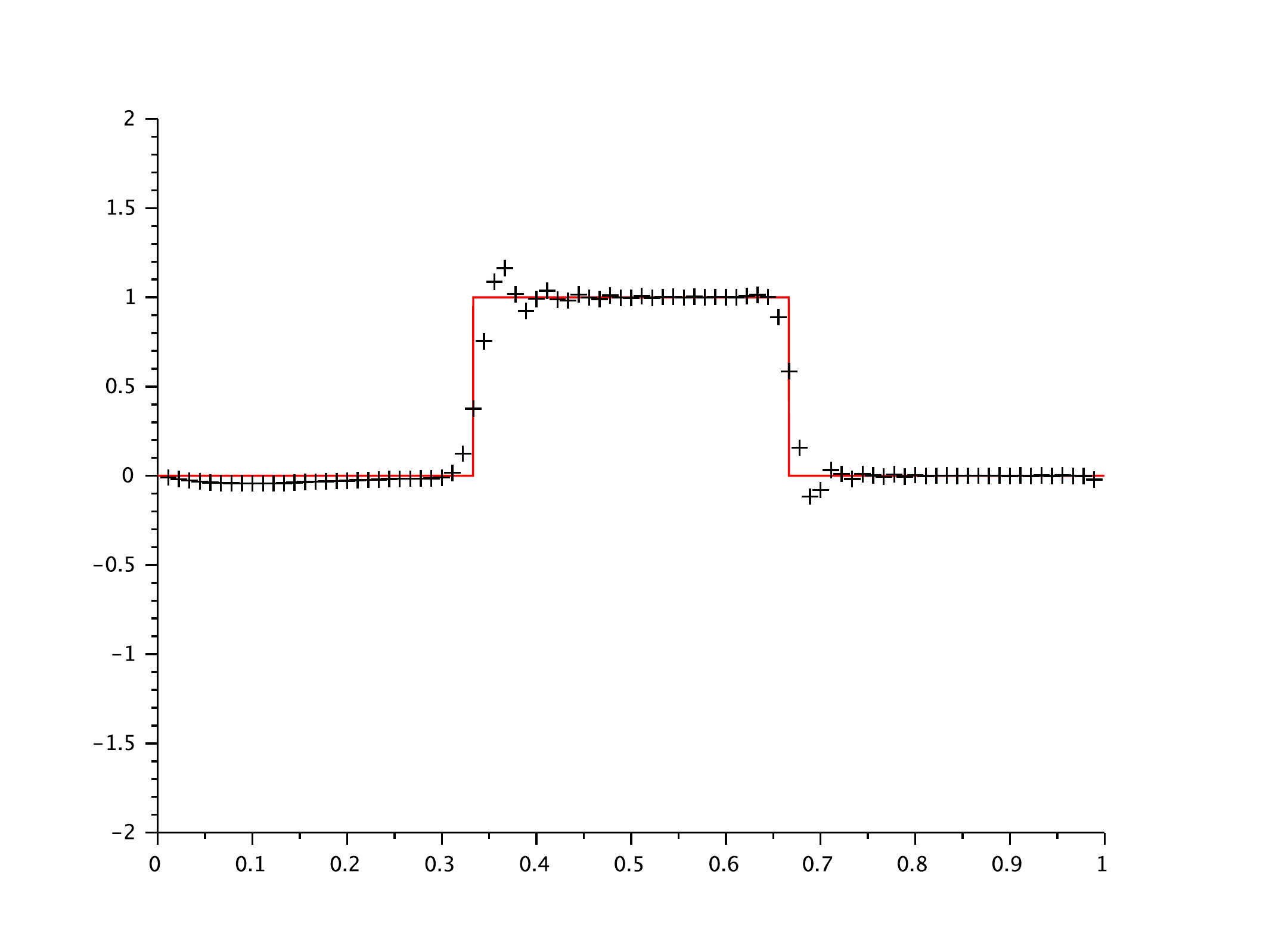}}\hfill
	\subfloat[$Q=\sin(\frac{x}{1-x})$]	
		{\includegraphics[width=0.3\textwidth]{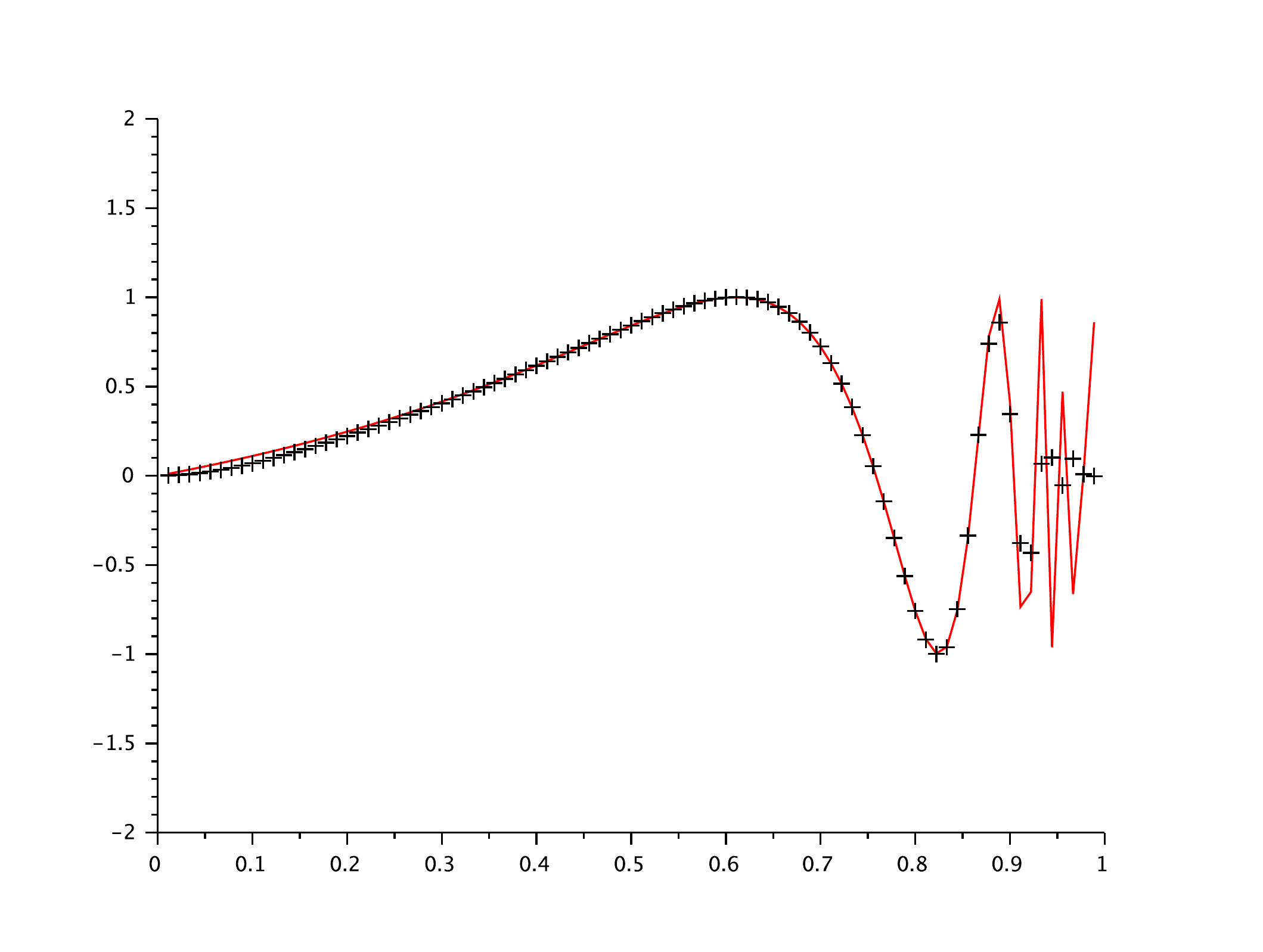}}\hfill
	\caption{Different examples of reconstruction for $CFL=0.9$ and $s=10$. \label{fig:examples2}}
\end{figure}
Figure~\ref{fig:noise} shows the results for $Q(x)=\sin(\pi x)$ with different level of noise in the measurements ($\alpha = 1\%$, $5\%$ and $10\%$). Here, we used the appropriate discretized functional constructed to deal with the discretization process.
\begin{figure}[!h]
	\centering
	\subfloat[$\alpha = 1\%$]
		{\includegraphics[width=0.3\textwidth]{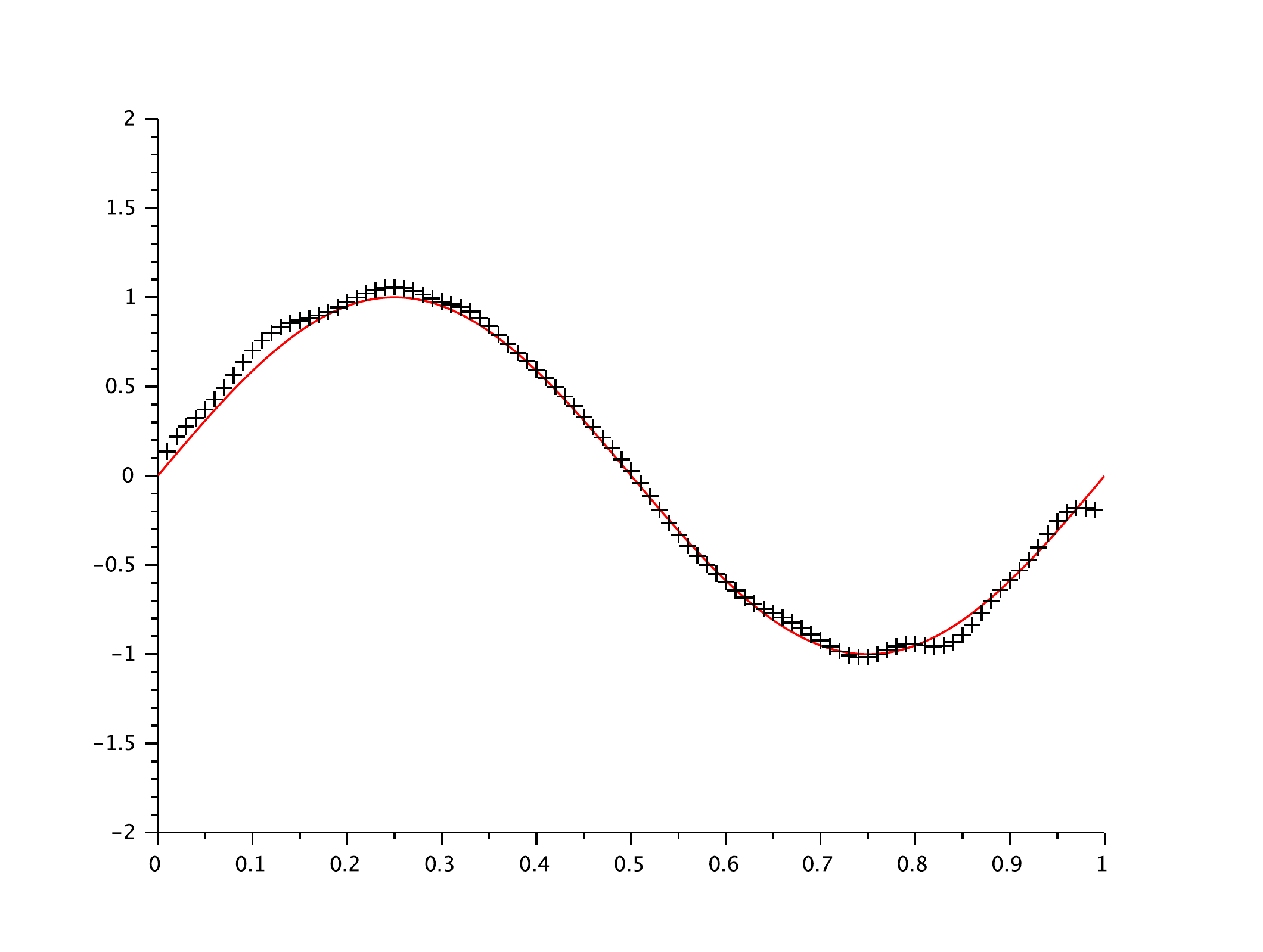}}\hfill
	\subfloat[$\alpha = 5\%$]	
		{\includegraphics[width=0.3\textwidth]{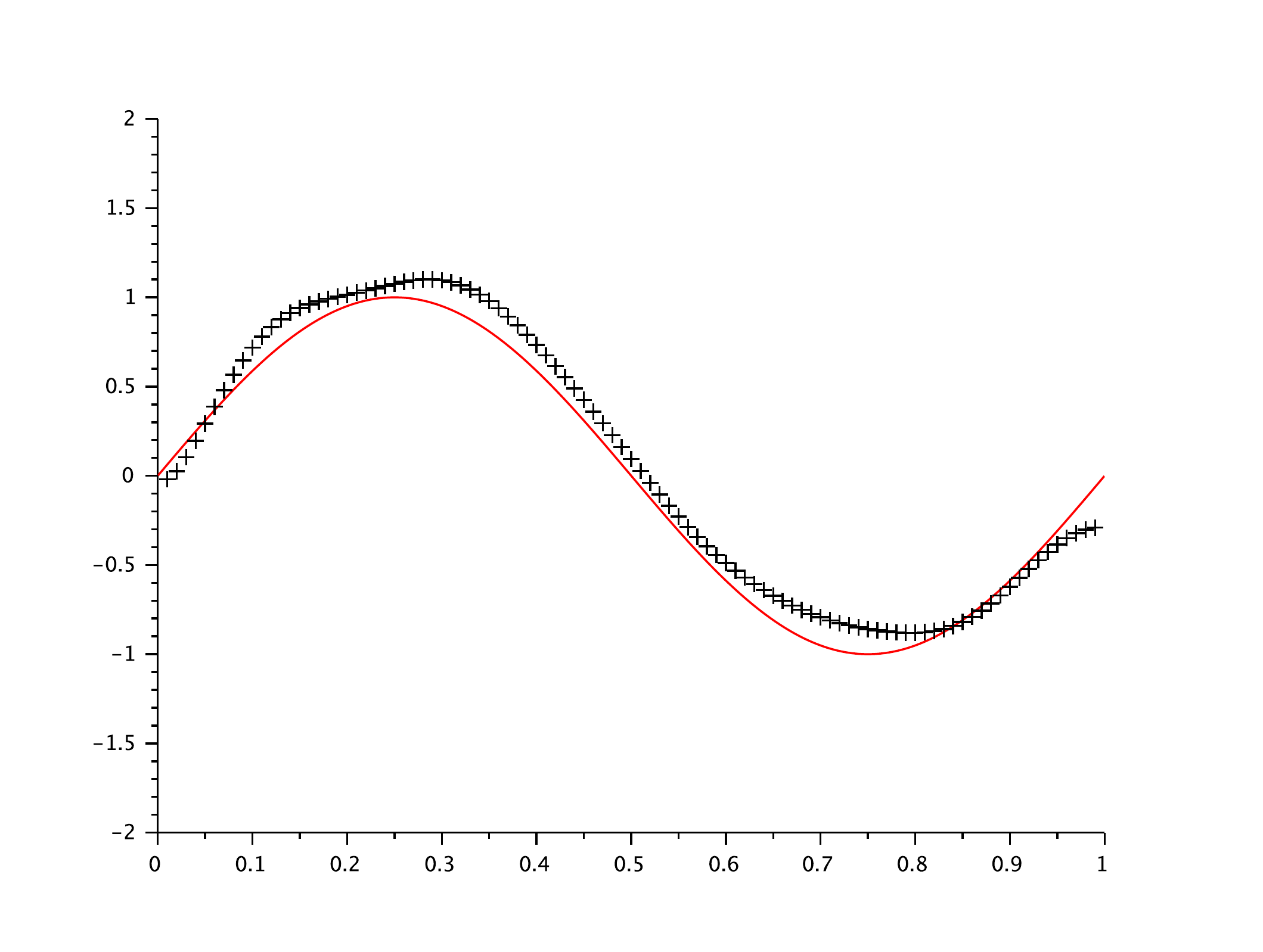}}\hfill
	\subfloat[$\alpha = 10\%$]
		{\includegraphics[width=0.3\textwidth]{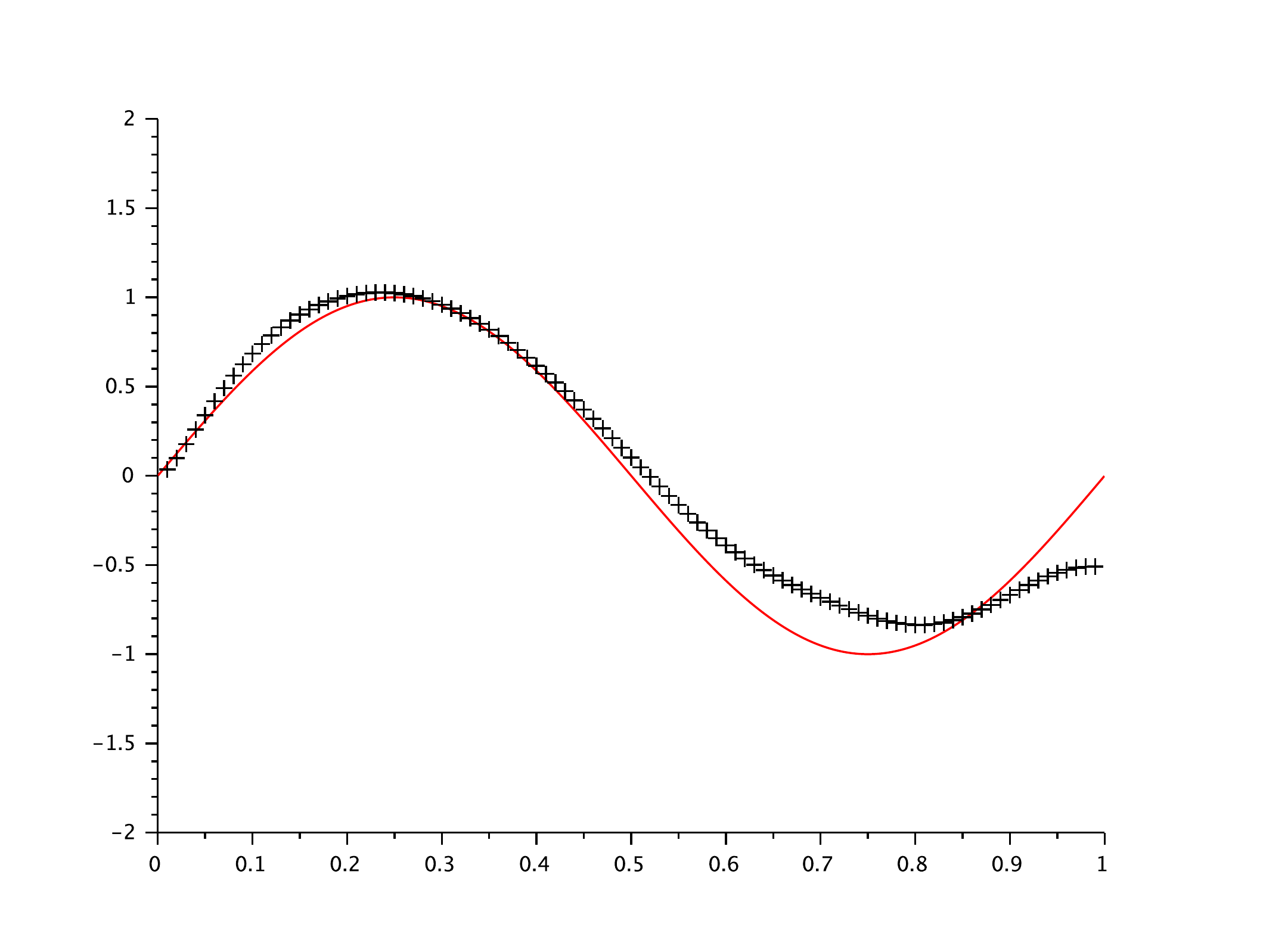}}\hfill
	\caption{Recovery of the potential $Q(x)=\sin(\pi x)$ in presence of noise in the data. The level of noise is denoted by $\alpha$. Here, $CFL=0.9$ and $s=10$. \label{fig:noise}}
\end{figure}

Eventually, Figure~\ref{fig:instabilités} shows on the left hand side, an example of result obtained when the functional is discretized without taking into account the additional terms \eqref{High-Freq-Carl-Term} requisite for its uniform coercivity with respect to the mesh size. Since the first iteration, severe oscillations occur and they amplify with the iterative process. On the right hand side, we illustrate the necessity of choosing a discretization space step small enough with respect to the value of the parameter $s$. Indeed, if the mesh size is too coarse, numerical instabilities appear.
\begin{figure}[!h]
	\centering
	\subfloat[Result for $CFL=0.9$, $s=4$ and no regularization term.]
		{\includegraphics[width=0.3\textwidth]{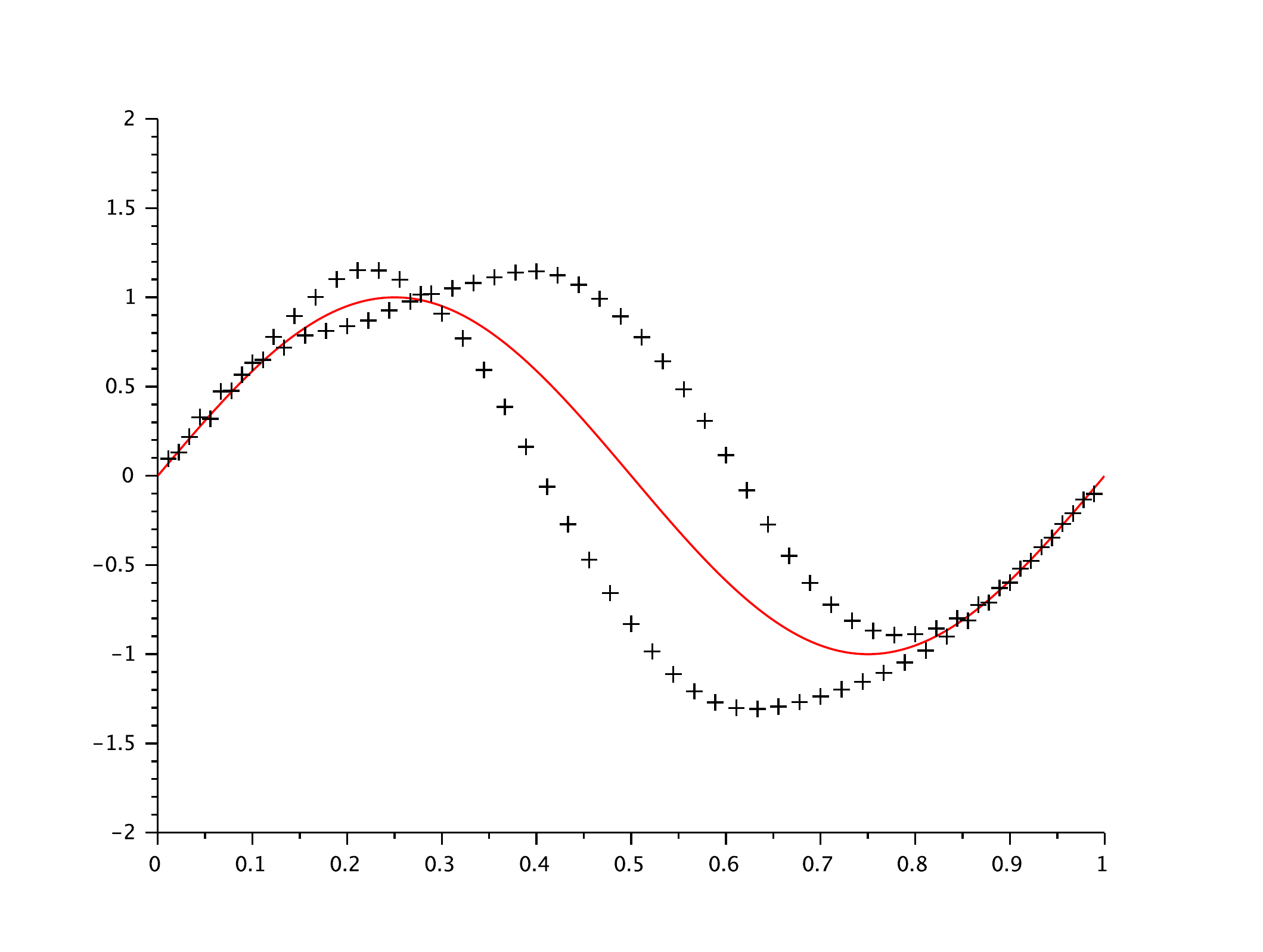}}\hspace{1cm}
	\subfloat[Result for $h=0.011$ and $s=17$, that is $sh = 0.187$.]	
		{\includegraphics[width=0.3\textwidth]{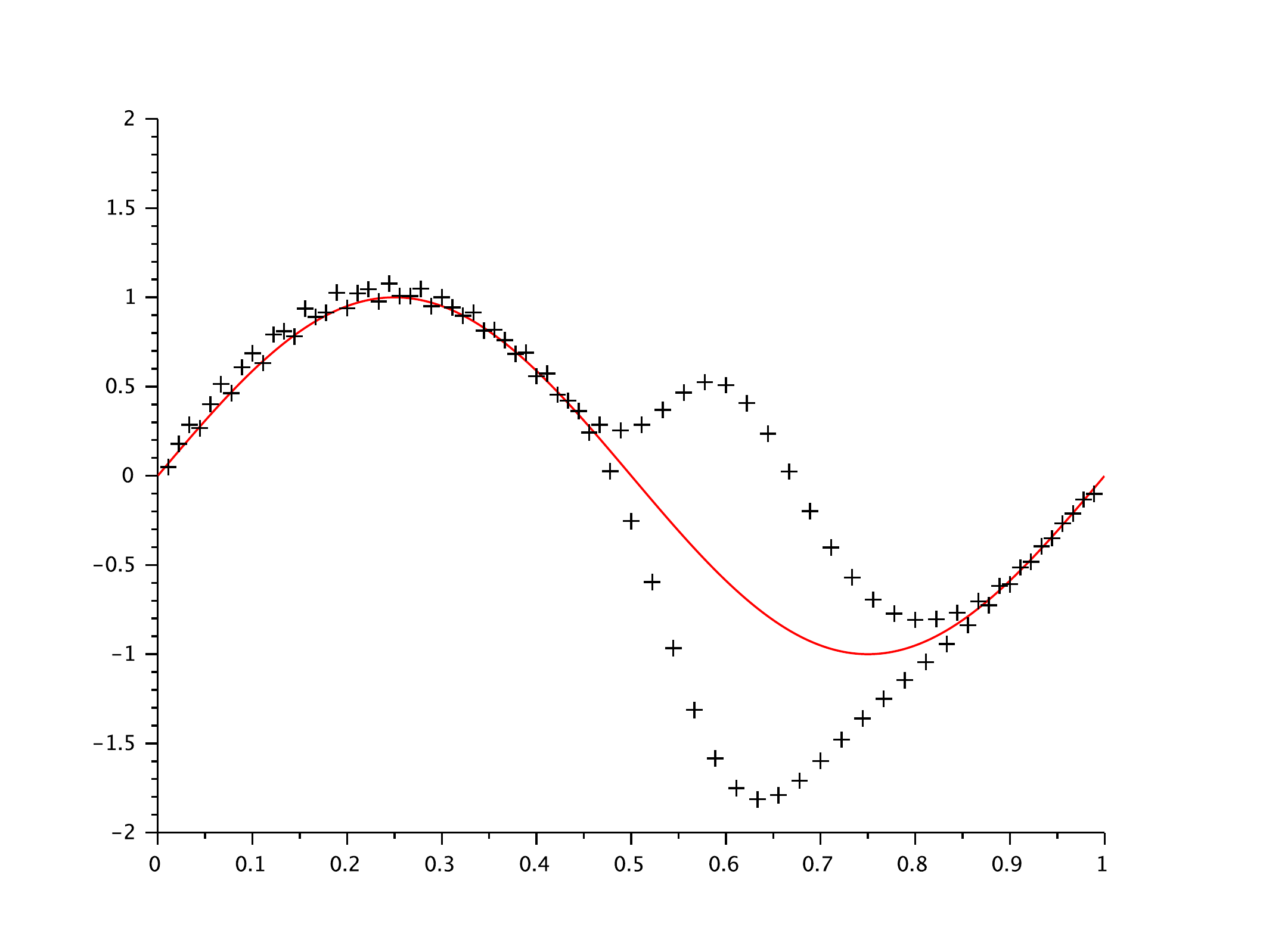}}
	\caption{Illustration of the need of the additional regularization term \eqref{High-Freq-Carl-Term} in the functional (left). Illustration of the needed condition \eqref{Why-sh-small} between $s$ and the space step $h$ (right). \label{fig:instabilités}}
\end{figure}


\subsection{Simulations for initial datum not satisfying (\ref{Positivity})}
%

So far, we presented numerical simulations in which the positivity assumption \eqref{Positivity} on $w_0$ was satisfied. In this section, we would like to briefly present what can be done in the case in which it is not satisfied. In that case, Step 3 of Algorithm \ref{AlgoDisc-Bis} can be replaced by :
\begin{equation}
     \tilde q_h^{k+1}(x_j)  =  
     \left\{
     \begin{array}{cc}
     q_h^k(x_j) + \dfrac{\partial_t \widetilde Z^k_h(0,x_j)}{w_{0}(x_j)}, &\text{ for } j \in \{1, \cdots,N\} \text{ such that } |w_0(x_j)| \geq \alpha,\\
     0, &\text{elsewhere}, \\
     \end{array}
     \right.
\end{equation}
where $\alpha >0$ is the constant appearing in \eqref{Positivity}.
As an example, let us consider 
$$
	w_0(x) = -a + x, \quad a \in (0,L),
$$
which cancels at $x = a$ in a single isolated point. If we take $\alpha = 10^{-2}$, we obtain the results given in Figure~\ref{fig:wrongw0}. Actually, the reconstruction is satisfactory outside a small neighborhood around $x=a$.

\begin{figure}[!h]
	\centering
	\subfloat[$Q=\sin(2 \pi x)$ and $a=0.5$.]
		{\includegraphics[width=0.3\textwidth]{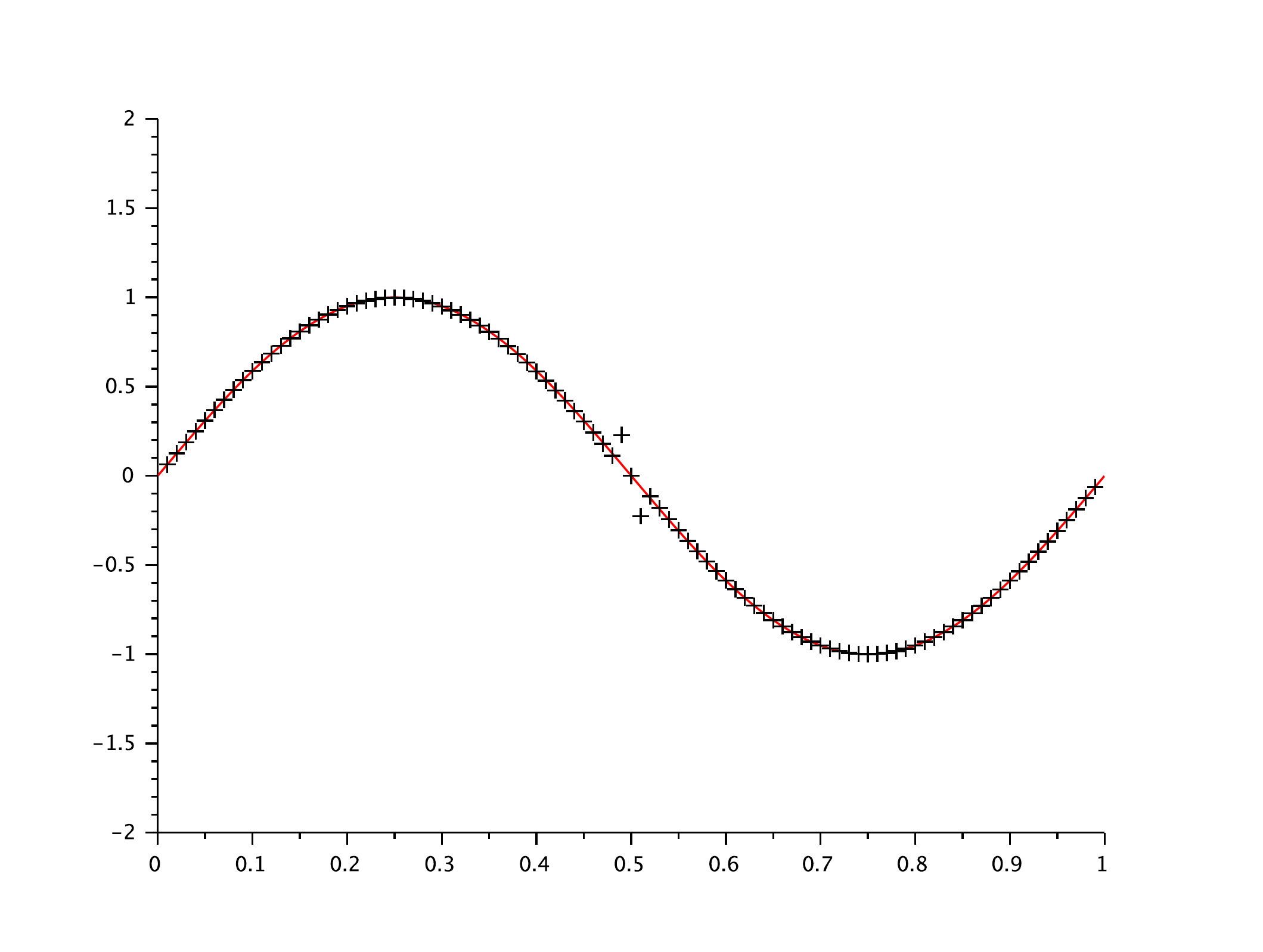}}\quad
	\subfloat[$Q=\sin(2 \pi x)$ and $a=0.2$.]
		{\includegraphics[width=0.3\textwidth]{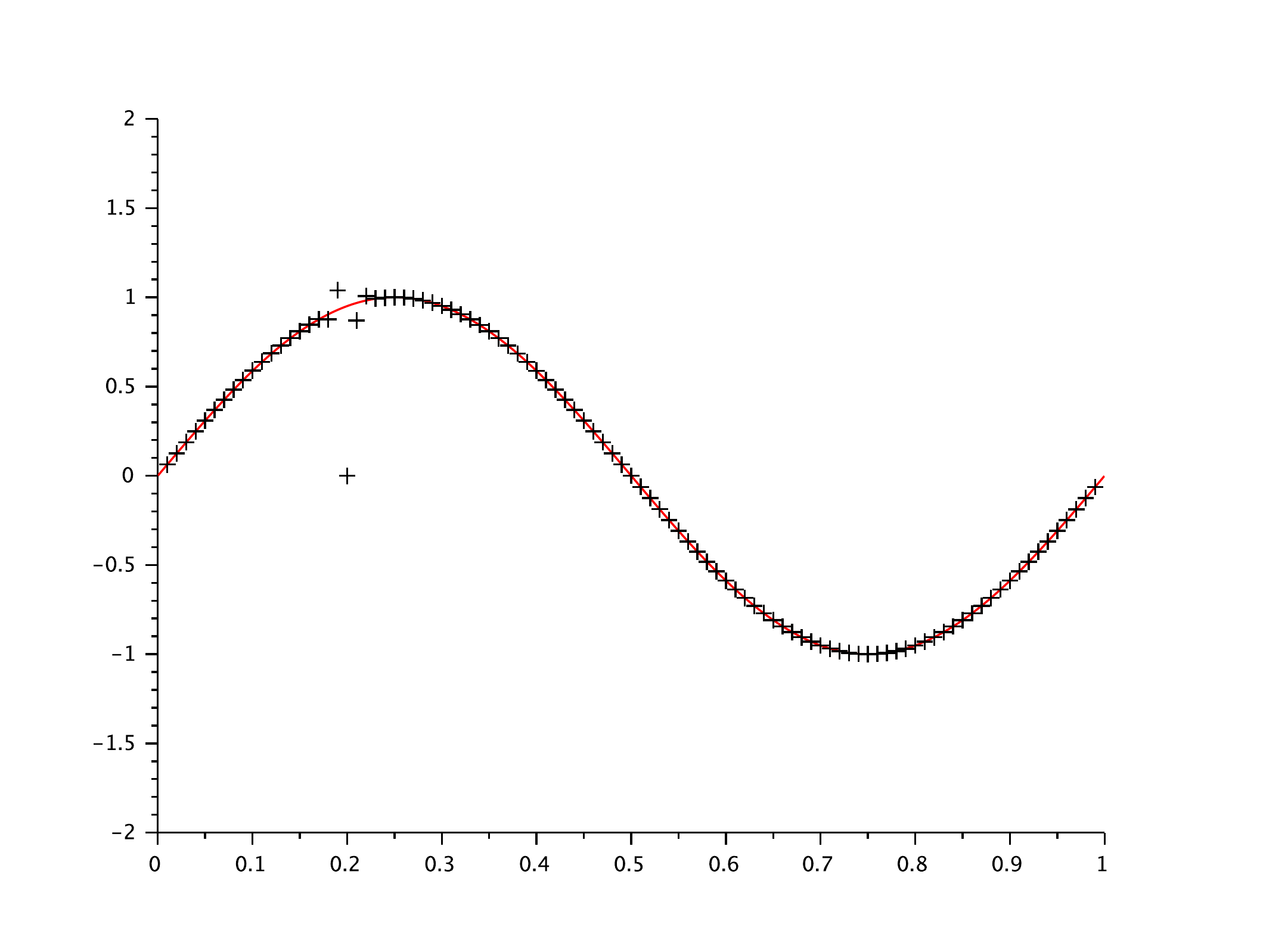}}\quad
	\subfloat[$Q$ heaviside and $a=0.5$.]	
		{\includegraphics[width=0.3\textwidth]{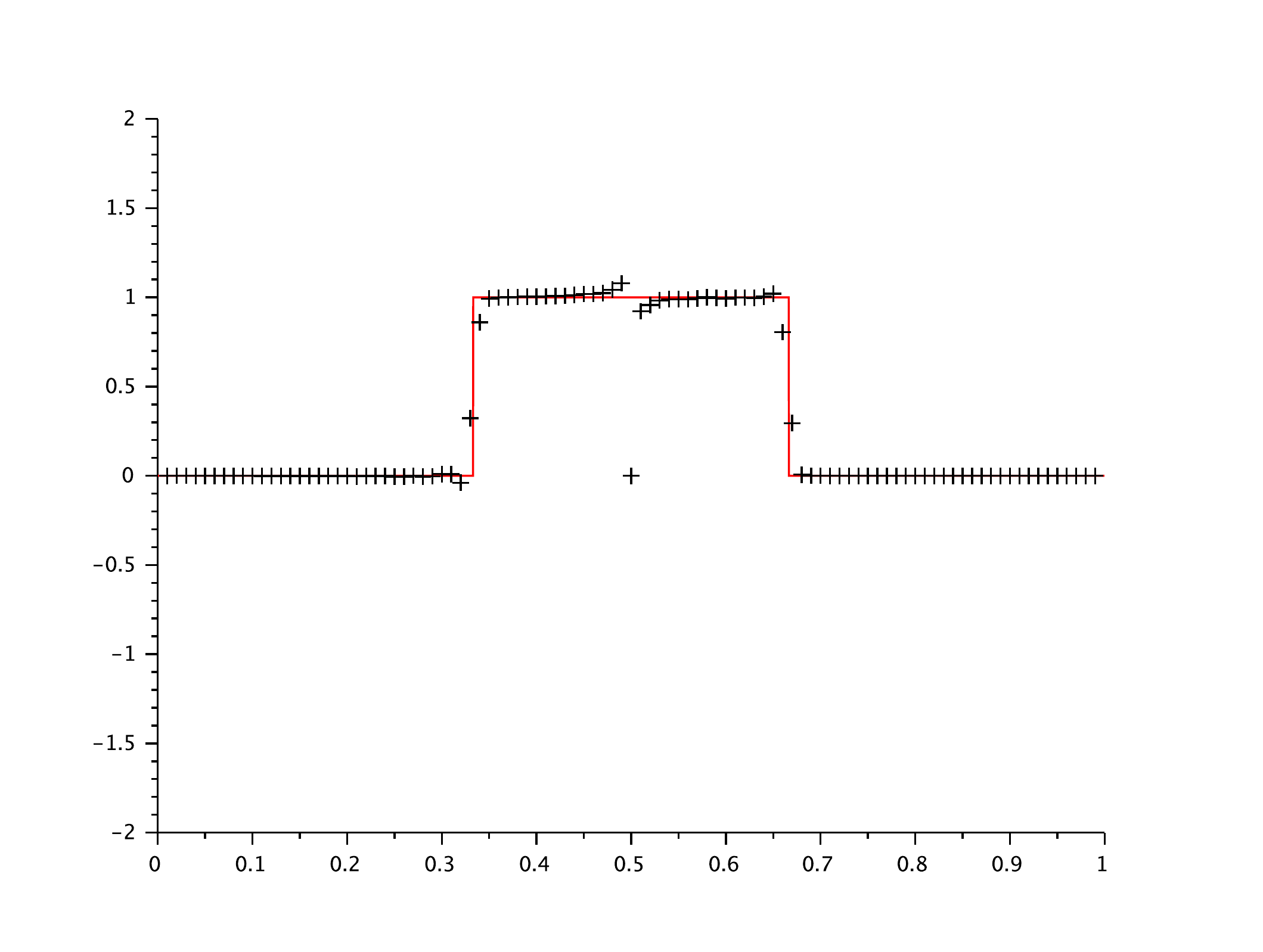}}
	\caption{Reconstructions for $w_0(x) = -a + x$ not satisfying (\ref{Positivity}), $CFL=1$ and $s=100$. \label{fig:wrongw0}}
\end{figure}

Note that here, we made the choice to set $0$ for the potential in the set $\{ x \in (0,L), \, |w_0(x)| \leq \alpha\}$. Of course, other choices are possible. Among them, one could for instance simply do a linear interpolation between the values at the boundary of the set $\{x \in (0,L), \, |w_0(x)| > \alpha\}$. Though, as illustrated in Figure \ref{fig:wrongw0}, it seems that Algorithm \ref{AlgoDisc-Bis} converges anyway in the set $\{x \in (0,L), \, |w_0(x)| > \alpha\}$. One can therefore perform any kind of interpolation process to complete the values of the potentials in the set $\{x \in (0,L), \, |w_0(x)| > \alpha\}$ after the convergence has been achieved.

\subsection{Simulations in two dimensions}

We also performed some reconstructions in two dimensions where $\Omega = [0,1]^2$, $x_0 = (-0.3,-0.3)$, $\Gamma_0=\{x=1\}\cup \{y=1\}$, $w_0(x_1,x_2)=2+\sin(\pi x_1)\sin(\pi x_2)$, $w_1 = 0$, $f = 0$, $f_\partial = 2$, $\beta=0.99$, $m=2$ and $CFL = 0.5 \leq \frac{\sqrt{2}}{2}$. Figure~\ref{fig:2D} presents the results obtained for three different potentials. We took $s=3$ and could not take it larger. Indeed, decreasing the space step $h$ to ensure that $sh$ remains small (condition \eqref{Why-sh-small}) leads to large systems (\ref{VarForm}) that exhaust the computational memory of \textsc{Scilab} pretty fast. The preliminary results of Figure~\ref{fig:2D} are obtained in an ideal framework where both direct and inverse problems are solved with the same numerical scheme on the same mesh and there is no noise. All theses simplifications will be removed in a forthcoming work where we wish to develop a convergent algorithm to reconstruct a non homogeneous wave speed from the information given by the flux $\mathcal{M}$.

\begin{figure}[!h]
	\centering
		\includegraphics[width=0.2\textwidth]{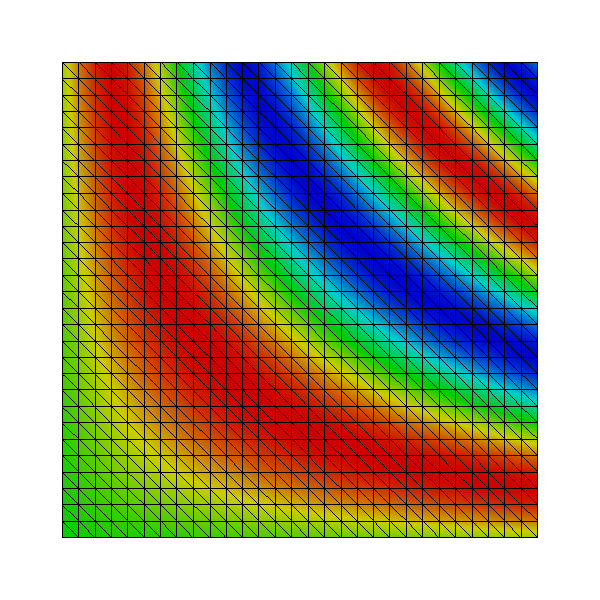}\quad
\includegraphics[width=0.2\textwidth]{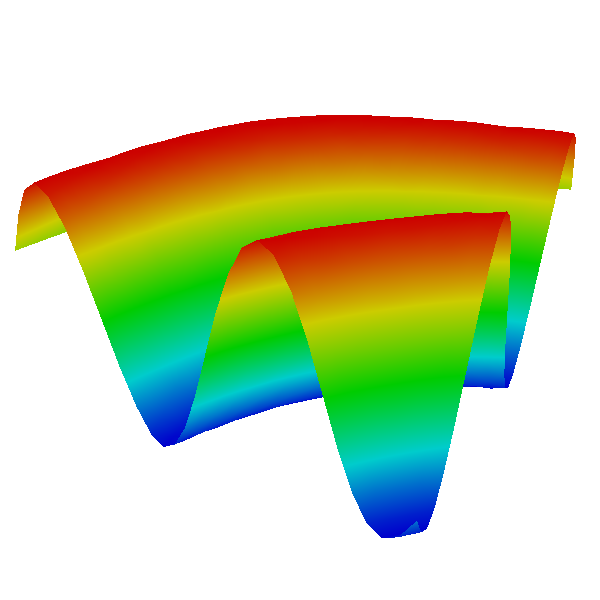}\qquad\qquad
		\includegraphics[width=0.2\textwidth]{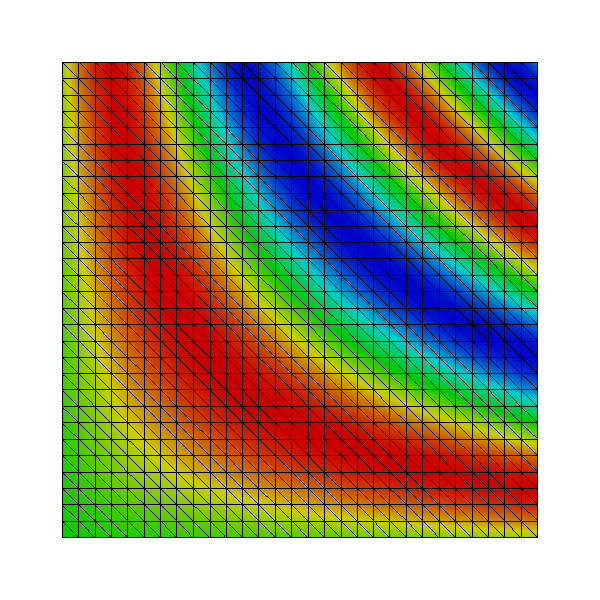}\quad
		\includegraphics[width=0.2\textwidth]{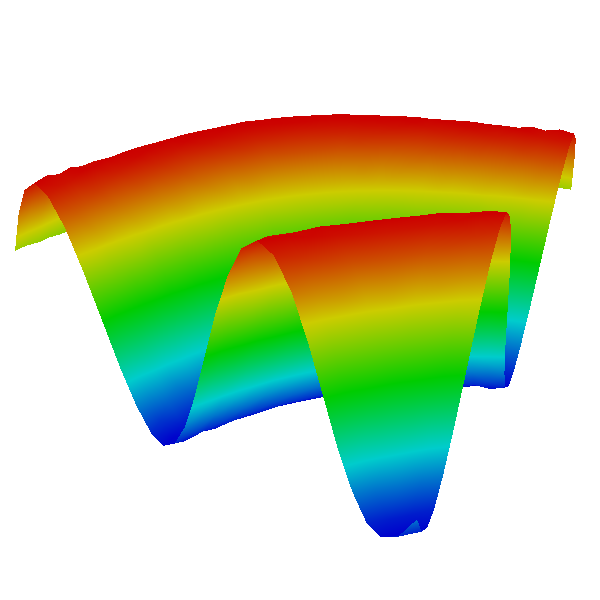}	\\
		\includegraphics[width=0.2\textwidth]{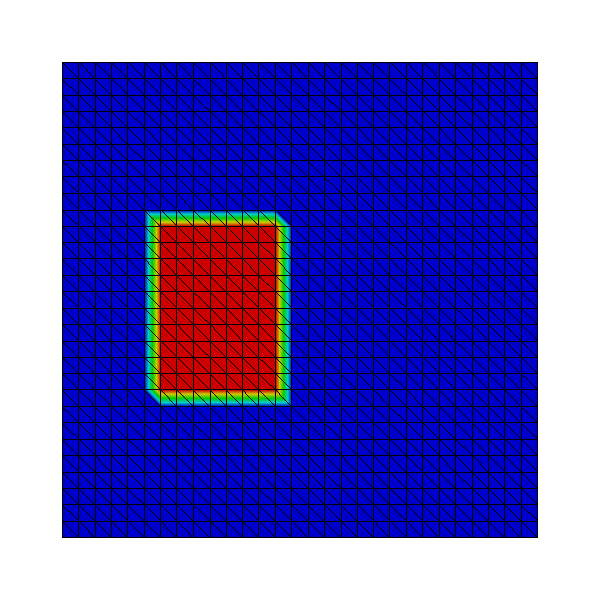}\quad
\includegraphics[width=0.2\textwidth]{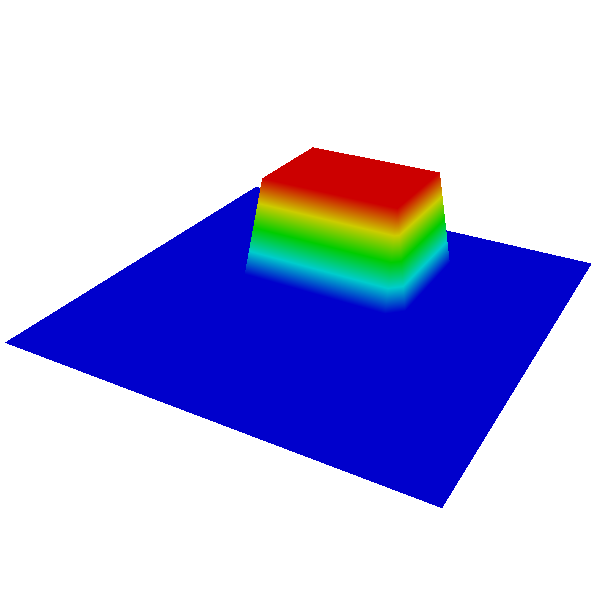}\qquad\qquad
		\includegraphics[width=0.2\textwidth]{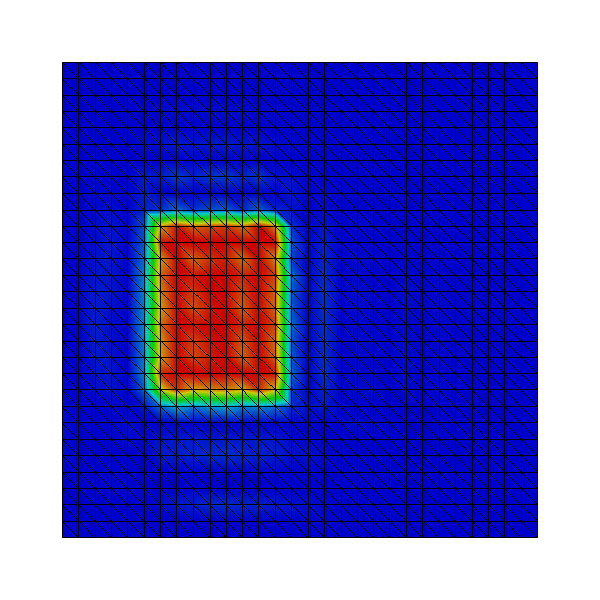}\quad
\includegraphics[width=0.2\textwidth]{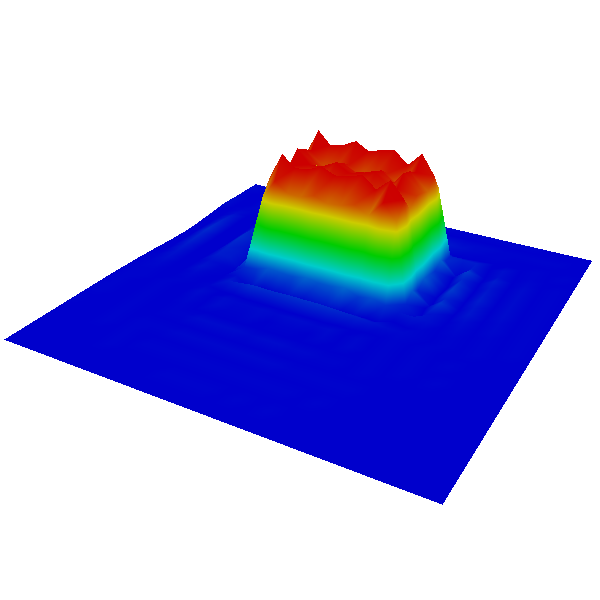}\\	
	\subfloat[Exact potentials.]
		{\includegraphics[width=0.2\textwidth]{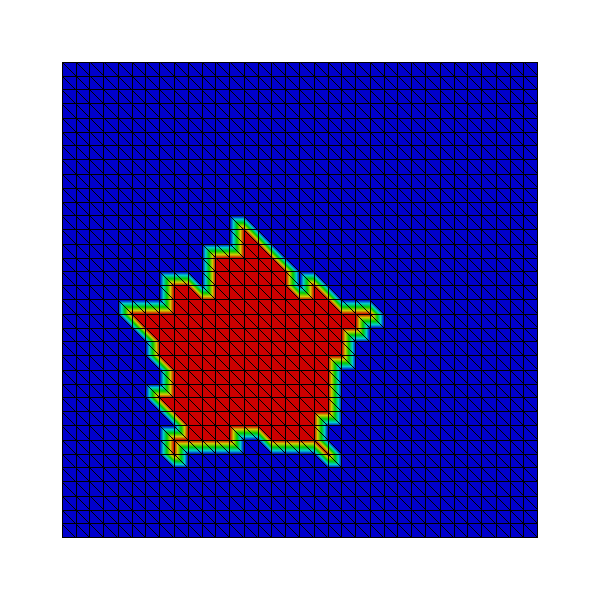}\quad
		\includegraphics[width=0.2\textwidth]{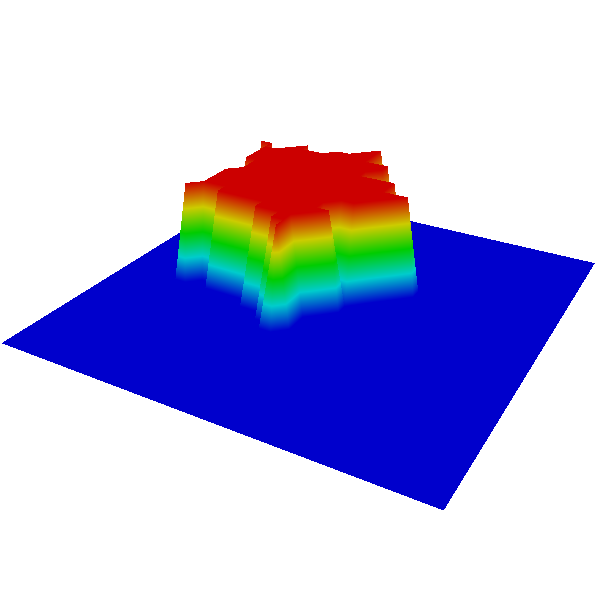}}\qquad\qquad
	\subfloat[Potentials recovered numerically.]
		{\includegraphics[width=0.2\textwidth]{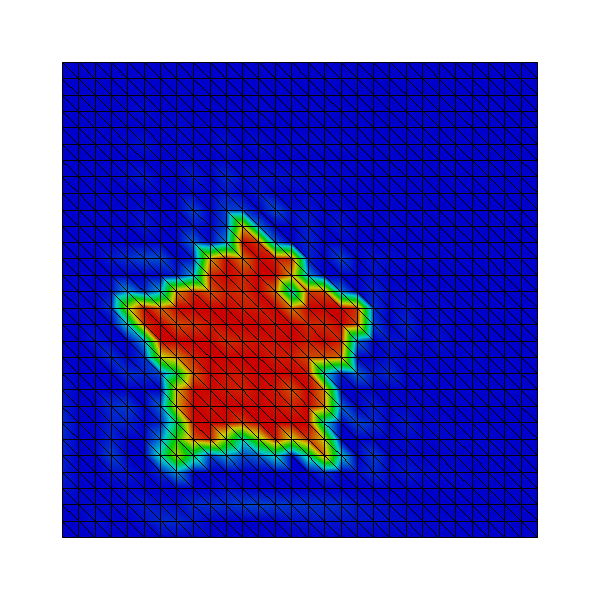}\quad
		\includegraphics[width=0.2\textwidth]{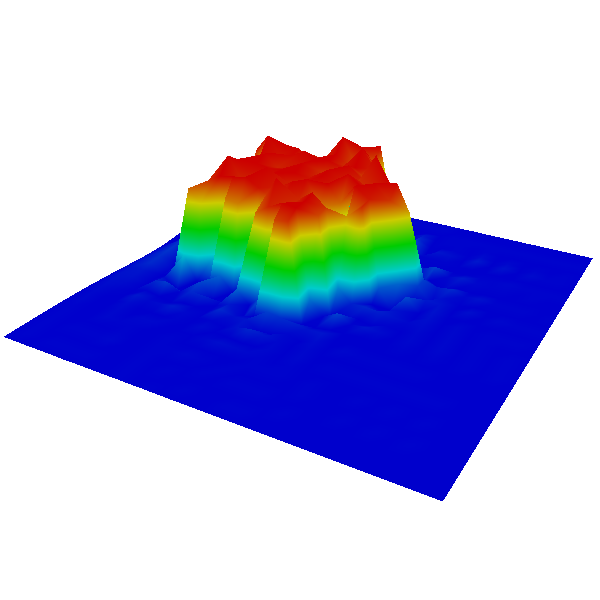}}	
	\caption{Different examples of reconstruction in the 2d case. \label{fig:2D}}
\end{figure}

\bibliographystyle{siamplain}

\providecommand{\MR}[1]{}

\end{document}